\documentclass{amsart}
\usepackage{verbatim,amsmath,amssymb,amscd,color,graphics}
\usepackage[enableskew]{youngtab}
\usepackage[all]{xy}
\usepackage[dvips]{graphicx}
\usepackage{tikz}
\usetikzlibrary{arrows,matrix}
\tikzset{tab/.style={matrix of math nodes,column sep=-.4, row sep=-.4,text height=9pt,text width=8pt,align=center,inner sep=1.5}}
 
\usepackage[colorlinks=true, pdfstartview=FitV, linkcolor=blue, citecolor=blue, urlcolor=blue]{hyperref}

\usepackage{listings}
\lstdefinelanguage{Sage}[]{Python}
{morekeywords={False,sage,True},sensitive=true}
\definecolor{dblackcolor}{rgb}{0.0,0.0,0.0}
\definecolor{dbluecolor}{rgb}{0.01,0.02,0.7}
\definecolor{dgreencolor}{rgb}{0.2,0.4,0.0}
\definecolor{dgraycolor}{rgb}{0.30,0.3,0.30}

\usepackage{xparse}

\makeatletter
\protected\def\specialmergetwolists{%
  \begingroup
  \@ifstar{\def\cnta{1}\@specialmergetwolists}
    {\def\cnta{0}\@specialmergetwolists}%
}
\def\@specialmergetwolists#1#2#3#4{%
  \def\tempa##1##2{%
    \edef##2{%
      \ifnum\cnta=\@ne\else\expandafter\@firstoftwo\fi
      \unexpanded\expandafter{##1}%
    }%
  }%
  \tempa{#2}\tempb\tempa{#3}\tempa
  \def\cnta{0}\def#4{}%
  \foreach \x in \tempb{%
    \xdef\cnta{\the\numexpr\cnta+1}%
    \gdef\cntb{0}%
    \foreach \y in \tempa{%
      \xdef\cntb{\the\numexpr\cntb+1}%
      \ifnum\cntb=\cnta\relax
        \xdef#4{#4\ifx#4\empty\else,\fi\x#1\y}%
        \breakforeach
      \fi
    }%
  }%
  \endgroup
}
\makeatother

\usepackage{xparse}
\DeclareDocumentCommand\rpp{ m m g }{
	\foreach \x [count=\s from 1] in {#1}{
	        {\ifnum\s=1
	                \draw (0,-\s)--(\x,-\s);
	                \fi}
	   \draw (0,-\s-1) to (\x,-\s-1);
	   \foreach \y in {0, ..., \x} {\draw (\y,-\s)--(\y,-\s-1);}
	}
	\specialmergetwolists{/}{#1}{#2}\ziplist
	\foreach \x/\y [count=\yi from 1] in \ziplist{
	    \node[anchor=west] at (\x,-\yi - .5) {$\y$};  
	}
	\IfValueT {#3}
	{\foreach \z [count=\zi from 1] in {#3} {\node[anchor=east] at (0,-\zi - .5) {$\z$};}}  
	{}
}

\numberwithin{equation}{section}
 
\newtheorem{theorem}{Theorem}
\newtheorem{proposition}[theorem]{Proposition}
\newtheorem{lemma}[theorem]{Lemma}
\newtheorem{corollary}[theorem]{Corollary}

\theoremstyle{definition}

\newtheorem{definition}[theorem]{Definition}
\newtheorem{remark}[theorem]{Remark}
\newtheorem{example}[theorem]{Example}
 
\numberwithin{theorem}{section}

\newcommand{\gotolowest}{\xrightarrow{\rm lowest}}
\newcommand{\p}{\mathcal{P}}
\newcommand{\rc}{\mathcal{RC}} 
\newcommand{\RC}{\operatorname{RC}} 
\newcommand{\lh}{\operatorname{lh}}
\newcommand{\lb}{\operatorname{lb}}
\newcommand{\ls}{\operatorname{ls}}
\newcommand{\rh}{\operatorname{rh}}
\newcommand{\rb}{\operatorname{rb}}
\newcommand{\rs}{\operatorname{rs}}
\newcommand{\emb}{\operatorname{emb}}
\newcommand{\wt}{\operatorname{wt}}
\newcommand{\la}{\lambda}
\newcommand{\ol}{\overline}
\newcommand{\ot}{\otimes}
\newcommand{\high}{\mathrm{high}}
\newcommand{\lusz}{\star} 

\newcommand{\iso}{\cong} 

\newcommand{\ku}{{}}
\newcommand{\mone}{\overline{1}}
\newcommand{\mtwo}{\overline{2}}
\newcommand{\mthree}{\overline{3}}
\newcommand{\mfour}{\overline{4}}
\newcommand{\mfive}{\overline{5}}
\newcommand{\msix}{\overline{6}}
\newcommand{\mseven}{\overline{7}}
\newcommand{\meight}{\overline{8}}
\newcommand{\mn}{\overline{n}}
\newcommand{\minusi}{\overline{\imath}}
\newcommand{\tone}{t_1}
\newcommand{\ttwo}{t_2}
\newcommand{\tN}{t_N}

\newcommand{\clfw}{\varpi} 

\definecolor{darkred}{rgb}{0.7,0,0} 
\newcommand{\defn}[1]{{\color{darkred}\emph{#1}}} 

\usepackage[colorinlistoftodos]{todonotes}

\begin{document}
 
\title{Type $D_n^{(1)}$ rigged configuration bijection}

\author[M.~Okado]{Masato Okado}
\address[M. Okado]{Department of Mathematics, Osaka City University, 3-3-138,
Sugimoto, Sumiyoshi-ku, Osaka, 558-8585, Japan}
\email{okado@sci.osaka-cu.ac.jp}

\author[R.~Sakamoto]{Reiho Sakamoto}
\address[R. Sakamoto]{Department of Physics, Tokyo University of Science, Kagurazaka, Shinjukuku, 
Tokyo 162-8601, Japan}
\email{reiho@rs.tus.ac.jp}
 
\author[A.~Schilling]{Anne Schilling}
\address[A. Schilling]{Department of Mathematics, University of California, One Shields
Avenue, Davis, CA 95616-8633, U.S.A.}
\email{anne@math.ucdavis.edu}
\urladdr{http://www.math.ucdavis.edu/\~{}anne}

\author[T.~Scrimshaw]{Travis Scrimshaw}
\address[T. Scrimshaw]{School of Mathematics, University of Minnesota, 206 Church St. SE, Minneapolis, MN 55455}
\email{tscrimsh@math.umn.edu}
\urladdr{http://www.math.umn.edu/~{}tscrimsh/}
 
  
\begin{abstract}
We establish a bijection between the set of rigged configurations and the set of tensor 
products of Kirillov--Reshetikhin crystals of type $D^{(1)}_n$ in full generality.
We prove the invariance of rigged configurations under the action of the combinatorial 
$R$-matrix on tensor products and show that the bijection preserves certain statistics 
(cocharge and energy). As a result, we establish the fermionic formula for type $D_n^{(1)}$.
In addition, we establish that the bijection is a classical crystal isomorphism.
\end{abstract}
 
\maketitle
 
\tableofcontents

\pagebreak

\section{Introduction}
In this paper, we establish a bijection between rigged configurations and paths for type $D_n^{(1)}$
and prove that it can be extended to a classical crystal isomorphism.
Rigged configurations are combinatorial objects which were introduced by Kerov, Kirillov and 
Reshetikhin~\cite{KKR:1986} through their insightful analysis of the Bethe Ansatz 
for quantum integrable systems. Observing that the number of Bethe
vectors is equal to the number of irreducible components of the multiple
tensor product of the vector representation of $\mathfrak{sl}_2$, 
they constructed a bijection from rigged configurations to standard tableaux.
This work was generalized to the symmetric tensor representations of $\mathfrak{sl}_n$ in~\cite{KR86}
and to rectangular shape ones in~\cite{KSS:2002}. In these two works, the charge
statistic is introduced for rigged configurations and it was shown to agree with
Lascoux--Sch\"{u}tzenberger's charge~\cite{LascouxSchuetzenberger.1978} for tableaux.

Paths~\cite{DJKMO.1989} (sometimes called Kyoto paths to avoid confusion with Littelmann's path
model~\cite{L94}) also originated from quantum integrable systems; not from the
Bethe ansatz, but from Baxter's corner transfer matrix~\cite{B89}. Thanks to
Kashiwara's crystal basis theory~\cite{Kashiwara:1991}, the notion of a path was 
reformulated as an element of the tensor product of crystal bases of 
certain finite-dimensional modules of quantum affine algebras, called
Kirillov--Reshetikhin (KR) modules~\cite{KR:1990}, and then related with affine Lie
algebra characters~\cite{KKMMNN91,KKMMNN92}. In this paper, a path is a highest weight element
in the crystal; that is, an element that is annihilated by the Kashiwara
operator $e_i$ for any index $i$ of the Dynkin diagram of the affine Lie algebra
except $0$.

There are two physical methods, the corner transfer matrix method and the Bethe ansatz, which
are used to analyze quantum integrable systems based on KR modules. Combinatorially, they give rise to 
a conjectural equality of generating functions $X=M$ of paths with the energy statistic $X$ and of rigged 
configurations with the charge statistic $M$~\cite{HKOTY99,HKOTT02}. Although the equality $X=M$ is yet 
to be proven bijectively in full generality except for type $A_n^{(1)}$, there is plenty of evidence for its 
correctness (for proofs in special cases, see below).
In fact, it was shown to be true when $q=1$ and the relevant affine algebra is of non-twisted type. 
In this case, $X$ turns out to be a branching number
of KR modules with respect to the quantized enveloping algebra corresponding
to the underlying finite-dimensional simple Lie algebra. The relations between characters
were shown to be $Q$-systems~\cite{Nakajima.2003,Hernandez.2006}.
They are known to imply the fermionic formula $M$ at $q=1$ in a weak sense, meaning that it 
may contain the binomial coefficient $p+m\choose m$ with $p < 0$~\cite{HKOTY99,KNT.2002}.
This last gap was filled eventually in~\cite{kedem.difrancesco.2008}. 
Naoi~\cite{naoi.2012} showed $X=M$ for types $A_n^{(1)}$ and
$D_n^{(1)}$ using fusion products of the current algebra and Demazure operators, but his proof is
not bijective in nature.

One of the aims of the present paper is to prove $X=M$ for type $D_n^{(1)}$ in full generality
by constructing an explicit bijection from paths to rigged configurations.
To explain previous developments of this bijective method, we note that KR
crystals are parameterized by two integers $r,s$, where $r$ refers to a node of the 
Dynkin diagram of the underlying simple Lie algebra, $D_n$ in our case, and $s$ is
any positive integer. Let $B^{r,s}$ denote the corresponding KR crystal. 
The existence of $B^{r,s}$ and its combinatorial structure is known for all
nonexceptional types~\cite{FOS:2009,O07,OS08}. 
Returning to the history of the bijective method, the simplest case when $B=(B^{1,1})^{\ot k}$ 
is treated in~\cite{OSS:2003a} and the case when $\bigotimes_{i=1}^k B^{1,s_i}$ in~\cite{SS:X=M},
not only for type $D_n^{(1)}$ but also for nonexceptional types. In type $D_n^{(1)}$, the bijection for 
$\bigotimes_{i=1}^k B^{r_i,1}$ was constructed in~\cite{S:2005} and for the 
single KR crystal $B=B^{r,s}$ in~\cite{OSS:2013}. In fact, we rely on these two papers
for many properties of combinatorial procedures used in this paper.

In~\cite{Sch:2006}, a crystal structure on rigged configurations of simply-laced types was defined.
This led to the generalization of rigged configurations to unrestricted rigged configurations,
which form the completion under the crystal operators. Unrestricted rigged configurations in type $A_n^{(1)}$
were characterized in~\cite{DS06}. This generalization turns out to be extremely
powerful. In particular, the bijection $\Phi$ from paths to rigged configurations can be shown
to extend to a crystal isomorphism. 

Let us briefly explain how our bijection $\Phi$ from tensor products of KR crystals to unrestricted
rigged configurations is constructed.
We consider the general case $B=B^{r_k,s_k}\ot\cdots\ot B^{r_2,s_1}\ot B^{r_1,s_1}$. 
As is revealed in~\cite{OSS:2013}, we regard an element of the single KR crystal $B^{r,s}$ as a tableau, 
called a Kirillov--Reshetikhin tableau, of $r\times s$ rectangular shape. For $r_k\le n-2$, we then define 
three procedures $\ls$, $\lb$, $\lh$ on an element of $B$. The left-split $\ls$ splits off the leftmost 
column of the leftmost KR crystal $B^{r_k,s_k}$. The left-box $\lb$ splits off the lowest box
of the leftmost column when the leftmost KR crystal is a column, that is $s_k=1$.
The left-hat $\lh$ deletes the box when the leftmost KR crystal is a box, that is $r_k=s_k=1$.
(If some $r_i$ are $n$ or $n-1$, we use a ``spin'' version of the operator $\lh$ called left-hat-spin
$\lh_s$).
These three operations on the path side correspond to $\gamma$, $\beta$, $\delta$ (resp. $\delta_s$ in the
spin case) on rigged configurations. The bijection intertwines these operators and proceeds inductively on the total 
number of boxes $\sum_{i=1}^kr_is_i$. 

Our main results are threefold. Firstly, we show that the above bijection $\Phi$ is 
well-defined (see Theorem~\ref{th:welldef_highest}).  At the same time, Lusztig's involution $\lusz$ on 
$B$ is shown to be related to $\theta$, which exchanges riggings and coriggings of rigged configurations. 
Secondly, we prove the $R$-invariance of rigged configurations (see Theorem~\ref{th:main}).
For the tensor product of KR crystals, we have a nontrivial bijection 
$R \colon B^{r_1,s_1} \otimes B^{r_2, s_2} \to B^{r_2,s_2} \otimes B^{r_1, s_1}$,
called the combinatorial $R$-matrix, that commutes with all of the Kashiwara operators 
$e_i,f_i$. It can be applied to any two successive factors of a multiple tensor product
of KR crystals. $R$-invariance means that this application of $R$ does not have
any effect on a rigged configuration. We prove this by using the combinatorial $R$-matrix involving 
the spin representation $B^{n,1}$ and to reduce the problem to the $R$-invariance 
for the type $A^{(1)}_n$ case, which was shown in~\cite{KSS:2002}. 
In the proof, even though we consider the $R$-invariance for highest weight elements,
we need to use the fact that the bijection is a classical crystal isomorphism~\cite{Sak:2013}
(see Theorem~\ref{th:welldef}).
Finally, we show that under the bijection $\theta\circ\Phi$, the coenergy statistic
on a path is transferred to the cocharge on a rigged configuration (see Theorem~\ref{thm:stat}), 
thereby proving the $X=M$ conjecture for type $D_n^{(1)}$ in a bijective fashion
(see Corollary~\ref{corollary.X=M}).

Let us address the question on why a bijective proof of the identity $X=M$ is very powerful.
We give three reasons here. The first one is the computation of the image of $R$.
Although $R$ is defined naturally in a representation-theoretical fashion, it is quite
nontrivial combinatorially. However, denoting the bijection $\Phi$ from 
$B_1\ot B_2$ by $\Phi_{B_1\ot B_2}$, $R:B_1\ot B_2\rightarrow B_2\ot B_1$
is simply realized as $\Phi^{-1}_{B_2\ot B_1}\circ\Phi_{B_1\ot B_2}$ thanks to
the $R$-invariance. As a second reason, we mention the application to box-ball 
systems~\cite{TS}. These are certain integrable dynamical systems formulated
on the tensor product of KR crystals. Time evolution on the box-ball system
is defined using $R$ and considered to be nonlinear. However in~\cite{KOSTY:2006}, 
it was found that $\Phi$ linearizes its motion. Finally, as we mentioned before the bijection
in fact extends to a crystal isomorphism.

Rigged configurations of simply-laced types, and hence in particular type $D_n^{(1)}$,
are of fundamental importance since those of non-simply-laced types can be constructed
from simply-laced types by Dynkin diagram foldings. On the level of crystals, this is the virtual crystal construction 
carried out in~\cite{OSS03III,OSS03II}. For rigged configurations, the virtual crystal
construction is studied in~\cite{schilling.scrimshaw.2015}. In particular, the crystal operators on
rigged configurations for simply-laced types~\cite{Sch:2006} are extended to non-simply-laced types 
in~\cite{schilling.scrimshaw.2015}.
The algorithm $\delta$ is known for arbitrary non-exceptional affine algebras~\cite{OSS:2003a},
as well as type $E^{(1)}_6$~\cite{OSano} and $D_4^{(3)}$~\cite{scrimshaw.2015}.
Moreover, $\delta$ was shown to commute with the virtual crystal construction for types $B_n^{(1)}$
and $A_{2n-1}^{(2)}$ in~\cite{schilling.scrimshaw.2015} and $D_4^{(3)}$ in~\cite{scrimshaw.2015},
all of which are constructed as foldings of type $D_n^{(1)}$.

This paper is organized as follows.
In Section~\ref{sec:review}, we review necessary facts from crystal base theory, define KR crystals,
and prove some results about the affine crystal structure as well as properties of the left and right-splitting
map that we need.
A review of type $D_n^{(1)}$ rigged configurations is given in Section~\ref{sec:RC}.
Section~\ref{sec:weldef} contains the heart of this paper with a proof that the bijection $\Phi$ is well-defined.
In Section~\ref{sec:Rinv}, we prove the $R$-invariance of the rigged configuration bijection.
We conclude in Section~\ref{sec:energy} with a proof that $\Phi$ preserves statistics (energy and
cocharge), which implies the fermionic formula of~\cite{HKOTY99}. The appendix is reserved for an
example of the bijection.

\subsection*{Acknowledgments}
This work benefited from computations in {\sc SageMath}~\cite{combinat,sage} (using implementations
of crystals and rigged configurations by Schilling and Scrimshaw) and {\sc Mathematica} (using
an implementation of rigged configurations by Sakamoto~\cite{Sak:web}).

AS and TS would like to thank Osaka City University for kind hospitality during their stay in July 2015.
Both authors were partially supported by the JSPS Program for Advancing Strategic International Networks to Accelerate
the Circulation of Talented Researchers "Mathematical Science of Symmetry, Topology and Moduli, Evolution of
International Research Network based on OCAMI".
MO was partially supported by the Grants-in-Aid for Scientific Research No. 23340007 and No. 15K13429 from JSPS.
The work of RS was partially supported by Grants-in-Aid for Scientific Research No. 25800026 from JSPS.
AS was partially supported by NSF grants OCI--1147247 and DMS--1500050.
TS was partially supported by NSF grant OCI--1147247 and RTG grant NSF/DMS--1148634.

\section{Crystals and tableaux}\label{sec:review}

\subsection{Affine algebra of type $D^{(1)}_n$}
We consider the Kac--Moody Lie algebra of affine type $D_n^{(1)}$ whose Dynkin diagram
is depicted in Figure \ref{fig:dynkin_diagram}. We denote the index set of the Dynkin
diagram by $I=\{0,1,\ldots,n\}$ and set $I_0=I\setminus\{0\}$.

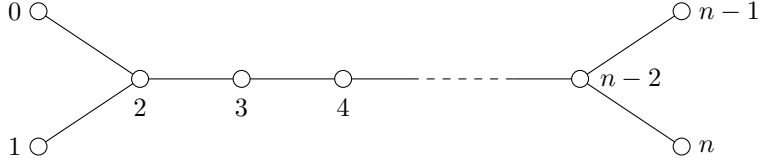
\begin{figure}[t]
\begin{tikzpicture}[scale=0.45]
\draw (0,2 cm) -- (3 cm,0);
\draw (0,-2 cm) -- (3 cm,0);
\draw (3 cm,0) -- (11 cm,0);
\draw[dashed] (11 cm,0) -- (14 cm,0);
\draw (14 cm,0) -- (16 cm,0);
\draw (16 cm,0) -- (19 cm,2 cm);
\draw (16 cm,0) -- (19 cm,-2 cm);
\draw[fill=white] (0 cm, 2 cm) circle (.25cm) node[left=3pt]{$0$};
\draw[fill=white] (0 cm, -2 cm) circle (.25cm) node[left=3pt]{$1$};
\draw[fill=white] (3 cm, 0 cm) circle (.25cm) node[below=4pt]{$2$};
\draw[fill=white] (6 cm, 0 cm) circle (.25cm) node[below=4pt]{$3$};
\draw[fill=white] (9 cm, 0 cm) circle (.25cm) node[below=4pt]{$4$};
\draw[fill=white] (16 cm, 0 cm) circle (.25cm) node[right=4pt]{$n-2$};
\draw[fill=white] (19 cm, 2 cm) circle (.25cm) node[right=3pt]{$n-1$};
\draw[fill=white] (19 cm, -2 cm) circle (.25cm) node[right=3pt]{$n$};
\end{tikzpicture}
\caption{The Dynkin diagram of type $D_n^{(1)}$.}
\label{fig:dynkin_diagram}
\end{figure}

Let $\alpha_i$, $\alpha^\vee_i$, and $\Lambda_i$ ($i\in I$) be
the simple roots, simple coroots, and fundamental weights of $D_n^{(1)}$, respectively.
Set $\varpi_i = \Lambda_i-\Lambda_0$ ($i\in I_0$), which are known as the level $0$
fundamental weights. In particular, $\alpha_i$, $\alpha^\vee_i$, and $\varpi_i$ ($i\in I_0$) can be identified with
the simple roots, simple coroots, and fundamental weights of the underlying
simple Lie algebra $D_n$. Using the standard orthonormal vectors
$\epsilon_i$ ($i\in I_0$) in the weight lattice of type $D_n$, the simple roots are
represented as
\begin{equation}
\label{simple_root}
\begin{split}
	\alpha_i&=\epsilon_i-\epsilon_{i+1}\quad (1\leq i\leq n-1),\\
	\alpha_n&=\epsilon_{n-1}+\epsilon_n,
\end{split}
\end{equation}
and the fundamental weights as
\begin{equation}
\label{fundamental_weight}
\begin{split}
	&\varpi_i = \epsilon_1+\cdots+\epsilon_i \qquad (1\leq i\leq n-2),\\
	&\varpi_{n-1} = (\epsilon_1+\cdots+\epsilon_{n-1}-\epsilon_n)/2,\\
	&\varpi_{n} = (\epsilon_1+\cdots+\epsilon_{n-1}+\epsilon_n)/2.
\end{split}
\end{equation}

Let $Q$, $Q^\vee$, and $P$ be the root, coroot, and weight lattices of type $D_n$, respectively. 
Let $\langle \cdot, \cdot \rangle \colon Q^\vee \times P \to \mathbb{Z}$ be the pairing such that 
$\langle \alpha^\vee_i, \varpi_j \rangle = \varpi_j(\alpha^\vee_i) = \delta_{i,j}$ is the Kronecker delta.
Note that $\langle \alpha^\vee_i, \alpha_j \rangle = \alpha_j(\alpha^\vee_i) = A_{i,j}$ is the Cartan matrix
of type $D_n$. The above can also be extended to the affine type $D_n^{(1)}$.

\subsection{Crystals and Kashiwara--Nakashima tableaux}
We refer to \cite{Kashiwara:1991} for the basics of crystal basis theory.
We denote by $e_i$ and $f_i$ the \defn{Kashiwara raising} and \defn{lowering operators}, repsectively.
For an element $b$ of a crystal $B$, we use the following standard notation for the 
length of the $i$-strings through $b$:
\[
\varepsilon_i(b)=
\max\{m\geq 0\, |\,e^m_i(b)\neq 0\},\qquad
\varphi_i(b)=
\max\{m\geq 0\, |\,f^m_i(b)\neq 0\}.
\]
They are related to the weight $\wt(b)$ by $\langle \alpha^\vee_i,\wt(b)\rangle=
\varphi_i(b)-\varepsilon_i(b)$.

For crystals $B_1, B_2$ of the same type, we can define their tensor product $B_2\ot B_1$ as follows.
As a set, it is the Cartesian product $B_2\times B_1$ of $B_2$ and $B_1$.
The action of the Kashiwara operators $e_i,f_i$ on 
$b_2\otimes b_1\in B_2\otimes B_1$ is given by
\begin{align*}
e_i(b_2\otimes b_1)&=
\left\{
\begin{array}{ll}
e_ib_2\otimes b_1&\mbox{ if }\varepsilon_i(b_2)>\varphi_i(b_1),\\
b_2\otimes e_ib_1&\mbox{ if }\varepsilon_i(b_2)\leq\varphi_i(b_1),\\
\end{array}
\right.\\
f_i(b_2\otimes b_1)&=
\left\{
\begin{array}{ll}
f_ib_2\otimes b_1&\mbox{ if }\varepsilon_i(b_2)\geq\varphi_i(b_1),\\
b_2\otimes f_ib_1&\mbox{ if }\varepsilon_i(b_2)<\varphi_i(b_1),\\
\end{array}
\right.
\end{align*}
where the result is declared to be $0$ if either of its tensor factors are $0$.
The weight is defined as $\wt(b_2\otimes b_1)=\wt(b_2)+\wt(b_1)$. Note that this is \emph{opposite} to the convention of Kashiwara for tensor products of crystals.

Let $B_1$ and $B_2$ be two crystals with index set $I$.  
A bijection $\psi \colon B_1 \rightarrow B_2$ is called a \defn{crystal isomorphism} if it is a bijection that
commutes with $e_i$ and $f_i$ by defining $\psi(0)=0$.

The crystals that we are concerned with in this paper are $U_q'(D_n^{(1)})$-crystals
and $U_q(D_n)$-crystals. For a subset $J \subseteq I$, we also use the terminology $J$-crystal to mean 
the crystal over the quantized enveloping algebra corresponding to the Levi subalgebra associated to $J$. 
Hence an $I_0$-crystal is nothing but a $U_q(D_n)$-crystal.

For a dominant integral weight $\la$, let $B(\la)$ be the crystal basis of the 
highest weight module of highest weight $\la$ of $U_q(D_n)$. 
$B(\la)$ has a unique element $u_\la$ satisfying 
$e_iu_\la=0$ for all $i\in I_0$. We call $u_\lambda$ the \defn{highest weight element}.
In order to perform explicit calculations on the crystal $B(\la)$,
it is convenient to use the realization by tableaux, called 
\defn{Kashiwara--Nakashima (KN) tableaux}~\cite{KN:1994}.
They are defined for $U_q(\mathfrak{g})$-crystals for Lie algebra $\mathfrak{g}$ of type
$A_n,B_n,C_n$ and $D_n$. In all of these cases, we start by looking at $B(\varpi_1)$.
In type $D_n$, the crystal graph $B(\varpi_1)$ is given as follows:
\smallskip
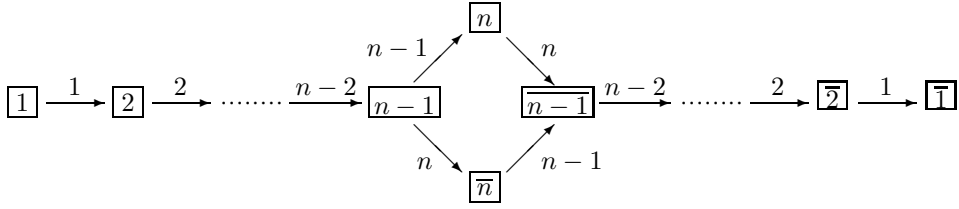
\begin{center}
\unitlength 10pt
\begin{picture}(36,7)(0,-0.2)
\put(0,3){$\Yvcentermath1\young(1)$}
\put(1.5,3.2){\vector(1,0){2.2}}
\put(2.3,3.5){1}
\put(4,3){$\Yvcentermath1\young(2)$}
\put(5.5,3.2){\vector(1,0){2.2}}
\put(6.3,3.5){2}
\multiput(8.2,3.2)(0.3,0){8}{\circle*{0.12}}
\put(10.7,3.2){\vector(1,0){2.7}}
\put(10.9,3.5){$n-2$}
\put(13.7,2.7){\frame{$\,n-1\,$\rule{0pt}{10pt}}}
\put(15.4,4.0){\vector(1,1){1.8}}
\put(13.6,5){$n-1$}
\put(15.4,2.4){\vector(1,-1){1.8}}
\put(15.5,0.7){$n$}
\put(17.5,6.2){$\Yvcentermath1\young(n)$}
\put(18.9,0.6){\vector(1,1){1.8}}
\put(18.9,5.7){\vector(1,-1){1.8}}
\put(20.2,5){$n$}
\put(20.2,0.7){$n-1$}
\put(19.5,2.7){\frame{$\,\overline{n-1}$\rule{0pt}{10pt}\,}}
\put(17.5,-0.2){$\Yvcentermath1\young(\mn)$}
\put(22.4,3.2){\vector(1,0){2.7}}
\put(22.6,3.5){$n-2$}
\multiput(25.6,3.2)(0.3,0){8}{\circle*{0.12}}
\put(28.1,3.2){\vector(1,0){2.2}}
\put(28.9,3.5){2}
\put(30.7,3){\frame{$\;\mtwo$\rule{0pt}{10pt}\;}}
\put(32.2,3.2){\vector(1,0){2.2}}
\put(33.0,3.5){1}
\put(34.8,3){\frame{$\;\mone$\rule{0pt}{10pt}\;}}
\end{picture}
\end{center}
\smallskip
Here $b\overset{i}{\longrightarrow}b'$ stands for $f_ib=b'$ or equivalently $b=e_ib'$.
The weight is given by $\wt\Bigl(\Yvcentermath1\young(i)\Bigr)=\epsilon_i$
and $\wt\Bigl(\begin{array}{|c|}\hline\raisebox{-.2ex}
{$\minusi$}\\\hline\end{array}\Bigr)=-\epsilon_i$.

We now explain KN tableaux for $B(\la)$. Suppose $\la=\sum_{i=1}^n\la_i\epsilon_i$.
Note that $\la_1\ge\cdots\ge \la_{n-1}\ge|\la_n|\ge 0$. We first assume that all $\la_i$ are integers.
The highest weight element $u_\la$ corresponds to the tableau of partition shape 
$(\la_1,\ldots,\la_{n-1},|\la_n|)$ whose entries in the $i$-th row are $\ol{n}$
if $i=n$ and $\la_n<0$, and $i$ otherwise. 
Note that we use English convention for partitions and draw the Young diagram corresponding to a partition 
with the largest row on the top.
For a given tableau
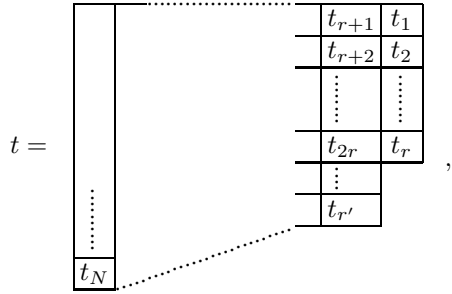
\begin{center}
\unitlength 12pt
\begin{picture}(13,10)
\put(0,5){$t=$}
\put(2,3){
\put(7,6.7){\line(1,0){4}}
\put(7,5.7){\line(1,0){4}}
\put(7,4.7){\line(1,0){4}}
\put(7,2.7){\line(1,0){4}}
\put(7,1.7){\line(1,0){4}}
\put(7,0.7){\line(1,0){2.7}}
\put(7,-0.3){\line(1,0){2.7}}
\put(11,6.7){\line(0,-1){5}}
\put(9.7,6.7){\line(0,-1){7}}
\put(7.8,6.7){\line(0,-1){7}}
\put(10,6){$t_1$}
\put(10,5){$t_2$}
\multiput(10.3,4.4)(0,-0.2){8}{\circle*{0.1}}
\put(10,2){$t_r$}
\put(8,6){$t_{r+1}$}
\put(8,5){$t_{r+2}$}
\multiput(8.3,4.4)(0,-0.2){8}{\circle*{0.1}}
\put(8,2){$t_{2r}$}
\multiput(8.3,1.5)(0,-0.2){4}{\circle*{0.1}}
\put(8,0){$t_{r'}$}
\multiput(2.4,6.7)(0.2,0){23}{\circle*{0.1}}
\put(0,6.7){\line(1,0){2.3}}
\put(0,6.7){\line(0,-1){9}}
\put(0,-2.3){\line(1,0){1.3}}
\put(0,-1.3){\line(1,0){1.3}}
\put(1.3,-2.3){\line(0,1){9}}
\multiput(0.6,-1.0)(0,0.2){10}{\circle*{0.1}}
\put(0.2,-2){$t_N$}
\multiput(1.4,-2.3)(0.2,0.07){28}{\circle*{0.1}}
}
\end{picture}
\, \raisebox{55pt}{,}
\end{center}
we introduce the so-called \defn{column reading word} 
$t_N\cdots t_2t_1$ of $t$ and regard it as an element of $B(\varpi_1)^{\otimes N}$ as
\[
\Yvcentermath1
t\longmapsto\young(\tN)\otimes\cdots\otimes
\young(\ttwo)\otimes\young(\tone)\, .
\]
Via this identification, we introduce an action of Kashiwara operators on $t$ using the tensor product rule. 
The whole set $B(\la)$ is generated from $u_\la$ by applying $f_i$ ($i\in I_0$). 

Next we introduce a representation of elements of $B(s\varpi_n)$ and 
$B(s\varpi_{n-1})$.
We consider $B(\varpi_n)$ and $B(\varpi_{n-1})$ first. As a set they are given by 
\begin{equation}
\label{equation.spin}
\begin{split}
	B(\varpi_n) &= \{(s_1,s_2,\ldots,s_n)\mid s_i=\pm,s_1s_2\cdots s_n=+\},\\
	B(\varpi_{n-1}) &= \{(s_1,s_2,\ldots,s_n)\mid s_i=\pm,s_1s_2\cdots s_n=-\}.
\end{split}
\end{equation}
The Kashiwara operators act by
\begin{subequations}
\label{eq:kashiwara_spin}
\begin{align}
e_i(s_1,\ldots,s_n)&=
\left\{
\begin{array}{ll}
(s_1,\ldots,+,-,\ldots,s_n)
&\mbox{ if }i\neq n,\,(s_i,s_{i+1})=(-,+),\\
(s_1,\ldots,s_{n-2},+,+)
&\mbox{ if }i=n,\,(s_{n-1},s_n)=(-,-),
\label{e_on_spin}
\\
0&\mbox{ otherwise,}
\end{array}
\right.\\
f_i(s_1,\ldots,s_n)&=
\left\{
\begin{array}{ll}
(s_1,\ldots,-,+,\ldots,s_n)
&\mbox{ if }i\neq n,\,(s_i,s_{i+1})=(+,-),\\
(s_1,\ldots,s_{n-2},-,-)
&\mbox{ if }i=n,\,(s_{n-1},s_n)=(+,+),\\
0&\mbox{ otherwise.}
\end{array}
\right.
\label{f_on_spin}
\end{align}
\end{subequations}
The weight is given by
\begin{equation}
\label{wt_on_spin}
\wt(s_1,\ldots,s_n)=
\frac{1}{2}(s_1\epsilon_1+\cdots+s_n\epsilon_n).
\end{equation}
In view of this, it is natural to associate to $(s_1,\ldots,s_n)$ a tableau
whose shape has half width and height $n$. In the $i$-th row, we put $s_i$.
We call it a \defn{spin column}.

For general $s$, we embed $B(s\varpi_n)$ (resp. $B(s\varpi_{n-1})$)
into $B(\varpi_n)^{\ot s}$ (resp. $B(\varpi_{n-1})^{\ot s}$) by 
$c_s\cdots c_1\longmapsto c_s\ot\cdots\ot c_1$ where $c_j$ are 
spin columns. In this way, we represent elements of $B(s\varpi_n)$ or 
$B(s\varpi_{n-1})$ by $s$ spin columns. The highest weight elements are given by
\smallskip
\begin{center}
\unitlength 10pt
\begin{picture}(11,6)(0.5,0)
\put(2,2.8){$u_{s\varpi_{n}}=$}
\put(6.2,0){\line(0,1){6}}
\put(7.1,0){\line(0,1){6}}
\multiput(8.0,0)(.9,0){2}{\line(0,1){6}}
\multiput(6.2,0)(0,1.5){5}{\line(1,0){5.2}}
\put(10.5,0){\line(0,1){6}}
\put(11.4,0){\line(0,1){6}}
\put(6.28,4.9){$+$}
\put(6.28,2.0){$+$}
\multiput(6.7,4.1)(0,-0.37){3}{\circle*{0.1}}
\put(6.28,0.5){$+$}
\put(0.8,0){
\put(6.35,4.9){$+$}
\put(6.35,2.0){$+$}
\multiput(6.7,4.1)(0,-0.37){3}{\circle*{0.1}}
\put(6.35,0.5){$+$}
}
\put(1.7,0){
\put(6.35,4.9){$+$}
\put(6.35,2.0){$+$}
\multiput(6.7,4.1)(0,-0.37){3}{\circle*{0.1}}
\put(6.35,0.5){$+$}
}
\put(4.2,0){
\put(6.35,4.9){$+$}
\put(6.35,2.0){$+$}
\multiput(6.7,4.1)(0,-0.37){3}{\circle*{0.1}}
\put(6.35,0.5){$+$}
}
\multiput(9.35,0.7)(0.37,0){3}{\circle*{0.1}}
\multiput(9.35,4.1)(0.37,-0.37){3}{\circle*{0.1}}
\multiput(9.35,2.2)(0.37,0){3}{\circle*{0.1}}
\multiput(9.35,5.2)(0.37,0){3}{\circle*{0.1}}
\end{picture}
\raisebox{0.9cm}{,}
\qquad
\unitlength 10pt
\begin{picture}(11,6)(0.5,0)
\put(1.2,2.8){$u_{s\varpi_{n-1}}=$}
\put(6.2,0){\line(0,1){6}}
\put(7.1,0){\line(0,1){6}}
\multiput(8.0,0)(.9,0){2}{\line(0,1){6}}
\multiput(6.2,0)(0,1.5){5}{\line(1,0){5.2}}
\put(10.5,0){\line(0,1){6}}
\put(11.4,0){\line(0,1){6}}
\put(6.28,4.9){$+$}
\put(6.28,2.0){$+$}
\multiput(6.7,4.1)(0,-0.37){3}{\circle*{0.1}}
\put(6.28,0.5){$-$}
\put(0.8,0){
\put(6.35,4.9){$+$}
\put(6.35,2.0){$+$}
\multiput(6.7,4.1)(0,-0.37){3}{\circle*{0.1}}
\put(6.35,0.5){$-$}
}
\put(1.7,0){
\put(6.35,4.9){$+$}
\put(6.35,2.0){$+$}
\multiput(6.7,4.1)(0,-0.37){3}{\circle*{0.1}}
\put(6.35,0.5){$-$}
}
\put(4.2,0){
\put(6.35,4.9){$+$}
\put(6.35,2.0){$+$}
\multiput(6.7,4.1)(0,-0.37){3}{\circle*{0.1}}
\put(6.35,0.5){$-$}
}
\multiput(9.35,0.7)(0.37,0){3}{\circle*{0.1}}
\multiput(9.35,4.1)(0.37,-0.37){3}{\circle*{0.1}}
\multiput(9.35,2.2)(0.37,0){3}{\circle*{0.1}}
\multiput(9.35,5.2)(0.37,0){3}{\circle*{0.1}}
\end{picture}
\raisebox{0.9cm}{.}
\end{center}

\subsection{Kirillov--Reshetikhin crystals}

In \cite{O07,OS08}, it was shown that Kirillov--Reshetikhin (KR) modules have
crystal bases for any nonexceptional affine Lie algebra $\mathfrak{g}$. We call them 
\defn{Kirillov--Reshetikhin (KR) crystals}. KR modules are finite-dimensional
$U_q'(\mathfrak{g})$-modules, where $U_q'(\mathfrak{g}) = U_q([\mathfrak{g},\mathfrak{g}])$. 
The combinatorial structure of KR crystals is explicitly given in~\cite{Sch:2008,FOS:2009}, which we 
briefly recall in this section for type $D_n^{(1)}$.

KR crystals are parametrized by $(r,s)$ ($r\in I_0,s\ge1$). The KR crystal 
indexed by $(r,s)$ is denoted by $B^{r,s}$. Since $U'_q(D_n^{(1)})$
contains $U_q(D_n)$ as a subalgebra, $B^{r,s}$ is decomposed into $U_q(D_n)$-crystals.
For $1\le r\le n-2$, we have
\begin{align}\label{eq:classical_decomposition_1}
B^{r,s}\iso\bigoplus_\lambda B(\lambda)\qquad\mbox{as }U_q(D_n)\mbox{-crystals}
\end{align}
where, identifying the weight $\la$ with the partition shape of the elements of $B(\la)$,
the direct sum is taken over all Young diagrams obtained by removing vertical
dominoes $\Yboxdim5pt\yng(1,1)$ from the rectangular shape $(s^r)$.
If $r=n-1,n$, we have
\begin{align}\label{eq:classical_decomposition_2}
B^{r,s}\iso B(s\varpi_r)\qquad\mbox{as }U_q(D_n)\mbox{-crystals}.
\end{align}

To define the affine Kashiwara operators $e_0$ and $f_0$, we first need a map $\sigma$ that is the analogue 
of the Dynkin diagram automorphism that interchanges nodes 0 and 1; see~\cite{Sch:2008}.
We begin by recalling the notion of $\pm$-diagrams. A \defn{$\pm$-diagram} $P$ is a sequence of shapes 
$\lambda \subseteq \eta \subseteq \mu$ such that $\mu / \eta$ and $\eta / \lambda$ are horizontal strips. 
We call $\lambda$ and $\mu$ the inner and outer shapes of $P$, respectively. We depict 
$P$ as a skew shape $\mu / \lambda$, where we fill the boxes of $\mu / \eta$ with $-$ and those of $\eta / \lambda$ 
with $+$. Next we define an involution $\mathfrak{S}$ on $\pm$-diagrams, where $\mathfrak{S}(P)$ is the 
$\pm$-diagram that interchanges the number of columns of a given height $h$ with only a $+$ and those
of height $h$ containing only $-$. In addition, it interchanges the number of columns
of height $2\le h\le r$ containing $\mp$ with the number of columns containing no sign of height $h-2$.

For $J\subseteq I$, we say an element $b$ of a crystal is a 
\defn{$J$-highest weight element}, if $e_ib=0$ for all $i\in J$.

\begin{proposition}[{\cite{Sch:2008,FOS:2009}}]
There exists a bijection $\kappa$ from $\pm$-diagrams to $\{2, \dotsc, n\}$-highest 
weight elements in $B^{r,s}$ with $1\le r \leq n-2$ as follows. Let $P = (\lambda \subseteq \eta \subseteq \mu)$.
Then we construct $\kappa(P)$ as follows:
\begin{enumerate}
\item Start with shape $\mu$ and add a $\mone$ in every cell that contains a $-$;
\item Fill the remainder of the columns with $23 \dotsm k$;
\item As we read the $\pm$-diagram from bottom to top (in English convention),
left to right, for every $+$ at height $h$ that is encountered, 
do one of the following, moving in the current tableau from bottom to top and left to right:
\begin{enumerate}
\item if we are at a $\mone$, replace it by $\overline{h+1}$;
\item otherwise if one encounters a $2$, replace the string $23 \dotsm k$ with $12 \dotsm h (h+2) \dotsm k$.
\end{enumerate}
\end{enumerate}
\end{proposition}

We define $\sigma \colon B^{r,s} \to B^{r,s}$ for $1\le r \leq n-2$ on $\{2, \dotsc, n\}$-highest weight elements
as $\sigma = \kappa \circ \mathfrak{S} \circ \kappa^{-1}$ and extend it to all elements in $B^{r,s}$ by
making it a $\{2, \dotsc, n\}$-crystal isomorphism. Explicitly, we have
\begin{equation}
\label{eq:01_involution}
\sigma = f_{\bf{a}^r} \circ \kappa \circ \mathfrak{S} \circ \kappa^{-1} \circ e_{\bf{a}},
\end{equation}
where 
\begin{equation} \label{e sequence}
e_{\bf{a}} b = e_{a_1} \cdots e_{a_\ell} b\quad\text{for}\quad
{\bf{a}} = a_1\cdots a_\ell
\end{equation}
such that $e_{\bf{a}} b$ is $\{2, \dotsc, n\}$-highest weight and 
$f_{\bf{a}^r} b' = f_{a_\ell} \dotsm f_{a_1} b'$.
The map $\sigma$ is an involution on $B^{r,s}$~\cite[Definition 4.2]{Sch:2008}. 
Then we define
\begin{subequations}
\label{e0 f0}
\begin{align}
e_0 & = \sigma \circ e_1 \circ \sigma,
\label{eq:e0}
\\ f_0 & = \sigma \circ f_1 \circ \sigma.
\label{eq:f0}
\end{align}
\end{subequations}

Let us introduce the following convention. By $c(i_1,\ldots,i_\ell)$, we denote the spin column whose 
$i_a$-th entry is $-$ for $1\le a\le \ell$ and $+$ elsewhere. Let $c^t d^{t'}$ stand for the tableau whose 
left $t$ columns are $c$ and right $t'$ columns are $d$.

Note that when $r=n-1,n$, we can define $\sigma$ as an involutive map from $B^{n,s}$ to $B^{n-1,s}$ and 
vice versa~\cite[Definition 6.3]{FOS:2009}. 
This $\sigma$ is also defined to be a $\{2,\ldots,n\}$-crystal isomorphism.
A $\{2,\ldots,n\}$-highest weight element of $B(s\varpi_n)$ (resp. $B(s\varpi_{n-1})$)
is given by the tableau $c()^{\alpha} c(1,n)^{s-\alpha}$ (resp. $c(n)^{\alpha} c(1)^{s-\alpha}$). The map
$\sigma$ is defined by~\cite[(2.7)]{Mohamad}
\begin{equation}
\label{eq:spin_involution}
\sigma \colon c()^{\alpha} c(1,n)^{s-\alpha} \longleftrightarrow c(n)^{s-\alpha} c(1)^{\alpha},
\end{equation}
or pictorially, we have
\[
\arraycolsep=1.1pt
\raisebox{5pt}{$\sigma \colon$} \;
\begin{array}{|c|c|c|c|c|c|}
\hline
+ & \cdots & + & - & \cdots & - \\\hline
+ & \cdots & + & + & \cdots & + \\\hline
\vdots & \ddots & \vdots & \vdots & \ddots & \vdots \\\hline
+ & \cdots & + & - & \cdots & - \\\hline
\multicolumn{3}{c}{\raisebox{10pt}{$\underbrace{\hspace{34pt}}_{\alpha}$}} & \multicolumn{3}{c}{\raisebox{10pt}{$\underbrace{\hspace{34pt}}_{s-\alpha}$}}
\end{array}
\; \raisebox{5pt}{$\longleftrightarrow$} \;
\begin{array}{|c|c|c|c|c|c|}
\hline
+ & \cdots & + & - & \cdots & - \\\hline
+ & \cdots & + & + & \cdots & + \\\hline
\vdots & \ddots & \vdots & \vdots & \ddots & \vdots \\\hline
- & \cdots & - & + & \cdots & + \\\hline
\multicolumn{3}{c}{\raisebox{10pt}{$\underbrace{\hspace{34pt}}_{s-\alpha}$}} & \multicolumn{3}{c}{\raisebox{10pt}{$\underbrace{\hspace{34pt}}_{\alpha}$}}
\end{array}
\; \raisebox{5pt}{.}
\]
With this $\sigma$, we can again define the affine crystal operators $e_0$ and $f_0$
by~\eqref{e0 f0}.

In particular, we have $B^{r,1} \iso B(\clfw_r)$ as $I_0$-crystals
for $r=n-1,n$ with the affine crystal operators, given in~\cite{S:2005}, explicitly as
\begin{subequations}
\label{eq:affine_kashiwara_spin}
\begin{align}
\label{eq:e0_spin}
e_0(s_1, \dotsc, s_n) & = \begin{cases}
(-, -, s_3, \dotsc, s_n) & \text{if $(s_1, s_2) = (+, +)$,} \\
0 & \text{otherwise,}
\end{cases}
\\ \label{eq:f0_spin}
f_0(s_1, \dotsc, s_n) & = \begin{cases}
(+, +, s_3, \dotsc, s_n) & \text{if $(s_1, s_2) = (-, -)$,} \\
0 & \text{otherwise.}
\end{cases}
\end{align}
\end{subequations}

We prepare two lemmas that will be used later. We use the notation
$e_i^{\max}b=e_i^{\varepsilon_i(b)}b$. In the following lemmas, we note that we read
columns of a KN tableau from bottom to top in accordance with our reading word. 
Spin columns are displayed in tuple notation as in~\eqref{equation.spin}.
The following lemma will be used in Section~\ref{sec:energy}.

\begin{lemma} \label{lem:e0 action}
\mbox{}
\begin{itemize}
\item[(1)] Let $2\le r\le n-2$. Let $b(\alpha) = c^{s-\alpha} c'^{\alpha} \in B^{r,s}$ $(0\le\alpha\le s)$, 
where $c = r\cdots21$ and $c' = nr\cdots31$.
Then $\varepsilon_0\bigl(b(\alpha)\bigr)=2s-\alpha$, $\varphi_0\bigl(b(\alpha)\bigr)=0,$ and 
$e_0^{\max}b(\alpha)$ is the tableau whose left $\alpha$ columns are 
$\ol{2}nr\cdots3$ and right $(s-\alpha)$ columns are $\ol{1}\ol{2}r\cdots3$.
\item[(2)] Let $b(\alpha) = c(n)^{s-\alpha} c(2)^{\alpha} \in B^{n-1,s}$ $(0\le\alpha\le s)$.
Then $\varepsilon_0\bigl(b(\alpha)\bigr) = s-\alpha$, $\varphi_0\bigl(b(\alpha)\bigr) = 0$, and 
$e_0^{\max}b(\alpha) = c(2)^{\alpha} c(1,2,n)^{s-\alpha}$.
Columns here are spin columns.
\end{itemize}
\end{lemma}

\begin{proof}
The $\{2,\dotsc, n\}$-highest weight element $e_{(r+1)^\alpha\cdots(n-1)^\alpha}f_{1^\alpha}b(\alpha)$ is 
the tableau whose left $(s-\alpha)$ columns are $r\cdots21$ and right $\alpha$ columns are 
$(r+1)\cdots32$. Note that the Kashiwara operators appearing above commute with
$e_0,f_0$. Thus (1) follows from~\cite[Lemma 9.4]{LOS12}.

Next consider (2). Notice that $e_{\bf a}b(\alpha) = c(n)^s$ where ${\bf a}=(n-1)^\alpha\cdots2^\alpha$.
Using~\eqref{eq:spin_involution}, we have
\[
	\sigma\bigl(b(\alpha)\bigr) = (f_{{\bf a}^{\bf r}}\circ\sigma\circ e_{\bf a})\bigl(c(n)^{s-\alpha}c(2)^\alpha \bigr) 
	= c(1,n)^{s-\alpha}c(1,2)^\alpha = b'.
\]

Since $\varepsilon_1(b')=s-\alpha$ and $\varphi_1(b')=0$, we have the result for 
$\varepsilon_0$ and $\varphi_0$. We then have $e_1^{s-\alpha} b'= c(2,n)^{s-\alpha}c(1,2)^\alpha=b''$.
Noting $e_{\bf a'}b'' = c()^{s-\alpha}c(1,n)^\alpha$ where ${\bf a'} = n^{s-\alpha}
(n-1)^\alpha(n-2)^s \cdots 2^s$ and calculating similarly, we obtain the desired result.
\end{proof}

For the next lemma, which will be used in the proof of Proposition~\ref{prop:spin_Rinv},
we need to characterize the elements $b \in B^{r,s}$ for $1 \leq r \leq n-2$ such that
$\varepsilon_i(b) \leq \delta_{i,n}$ for $i \in I_0$. These are the elements that differ
from the $I_0$-highest weight element $u_{\overline{\lambda}}$ by the addition of a vertical strip 
whose (column) reading word is given by $w = \cdots n \mn n \mn$, where 
$\overline{\lambda} := \wt(b) - \wt(w)$. Note that $\wt(w) \in \{\pm \epsilon_n, 0\}$ and 
$\overline{\lambda} \in P^+$.

\begin{example}
Consider the $\{1,2,3,4,5,6,7\}$-highest weight element
\[
b = \begin{tikzpicture}[baseline]
\matrix [matrix of math nodes,column sep=-.44, row sep=-.4,text height=9pt,text width=9pt,align=center,inner sep=1.4] 
 {
 	\node[draw]{1}; & 
	\node[draw]{1}; & 
	\node[draw]{1}; & 
	\node[draw,fill=gray!40]{\meight}; \\
 	\node[draw]{2}; & 
	\node[draw]{2}; & 
	\node[draw]{2}; & 
	\node[draw,fill=gray!40]{8}; \\
	\node[draw]{3}; & 
	\node[draw,fill=gray!40]{\meight}; \\
	\node[draw]{4}; & 
	\node[draw,fill=gray!40]{8}; \\
	\node[draw]{5}; \\
	\node[draw,fill=gray!40]{\meight}; \\
 };
\end{tikzpicture}
\in B(2\clfw_2+\clfw_4+\clfw_6) \subseteq B^{6,4}
\]
of type $D_8^{(1)}$. We have
\begin{align*}
	\wt(b) & = 2\clfw_2 + \clfw_5 + \clfw_7 - \clfw_8,\\ 
	w & = \meight 8 \meight 8 \meight, \\
	\wt(w) &= -\epsilon_8,\\
	\overline{\lambda} & = 2\clfw_2 + \clfw_5 = \wt(b) + \epsilon_8.
\end{align*}
Using the notation in the proof of Lemma~\ref{lemma:f0_bump} below, we have $\xi = 4$ since $y_5 = 6$ and 
$y_i = i$ for $i \leq 4$.
\end{example}

\begin{lemma}
\label{lemma:f0_bump}
Consider $b \in B(\mu) \subseteq B^{r,s}$ for $1\le r \leq n-2$ such that 
$\varepsilon_i(b) \leq \delta_{i,n}$ for all $i \in I_0$. Then $f_0 b$ is given by doing exactly one of the following:
\begin{enumerate}
\item \label{f0bump_case:column} Suppose there exists a column of height $h \ge 1$ with column reading word
$\cdots n \mn n \mn$. Then replace it with $\cdots n \mn 21$ of height $h+2$ if this yields 
a valid tableau and $h < r$.

\item \label{f0bump_case:odd} Suppose there exists a column $\Yvcentermath1\young(\mn)$ of height 1 and a column 
of height $h \geq 1$ with column reading word $\mn n \cdots \mn n 1$. Then replace the largest column of height $h$ 
with the column of height $h+2$ with column reading $\mn n \mn \cdots n \mn n \mn 21$ and the column $\young(\mn)$ 
with the column $\young(1)$ if this yields a valid tableau and $h < r$.

\item \label{f0bump_case:other} In all other cases, slide in a vertical domino $\Yvcentermath1\young(1,2)$ from 
the left at height 0 unless $\mu_1 = s$, in which case $f_0 b = 0$.
\end{enumerate}
\end{lemma}

\begin{proof}
We use the notation introduced just before the lemma.
Let $y_i$ denote the heights of cells of $\mu / \overline{\lambda}$ read from top to bottom. Let $\xi$ be the 
largest value such that $y_j = j$ for all $1 \leq j \leq \xi$ (which could be 0).
The following element is the $\{2,\dotsc,n\}$-highest weight element in the same component as $b$:
\begin{align*}
	e_{\bf{a}} b := \, & (e_{y_k} \dotsm e_{n-2} e_{n'}) \dotsm (e_{y_{\xi+1}} e_{y_{\xi+1}+1} \dotsm e_{n-2} e_{n'''})\\ 
	& (e_{y_\xi+1} \dotsm e_{n-2} e_{n''}) \dotsm (e_{y_1+1} \dotsm e_{n-2} e_n) b,
\end{align*}
where $k = \lvert \mu / \overline{\lambda} \rvert$ and $n'' = n-1,n$ depending on the parity of $\xi$ (equivalently $r$ 
since the column heights of $\mu$ must also have the same parity) and $n''' = n,n-1$ respectively (i.e., reversed parity of 
$\xi$). Let $P_b := \kappa^{-1}(e_{\bf{a}} b)$ be the corresponding $\pm$-diagram.

We first assume that there are no $\mn$ letters in the first row, that is, we are in Case~(\ref{f0bump_case:other}). 
Note that we have $e_{\bf{a}} b = u_{\mu}$, the corresponding $\pm$-diagram $P_b$ is of outer shape $\mu$ with 
only a $+$ in all columns, and $\xi = 0$ (note that all of these conditions are equivalent).
We first consider the case when $\mu_1 = s$. Then $\mathfrak{S}(P_b)$ is also of outer shape $\mu$ with 
only a $-$ in all columns. Thus $\kappa\bigl(\mathfrak{S}(P_b)\bigr)$ is given by columns of the form 
$23\dotsm h\mone$, and we obtain $\sigma(b)$ by making the entry at height $i$ an $n$ (resp. $\mn$) 
if there is a $n$ (resp. $\mn$) in the same column in $b$ at height $i+1$.
Hence there are no $1$ nor $\mtwo$ entries in $\sigma(b)$, and we have $f_1\bigl(\sigma(b) \bigr) = 0$. 
Therefore, $f_0 b = 0$ as desired.

Now we consider the case when $\mu_1 < s$. Here, $\mathfrak{S}(P_b)$ contains $s - \mu_1$ columns with a $\pm$ 
of height 2. Note that these are the only $+$ signs occurring in $\mathfrak{S}(P_b)$ since $e_{\bf{a}} b = u_{\mu}$. 
Thus each of the $+$ signs in these columns changes either a $2$ to a $1$ or a $\mone$ to a $\mtwo$ when computing 
$\kappa\bigl(\mathfrak{S}(P_b)\bigr)$, and we obtain $n,\mn$ in $\sigma(b)$ as in the previous case. 
Hence $f_1$ changes the last $1$ or $\mtwo$ in the reading word in $\sigma(b)$. So $P_{f_1(\sigma(b))}$ differs 
from $\mathfrak{S}(P_b)$ by removing the leftmost $+$ sign, and therefore $f_0 b$ differs from $b$ by the addition 
of a column $\Yvcentermath1\young(1,x)$, where $x = 2$ or $\mn$ is the rightmost entry in the second row of $b$, as desired.

Now assume we are in Case~(\ref{f0bump_case:odd}), that is, there is a (necessarily unique) column 
$\Yvcentermath1\young(\mn)$ and a column $c = \dotsm \mn n 1$ of height $h$ (we pick the largest leftmost such column
if several exist). Note that $r$ must be odd, $h \leq \xi$, and $\mu_1 = s$. We first consider the case when $\xi < r$.  
Then $P_b$ is of outer shape $\mu$ with only $+$ in every column except for a column of height $\xi$ with no sign. There 
is exactly one column $c'$ with a $+$ in $\mathfrak{S}(P_b)$, and the $+$ is at height $\xi+1$. Hence 
$\kappa\bigl(\mathfrak{S}(P_b)\bigr)$ is the tableau where the leftmost column is of the form 
$\overline{\xi + 2} k \dotsm 32$ and all other columns of height $k$ are of the form $\mone k \dotsm 32$ (or only 
contain $\mone$ if of height 1). Next we apply the sequence $f_{\bf{a}^r}$.
This will change $\overline{\xi+2}$ to $\overline{\chi+1}$, where $\chi$ is the height of the column to the right of $c'$.
Note that $\chi = 1$ precisely when $\xi = h$ and the column to the left of $c$ has height strictly greater than $c$
(which is also precisely when Case~(\ref{f0bump_case:odd}) applies and yields a valid tableau).
It is straightforward to check that $\sigma(b)$ does not contain any additional $1$ or $\mtwo$ entries; 
more explicitly, the other changed entries either become $n$ or $\mn$. Thus 
$f_1\bigl(\sigma(b) \bigr)$ changes the $\mtwo$ to a $\mone$ if Case~(\ref{f0bump_case:odd}) applies
and $f_1\bigl(\sigma(b) \bigr) = 0$ otherwise. Therefore it is easy to see that our claim follows using the fact that 
$P_{f_1(\sigma(b))}$ has a $-$ in all columns of shape $\mu'$, which is the outer shape of $\mathfrak{S}(P_b)$. 
If $\xi = r$, then $\mathfrak{S}(P_b)$ contains no $+$ signs, and it is easy to see that 
$f_1\bigl( \sigma(b) \bigr) = 0$ as there are no $1$ nor $\mtwo$ entries in $\sigma(b)$.

Lastly, assume that there exists a column $\dotsm n \mn n \mn$ of height $\ge 1$ and we are not in Case~(\ref{f0bump_case:odd}).
We first consider Case~(\ref{f0bump_case:other}), where either $\xi > h$ or $\mu_1 - \mu_{h+1} > 1$ 
(if $\ell(\mu) \leq h$, we consider $\mu_{h+1} = 0$). Therefore $P_b$ has a $+$ in every non-empty column 
except the leftmost column of height $\xi$. Hence the leftmost $+$ in $\mathfrak{S}(P_b)$ is at height $\xi+1$, and if 
$\xi = r$, then there is no such $+$. Moreover, the remaining $s - \mu_1$ number of $+$ signs occur at height 1. 
Thus $\kappa\bigl(\mathfrak{S}(P_b)\bigr)$ has a bottom left entry of $\overline{\xi+2}$ if $\xi < r$, of 
$\mtwo$ if $\mu_1 < s$ and $\xi = r$, or of $r+1$ otherwise. Thus from the description of $\bf{a}$, we have the bottom 
left entry $x$ of $\sigma(b)$ as follows. If $\xi < r$, then $x = \overline{\chi+1}$, since it transforms under $f_{y_j}$ for 
each $\chi < j < \xi$, where $\chi$ is the height of the column to the right of the column which contains the leftmost $+$ 
in $\mathfrak{S}(P_b)$, similar to above. 
If $\mu_1 < s$ and $\xi = r$, then $x = \mtwo$. Otherwise $x = n$ or $x = \mn$ depending 
on the parity of $\xi$. In addition, all $1$ and $\mtwo$ entries, of which there are $s - \mu_1$ many of them, are 
unchanged from $\kappa\bigl(\mathfrak{S}(P_b)\bigr)$. For $j > \xi$, all other changed entries are $\mn, n$ in 
$\sigma(b)$ as above. If $\mu_1 = s$, then there are no additional $+$ signs in $\mathfrak{S}(P_b)$ other 
than the one at height $\xi + 1$, and hence there are no $1, \mtwo$ entries in $\sigma(b)$. Therefore 
$f_1\bigl(\sigma(b) \bigr) = 0$ and $f_0 b = 0$ as desired. Now if $\mu_1 < s$, then $f_1$ changes the 
rightmost $1$ or $\mtwo$ in the reading word in $\sigma(b)$. Thus $P_{f_1(\sigma(b))}$ differs from 
$\mathfrak{S}(P_b)$ by removing the leftmost $+$ sign in a column of height 1. Thus it is easy to see 
that $f_0 b$ differs from $b$ as claimed.

Now we consider Case~(\ref{f0bump_case:column}), so that $\xi = h$ and $\mu_1 - \mu_{h+1} = 1$. The tableau
$\sigma(b)$ is similar to the above except that the bottom left entry has to be $\mtwo$ since it transforms under 
$f_{y_j}$ for all $j$. Thus in this case, $P_{f_1(\sigma(b))}$ differs from $\mathfrak{S}(P_b)$ by removing the 
leftmost $+$ sign at height $h+1$. Thus it can be easily checked that $f_0 b$ is as claimed.
\end{proof}

\subsection{Lusztig's involution on $B^{r,s}$}
\label{subsection.lusztig}

Let $w_0$ be the longest element of the Weyl group of type $D_n$. There exists a type $D_n$
Dynkin diagram automorphism $\tau \colon I_0 \to I_0$ satisfying
\[
w_0 \varpi_i  = -\varpi_{\tau(i)}\quad\text{and}\quad
w_0 \alpha_i  = -\alpha_{\tau(i)}.
\]
In fact, $\tau$ is the identity if $n$ is even, and interchanges $n-1$ and $n$ and
fixes all other Dynkin nodes if $n$ is odd. 

On a $U_q(D_n)$-crystal $B(\la)$, it is known~\cite{S:2005,SS:X=M} that there exists a unique involution,
called \defn{Lusztig's involution}, $\lusz \colon B(\la) \to B(\la)$ satisfying
\begin{equation}
\label{eq:dual_map}
	\wt(b^{\lusz}) = w_0 \wt(b),\quad
	(e_i b)^{\lusz} = f_{\tau(i)}b^{\lusz},\quad
	(f_ib)^{\lusz} = e_{\tau(i)}b^{\lusz}.
\end{equation}
As seen from~\eqref{eq:dual_map}, $\lusz$ sends the $I_0$-highest weight element of $B(\la)$ to the $I_0$-lowest weight element,
which is the element that satisfies $f_i b=0$ for all $i\in I_0$. By defining $\tau(0)=0$, we extend
$\tau$ to the Dynkin diagram of type $D_n^{(1)}$ and the involution $\lusz$ on the KR
crystal $B^{r,s}$.
For a crystal $B$, let $B^{\lusz}$ be the crystal with the same set as $B$, but with the crystal structure given by~\eqref{eq:dual_map}. There is a natural isomorphism of crystals
\begin{equation} \label{* tensor}
	(B_2 \otimes B_1)^{\lusz} \iso B_1^{\lusz} \otimes B_2^{\lusz}
\end{equation}
such that $(b_2 \otimes b_1)^{\lusz} = b_1^{\lusz} \otimes b_2^{\lusz}$.

\subsection{Left and right-split on $B^{r,s}$} \label{subsec:lr-split}

In~\cite{OSS:2013}, we defined the \defn{filling map} $\mathrm{fill}$ on $B^{r,s}$ for $1\le r\le n-2,s\ge1$. 
For an $I_0$-highest weight element $u_\la$, $\mathrm{fill}(u_\la)$ is a tableau of rectangular shape $(s^r)$
which does not necessarily satisfy the conditions of KN tableaux in general. However, the filling map
is necessary for the path to rigged configuration bijection.

The filling map can be defined inductively by cutting the leftmost column.

\begin{definition} \label{def:left split}
Let $\la=k_p\varpi_p+k_q\varpi_q+\sum_{0\le j<q}k_j\varpi_j$ 
($p>q,k_p,k_q>0,k_p+k_q+\sum_{0\le j<q}k_j=s$). Here we have set $\varpi_0=0$.
We define the map, which we call \defn{left-split}, $\ls \colon B^{r,s} \to B^{r,1}\ot B^{r,s-1}$
for $1\le r\le n-2,s\ge2$ as follows.
For $\la$ define a pair $(c,\la')$ of a column $c$ of height $r$ and a weight $\la'$ by:
\begin{enumerate}
\item \label{ls_p_eq_r} If $p=r$, then
\begin{align*}
c &= r \cdots 21,
\\ \la' &= (k_r-1)\varpi_r+k_q\varpi_q+\sum_{j<q}k_j\varpi_j.
\end{align*}
\item \label{ls_p_lt_r} If $p<r$ and $k_p\ge2$, then 
\begin{align*}
c &= \ol{p+1}\cdots \ol{r} p \cdots 21,
\\ \la' &= \varpi_r+(k_p-2)\varpi_p+k_q\varpi_q+\sum_{j<q}k_j\varpi_j.
\end{align*}
\item \label{ls_kp1} If $p<r$ and $k_p=1$, then
\begin{align*}
c & = \ol{p+1} \cdots \ol{r} r \cdots q(r-p+q+1) \cdots 21,
\\ \la' & = \varpi_{r-p+q}+(k_q-1)\varpi_q+\sum_{j<q}k_j\varpi_j.
\end{align*}
\end{enumerate}
We remark that one can regard $c$ as an element of the $I_0$-crystal $B^{r,1}$ by embedding
into $B(\varpi_1)^{\ot r}$ via the column reading. 
For $u_\la\in B^{r,s}$ we define $\ls(u_\la)=c\ot u_{\la'}\in B^{r,1}\ot B^{r,s-1}$.
The image $\ls(b)$ for an arbitrary element $b\in B^{r,s}$ is defined in such a way
that $e_i,f_i$ ($i\in I_0$) commute with $\ls$.
\end{definition}

For $B^{r,s}$ with $r=n-1$ or $n$, the left split map 
$\ls \colon B^{r,s} \to B^{r,1}\ot B^{r,s-1}$ is defined
for the unique $I_0$-highest weight element $u_{s\varpi_r}$ as $\ls(u_{s\varpi_r})=
u_{\varpi_r}\ot u_{(s-1)\varpi_r}$ and extended to any element again by the commutativity
with  $e_i,f_i$ ($i\in I_0$).

The \defn{right-split} map $\rs \colon B^{r,s} \to  B^{r,s-1}\ot B^{r,1}$ is defined by 
$\rs=\lusz \circ\ls\circ \lusz$ using Definition \ref{def:left split} and \eqref{* tensor}.
It also commutes with $e_i,f_i$ ($i\in I_0$). 
For $u_{s\varpi_r}\in B^{r,s}$, the right split map is given by
$\rs(u_{s\varpi_r})=u_{(s-1)\varpi_r}\ot u_{\varpi_r}$. 
For the explicit form of $\rs(u_\la)$ for general $u_\la\in B^{r,s}$ when $r\le n-2$ 
we have the following proposition.

\begin{proposition} \label{lem:rc on KN}
Let $\la=\varpi_p+\varpi_q+\mu$ where $p\ge q$ and $\mu=\sum_{a,j_a\le q}\varpi_{j_a}$.
For $u_\la\in B^{r,s}$ ($1\le r\le n-2,s\ge2$), $\rs(u_\la)$ is given by $t\ot u_{\varpi_r}$ 
where $t$ is represented as a KN tableau as follows. The first column is 
\[
\ol{p+1} \cdots \ol{r-1} \ol{r} q \cdots 21
\]
and the other part is the KN tableau for $u_{\mu}$. 
\end{proposition}

\begin{proof}
For an $I_0$-highest weight element $u$, let $u\gotolowest v$ represent
that $v$ is the $I_0$-lowest weight element corresponding to $u$.

(1) Consider the case when $p=r$. Write $\lambda = \varpi_r+\mu$. 
Then as an $I_0$-crystal we can regard $u_\la$ as $u_{\varpi_r}\ot u_\mu$ where
$u_{\varpi_r}$ corresponds to the column $c$ split in Definition~\ref{def:left split}~(i).
Since $u_{\varpi_r}\otimes u_\mu \gotolowest v_{-\varpi_r}\otimes v_{-\mu}$,
where $v_\xi$ stands for the lowest weight element of weight $\xi$, 
we see that the leftmost column of $u_\la^{\lusz}$ is $\mone\cdots\overline{r-1}\overline{r}$ and the rest is 
$v_{-\mu}$ in the KN tableau representation. Hence, by applying $\lusz \circ\ls$ we obtain $u_{\varpi_q+\mu}$,
which is the desired result.

To prove the other cases we need the following lemma.
\begin{lemma}
\label{lem:rs}
Let $r > p\ge q$. We have:

(1)
\[
\begin{array}{|c|c|}
\hline
1&1\\\hline
\vdots&\vdots\\\hline
q&\\\hline
r-p+q+1&\\\hline
\vdots&\\\hline
r&r-p+q\\\hline
\overline{r}\\\cline{1-1}
\vdots\\\cline{1-1}
\raisebox{-2pt}{$\overline{p+1}$} \\\cline{1-1}
\end{array}
\quad
\gotolowest
\quad
\begin{array}{|c|c|}
\hline
\overline{r}&p+1\\\hline
\vdots&\vdots\\\hline
&r\\\hline
&\overline{q}\\\hline
&\vdots\\\hline
&\raisebox{-2pt}{$\overline{1}$}\\\hline
\raisebox{-2pt}{$\overline{1}$}\\\cline{1-1}
\end{array}
\]
where the LHS is an $I_0$-highest weight element of weight $\varpi_p+\varpi_q$, when
viewed inside $B(\varpi_1)^{\otimes (p+q)}$. Note that these are not KN tableaux.

(2)
\[
\begin{array}{|c|}
\hline
p+1\\\hline
\vdots\\\hline
r\\\hline
\overline{q}\\\hline
\vdots\\\hline
\raisebox{-2pt}{$\overline{1}$}\\\hline
\end{array}
=e_{\bf{a}}\;
\begin{array}{|c|}
\hline
\raisebox{-1pt}{$\overline{r-p+q}$}\\\hline
\vdots\\\hline
\raisebox{-2pt}{$\overline{1}$}\\\hline
\end{array}
\]
where 
\begin{align*}
\bf{a}&={\bf{b}}_0{\bf{b}}_1\cdots{\bf{b}}_{r-p+1}{\bf{c}}_{n-r+p-q-1}{\bf{c}}_{n-r+p-q-2}\cdots{\bf{c}}_1,\\
{\bf{b}}_i&=(r-i)(r-i+1)\cdots(n-i-2)(n-i-1)^2\cdots(n-2)^2(n-1)n,\\
{\bf{c}}_j&=(q+j)(q+j+1)\cdots(r-p+q+j-1),
\end{align*}
and for a word $\bf{a}$ $e_{\bf{a}}$ is defined in \eqref{e sequence}. $\bf{a}$ contains
only letters larger than $q$.
\end{lemma}

(2) Consider the case when $r>p=q$. We have $\lambda = 2\varpi_p + \mu$. 
Then we can regard $u_\la$ as $u_{2\varpi_p}\ot u_\mu$, where $u_{2\varpi_p}$ corresponds
to the LHS of Lemma~\ref{lem:rs} (1) with $p=q$ split two times by Definition~\ref{def:left split}~(ii) and~(i). 
Then we have
\[
u_\lambda = u_{2\varpi_p}\otimes u_\mu \gotolowest
v_{-2\varpi_p} \otimes v_{-\mu}=
\begin{array}{|c|c|}
\hline
\overline{r}&p+1\\\hline
\vdots&\vdots\\\hline
&r\\\hline
&\overline{p}\\\hline
&\vdots\\\hline
\raisebox{-2pt}{$\overline{1}$}&\raisebox{-2pt}{$\overline{1}$}\\\hline
\end{array}
\otimes v_{-\mu}.
\]
Hence, the rest of $u_\la$ by cutting the leftmost column is 
\[
\begin{array}{|c|}
\hline
p+1\\\hline
\vdots\\\hline
r\\\hline
\overline{p}\\\hline
\vdots\\\hline
\raisebox{-2pt}{$\overline{1}$}\\\hline
\end{array}
\otimes v_{-\mu}=e_{\bf{a}}v_{-\varpi_r}\otimes v_{-\mu}=e_{\bf{a}}(v_{-\varpi_r}\otimes v_{-\mu}).
\]
Here we have used Lemma~\ref{lem:rs}~(2). By applying $\lusz$ we obtain 
\[
f_{\bf{a}}(u_{\varpi_r}\otimes u_\mu)=f_{\bf{a}}
\left(
\begin{array}{|c|}
\hline
1\\\hline
\vdots\\\hline
r\\\hline
\end{array}
\otimes u_\mu
\right)=
\begin{array}{|c|}
\hline
1\\\hline
\vdots\\\hline
p\\\hline
\overline{r}\\\hline
\vdots\\\hline
\raisebox{-2pt}{$\overline{p+1}$}\\\hline
\end{array}
\otimes u_\mu
\]
as desired.

(3) Finally consider the case when $r>p>q$. 
We follow the same procedure as (2) and write $\lambda =\varpi_p+\varpi_q+\mu$. 
Then we can regard $u_\la$ as $u_{\varpi_p+\varpi_q}\ot u_\mu$ where $u_{\varpi_p+\varpi_q}$
corresponds to the LHS of Lemma \ref{lem:rs} (1). Its left column is split by Definition 
\ref{def:left split} (iii). By using Lemma \ref{lem:rs} (2) again we have
\[
u_\lambda=u_{\varpi_p+\varpi_q}\otimes u_\mu \gotolowest
v_{-\varpi_p-\varpi_q}\otimes v_{-\mu}=
\begin{array}{|c|}
\hline
\overline{r}\\\hline
\vdots\\\hline
\raisebox{-2pt}{$\overline{1}$} \\\hline
\end{array}
\otimes e_{\bf{a}}(v_{-\varpi_{r-p+q}}\otimes v_{-\mu}).
\]
Applying $\lusz$ we obtain 
\[
f_{\bf{a}}(u_{\varpi_{r-p+q}}\otimes u_\mu)=f_{\bf{a}}
\left(
\begin{array}{|c|}
\hline
1\\\hline
\vdots\\\hline
r-p+q\\\hline
\end{array}
\otimes u_\mu
\right)=
\begin{array}{|c|}
\hline
1\\\hline
\vdots\\\hline
q\\\hline
\overline{r}\\\hline
\vdots\\\hline
\raisebox{-2pt}{$\overline{p+1}$}\\\hline
\end{array}
\otimes u_\mu.
\]
\end{proof}

\subsection{Paths and various operations on them} \label{subsec:path}

Let $B$ be a tensor product of KR crystals.
We start with the definition of a path of $B$.

\begin{definition}
\label{definition.paths}
A \defn{path} of $B$ is an $I_0$-highest weight element of the crystal $B$.
The set of paths of $B$ is denoted by $\mathcal{P}(B)$. We also set
$\mathcal{P}(B,\lambda)=\{b\in \mathcal{P}(B)\,|\,\wt(b)=\lambda\}$ for
a dominant integral weight $\la$.
\end{definition}

For later purposes, we introduce several operations on paths.

\begin{definition}\label{def:operations_on_tableaux}
For a path $b=b_k\otimes b_{k-1}\otimes\cdots\otimes b_1\in
B^{r_k,s_k}\otimes B^{r_{k-1},s_{k-1}}\otimes\cdots\otimes B^{r_1,s_1}$
we define the following operations:
\begin{enumerate}
\item[(1)]
Suppose that $B^{r_k,s_k}=B^{1,1}$. Then we define the operation called \defn{left-hat} by
\[
\lh(b)=b_{k-1}\otimes\cdots\otimes b_1\in
B^{r_{k-1},s_{k-1}}\otimes\cdots\otimes B^{r_1,s_1}.
\]
\item[(1')] Suppose that $B^{r_k,s_k} = B^{r_k,1}$ and $r_k = n-1,n$. Then we define the operation 
called \defn{left-hat-spin} by
\[
\lh_s(b) = b_{k-1} \otimes \cdots \otimes b_1 \in
B^{r_{k-1},s_{k-1}} \otimes \cdots \otimes B^{r_1,s_1}.
\]
\item[(2)]
Suppose that $B^{r_k,s_k}=B^{r_k,1}$ and $2 \leq r_k \leq n-2$, so that
$b_k$ has the form
$b_k=
\begin{array}{|c|}
\hline
t_1\\
\hline
\vdots\\
\hline
t_{r_k-1}\\
\hline
t_{r_k}\\
\hline
\end{array}$ .
Then we define the operation called \defn{left-box} by
\[
\mathrm{lb}(b)=
\begin{array}{|c|}
\hline
t_{r_k}\\
\hline
\end{array}\otimes
\begin{array}{|c|}
\hline
t_1\\
\hline
\vdots\\
\hline
t_{r_k-1}\\
\hline
\end{array}
\otimes b_{k-1}\otimes\cdots\otimes b_1\in B,
\]
where $B = 
B^{1,1}\otimes B^{r_k-1,1}\otimes
B^{r_{k-1},s_{k-1}}\otimes\cdots\otimes B^{r_1,s_1}$.
\item[(3)]
Suppose that $s_k>1$.
We define the operation called \defn{left-split} by
\[
\mathrm{ls}(b)=
\mathrm{ls}(b_k)\otimes b_{k-1}\otimes\cdots\otimes b_1\in B,
\]
where $B=B^{r_k,1}\otimes B^{r_k,s_k-1}\otimes B^{r_{k-1},s_{k-1}}\otimes
\cdots\otimes B^{r_1,s_1}$.
See Definition \ref{def:left split} for $\mathrm{ls}$ in the single KR crystal.
\end{enumerate}
\end{definition}

Next we define right analogs of $\lh$, $\lh_s$, $\lb$, and $\ls$. Let 
$B=B^{r_k,s_k}\ot\cdots\ot B^{r_1,s_1}$. For $b\in\p(B)$ we define
$\diamond(b) = \high(b_1^{\lusz} \ot \cdots \ot b_k^{\lusz}) \in \p(B^{\lusz})$. Here $\high(b)$ 
denotes the highest weight element in the same $I_0$-component as $b$.
\begin{enumerate}
\item[(1)] We define \defn{right-hat} $\rh \colon B \otimes B^{1,1} \to B$ by $\rh = \high \circ \lusz \circ \lh \circ\,\lusz= \diamond \circ \lh \circ \,\diamond$.

\item[(1')] We define \defn{right-hat-spin} $\rh_s \colon B \otimes B^{r,1} \to B$, where $r = n-1,n$, by 
$\rh_s = \high\circ \lusz \circ\lh_s\circ\,\lusz = \diamond \circ \lh_s \circ \,\diamond$.

\item[(2)] We define \defn{right-box} $\rb \colon B \otimes B^{r,1} \to B \otimes B^{r-1,1} \otimes B^{1,1}$ for 
$2 \leq r \leq n-2$ by $\rb = \lusz \circ \lb \circ \,\lusz = 
\diamond \circ \lb \circ \,\diamond$.

\item[(3)] We define \defn{right-split} $\rs \colon B \otimes B^{r,s} \to B \otimes B^{r,s-1} \otimes B^{r,1}$ for 
$s \geq 2$ by $\rs = \lusz \circ \ls \circ \,\lusz=\diamond \circ \ls \circ \,\diamond$.
\end{enumerate}

\begin{proposition} \label{prop:comm1}
When there are at least two KR crystals in the tensor product $B$, the left operation
$\mathrm{lx}$ commutes with the right one $\mathrm{ry}$ for any pair of
$(\mathrm{x,y})$ where $\mathrm{x,y=h,h_s,b,s}$ as long as they are well-defined.
\end{proposition}

\begin{proof}
If $\rh$ or $\rh_s$ are not involved, they clearly commute with each other since
the left operation only changes the leftmost component and the right one does
the rightmost one. Suppose the right operation is $\rh$. It commutes with any
left operations since $\lusz \circ \lh \circ \lusz$ only changes the rightmost component
and high commutes with left operations. The case of $\rh_s$ is similar.
\end{proof}

It is known~\cite{KKMMNN91} that for KR crystals $B^{r,s},B^{r',s'}$ there exists an isomorphism of crystals
\[
	R \colon B^{r,s}\ot B^{r',s'}\longrightarrow B^{r',s'}\ot B^{r,s},
\]
called \defn{combinatorial $R$-matrix}. Note that $R$ commutes with $e_i,f_i\,(i\in I)$.

\section{Rigged configurations of type $D^{(1)}_n$}
\label{sec:RC}

\subsection{Definition of rigged configurations}
\label{section.definition rc}

We define two classes of rigged configurations of type $D_n^{(1)}$: \defn{rigged configurations} and 
\defn{unrestricted rigged configurations}. Rigged configurations will turn out to be in bijection with the paths
of Definition~\ref{definition.paths}. Unrestricted rigged configurations are obtained by defining the classical 
Kashiwara operators $e_i$ and $f_i$ for $i\in I_0$ on the rigged configurations and then considering the connected
crystal components generated by the rigged configurations (which turn out to correspond to the $I_0$-highest 
weight vectors).

Let us prepare the definition of several combinatorial objects which constitute \defn{rigged configurations}
and then give the defining conditions imposed on them.
A rigged configuration consists of a sequence of partitions $\nu=\bigl(\nu^{(1)},\ldots,\nu^{(n)}\bigr)$,
called a \defn{configuration}, together with riggings $J=\bigl(J^{(1)},\ldots,J^{(n)}\bigr)$.
More precisely, if the partition $\nu^{(a)}=\bigl(\nu^{(a)}_1,\ldots,\nu^{(a)}_\ell\bigr)$ has $\ell$ parts, then
$J^{(a)}=\bigl(J^{(a)}_1,\ldots,J^{(a)}_\ell\bigr)$ is a sequence of integers called the \defn{riggings}.
The riggings are paired with the rows of the configuration as the multisets
$\bigl(\nu^{(a)},J^{(a)}\bigr) = \bigl\{\bigl(\nu^{(a)}_1,J^{(a)}_1\bigr),\ldots,\bigl(\nu^{(a)}_\ell,J^{(a)}_\ell\bigr)\bigr\}$ 
for each $a \in I_0$. We call the pairs $\bigl(\nu^{(a)}_i,J^{(a)}_i\bigr)$ \defn{strings} and associate
the pair $\bigl(\nu^{(a)},J^{(a)}\bigr)$ to the node $a\in I_0$ of the Dynkin diagram of $D_n$.
We denote by $m^{(a)}_i(\nu)$ the number of rows of length $i$ of $\nu^{(a)}$.
We identify two rigged configurations $(\nu, J)$ and $(\widetilde{\nu}, \widetilde{J})$ when 
$\bigl(\nu^{(a)}, J^{(a)}\bigr) = \bigl(\widetilde{\nu}^{(a)}, \widetilde{J}^{(a)}\bigr)$ as multisets for all $a \in I_0$.

Rigged configurations also depend on the tensor product $B=B^{r_k,s_k}\otimes\cdots\otimes B^{r_1,s_1}$.
Let $L = \bigl\{\bigl( L^{(a)}_i \bigr) \bigr\}_{a \in I_0, i \in \mathbb{Z}_{>0}}$ be the number of components $B^{a,i}$ 
within $B$, also called the \defn{multiplicity array} of $B$.
We can define a configuration $\mu(L) = (\mu^{(1)}, \dotsc, \mu^{(n-1)},\mu^{(n)})$ from $L$ by $m_i^{(a)}(\mu) = L_i^{(a)}$.
Thus the rigged configuration $(\nu,J)$ is described by the configuration and a sequence of riggings
\[
(\nu,J)=
\Bigl(
\bigl(\nu^{(1)},J^{(1)}\bigr),\ldots,\bigl(\nu^{(n-1)},J^{(n-1)}\bigr),\bigl(\nu^{(n)},J^{(n)}\bigr)
\Bigr)
\]
together with additional data under certain constraints to be described below. We denote by $Q^{(a)}_i(\nu)$
the number of boxes in the first $i$ columns of $\nu^{(a)}$:
\[
Q^{(a)}_i(\nu)=\sum_{j\ge 1}\min(i,j)m^{(a)}_j(\nu).
\]
Therefore we have $Q^{(a)}_\infty(\nu)=\lvert\nu^{(a)}\rvert$, where $\lvert \nu^{(a)} \rvert$
is the total number of boxes in $\nu^{(a)}$.
From the data $\mu(L)$ and $\nu$ we define the \defn{vacancy number}
$P^{(a)}_i(\mu(L),\nu)$ (usually abbreviated by $P^{(a)}_i(\nu)$) by the formula
\begin{equation}\label{eq:def_vacancy}
\begin{split}
P^{(a)}_i(\nu) & = Q^{(a)}_i\bigl( \mu(L) \bigr) + \sum_{b \in I_0} A_{a,b} Q^{(b)}_i(\nu)
\\ & = Q^{(a)}_i\bigl( \mu(L) \bigr) - 2Q^{(a)}_i(\nu) + \sum_{\substack{b\in I_0 \\ b \sim a}} Q^{(b)}_i(\nu),
\end{split}
\end{equation}
where $A_{a,b}$ is the Cartan matrix of $D_n$ and $a \sim b$ means that the vertices $a$ and $b$ are connected
by a single edge in the Dynkin diagram.

\begin{definition}
Fix a multiplicity array $L$. 
Then $(\nu,J)$ is a \defn{rigged configuration} of type $D_n^{(1)}$ if all the strings $\bigl(\nu^{(a)}_i,J^{(a)}_i\bigr)$ 
and the corresponding vacancy numbers satisfy the following condition for all $a\in I_0$ and $i\ge 1$
\begin{align}\label{eq:def_RC}
0\leq J^{(a)}_i\leq P^{(a)}_{\nu^{(a)}_i}(\nu).
\end{align}
The \defn{weight} $\lambda$ of the rigged configuration is defined by the relation
(sometimes called the $(L,\lambda)$-configuration condition)
\begin{equation}
\label{wt_RC_1}
\sum_{a\in I_0,i>0}im^{(a)}_i(\nu)\alpha_a=
\sum_{a\in I_0,i>0}iL^{(a)}_i\varpi_a-\lambda.
\end{equation}
The set of rigged configurations of weight $\lambda$ and multiplicity array
$L$ is denoted by $\rc(L,\la)$. We also let $\rc(L) = \bigsqcup_{\lambda \in P^+} \rc(L, \lambda)$.
\end{definition}

\begin{example}
\label{ex:rigged_config}
The following object is a rigged configuration corresponding to the tensor product
$B = B^{3,2}\otimes B^{3,1}\otimes B^{2,2}\otimes B^{1,2}\otimes B^{1,1}$ of type $D^{(1)}_5$:
\[
\begin{tikzpicture}[baseline=-30, scale=0.4]
\rpp{2,1}{1,1}{2,1}
\begin{scope}[xshift=5cm]
\rpp{3,2,1}{0,1,0}{0,1,1}
\end{scope}
\begin{scope}[xshift=11cm]
\rpp{3,2,1,1}{0,1,1,0}{0,1,1,1}
\end{scope}
\begin{scope}[xshift=17cm]
\rpp{1,1}{0,0}{0,0}
\end{scope}
\begin{scope}[xshift=22cm]
\rpp{2,1}{0,0}{0,0}
\end{scope}
\end{tikzpicture}.
\]
Here we put the vacancy number (resp. rigging) on the left (resp. right)
of the corresponding row of the configuration represented by a Young diagram.
We order the riggings for rows of the same length in the same partition weakly decreasingly from top to bottom
(since recall that we view $\bigl(\nu^{(a)},J^{(a)}\bigr) = \bigl\{\bigl(\nu^{(a)}_1,J^{(a)}_1\bigr),\ldots,
\bigl(\nu^{(a)}_\ell,J^{(a)}_\ell\bigr)\bigr\}$ as multisets).
\end{example}

Note that if we expand the weight $\lambda$ by the basis $\epsilon_i$, we can rewrite~\eqref{wt_RC_1} as follows:
\begin{equation}
\label{wt_RC_2}
\lambda=
\sum_{i \in I_0}\lambda_i\epsilon_i=
\sum_{a\in I_0}\Bigl(\lvert \mu^{(a)} \rvert \varpi_a - \lvert \nu^{(a)} \rvert \alpha_a\Bigr).
\end{equation}
Then we can use the expressions~\eqref{simple_root} and~\eqref{fundamental_weight}
to obtain the explicit expressions for the weight $\lambda_i$.
We write the weight of the rigged configuration by $\wt(\nu,J)$.

Following \cite{Sch:2006}, we introduce the classical Kashiwara operators on the rigged configurations
and use them to define the unrestricted rigged configurations.
For the string $(i,x) \in \bigl(\nu^{(a)},J^{(a)}\bigr)$, we call the quantity $P^{(a)}_i(\nu) - x$ the \defn{corigging}.

\begin{definition}
\label{definition.crystal operators on rc}
The \defn{unrestricted rigged configurations} are obtained by all possible applications of
the Kashiwara operators $f_a$ ($a \in I_0$) on the rigged configurations.
Here the Kashiwara operators $f_a$ and $e_a$ for $a\in I_0$ on unrestricted rigged configurations are defined 
as follows. Let $x$ be the smallest rigging of $\bigl(\nu^{(a)},J^{(a)}\bigr)$.
\begin{enumerate}
\item[$e_a$:]
Let $\ell$ be the minimal length of the strings of $\bigl(\nu^{(a)},J^{(a)}\bigr)$
with the rigging $x$. If $x \geq 0$, define $e_a(\nu,J)=0$.
Otherwise $e_a(\nu,J)$ is obtained by replacing the string $(\ell,x)$
by $(\ell-1,x+1)$ while changing all other riggings to keep coriggings fixed.
\item[$f_a$:]
Let $\ell$ be the maximal length of the strings of $\bigl(\nu^{(a)},J^{(a)}\bigr)$
with the rigging $x$. Then $f_a(\nu,J)$ is obtained by the following procedure.
If $x>0$, add a string $(1,-1)$ to $\bigl(\nu^{(a)},J^{(a)}\bigr)$.
Otherwise replace the string $(\ell,x)$ by $(\ell+1,x-1)$. Change other riggings to keep coriggings fixed.
If the new rigging is strictly larger than the corresponding new vacancy number,
define $f_a(\nu,J)=0$.
\end{enumerate}
Let $\RC(L)$ denote the set of all unrestricted rigged configurations generated from $\rc(L)$
by the Kashiwara operators. Let
\[
	\RC(L, \lambda) = \{ (\nu, J) \in \RC(L) \mid \wt(\nu, J) = \lambda\}.
\]
\end{definition}

\begin{example}
\label{ex:rc_kashiwara}
Consider $(\nu, J)$ in Example~\ref{ex:rigged_config}. Then $\wt(\nu, J) = 3 \clfw_1 + 3 \clfw_4 + \clfw_5$,
\begin{equation*}
f_4(\nu, J) = 
\begin{tikzpicture}[baseline=-30, scale=0.4]
\rpp{2,1}{1,1}{2,1}
\begin{scope}[xshift=5cm]
\rpp{3,2,1}{0,1,0}{0,1,1}
\end{scope}
\begin{scope}[xshift=11cm]
\rpp{3,2,1,1}{1,2,1,0}{1,2,1,1}
\end{scope}
\begin{scope}[xshift=17cm]
\rpp{2,1}{-1,0}{0,0}
\end{scope}
\begin{scope}[xshift=22cm]
\rpp{2,1}{0,0}{0,0}
\end{scope}
\end{tikzpicture},
\end{equation*}
and $f_2(\nu, J) = 0$.
\end{example}

\subsection{Operations $\delta$, $\beta$ and $\gamma$}
\label{sec:def_delta}
The rigged configuration bijection 
\[
	\Phi \colon \mathcal{P}(B,\lambda) \to \rc(L,\lambda)
\]
is a bijection between paths, the set of $I_0$-highest weight elements in a tensor products of KR crystals $B$, 
and the set of rigged configurations. Here $L$ is the multiplicity array of $B$. In this section, we prepare the main 
ingredients of the bijection. In fact, we will define the various maps not just on $\rc(L)$, but more generally
on $\RC(L)$. One of the basic operations is
\[
\delta \colon (\nu,J)\longmapsto \{(\nu',J'),k\} \; ,
\]
where $(\nu,J)$ and $(\nu',J')$ are rigged configurations
and $k \in \{1,2,\ldots,n,\overline{n},\ldots,\overline{2},\overline{1}\}$.
In the description, we call a string $(\ell,x)$ of $\bigl(\nu^{(a)},J^{(a)}\bigr)$ \defn{singular}
if we have $x = P^{(a)}_\ell(\nu)$, that is, the rigging takes the maximal possible value.
To begin with, we consider the generic case (non-spin case).

\begin{definition}\label{def:delta}
Let us consider a rigged configuration $(\nu,J) \in \RC(L)$ corresponding to the tensor product of the form 
$B^{1,1}\otimes B'$. The map
\[
	\delta \colon (\nu,J)\longmapsto \{(\nu',J'),k\}
\]
is defined by the following procedure.
Set $\ell^{(0)}=1$.
\begin{enumerate}
\item[(1)]
For $1\leq a\leq n-2$, suppose that $\ell^{(a-1)}$ is already determined.
Then we search for the shortest singular string in $(\nu^{(a)},J^{(a)})$
that is longer than or equal to $\ell^{(a-1)}$.
\begin{enumerate}
\item If there exists such a string, set $\ell^{(a)}$ to be
the length of the selected string and continue the process recursively.
If there is more than one such string, choose any of them.
\item If there is no such string, set $\ell^{(a)}=\infty$,
$k=a$ and stop.
\end{enumerate}
\item[(2)]
Suppose that $\ell^{(n-2)}<\infty$.
Then we search for the shortest singular string in $(\nu,J)^{(n-1)}$
(resp. $(\nu,J)^{(n)}$) that is longer than or equal to $\ell^{(n-2)}$
and define $\ell^{(n-1)}$ (resp. $\ell^{(n)}$) as in part (1).
\begin{enumerate}
\item If $\ell^{(n-1)}=\infty$ and $\ell^{(n)}=\infty$,
set $k=n-1$ and stop.
\item If $\ell^{(n-1)}<\infty$ and $\ell^{(n)}=\infty$,
set $k=n$ and stop.
\item If $\ell^{(n-1)}=\infty$ and $\ell^{(n)}<\infty$,
set $k=\overline{n}$ and stop.
\item If $\ell^{(n-1)}<\infty$ and $\ell^{(n)}<\infty$,
set $\ell_{(n-1)}=\max(\ell^{(n-1)},\ell^{(n)})$ and continue.
\end{enumerate}
\item[(3)]
For $1\leq a\leq n-2$, assume that $\ell_{(a+1)}$ is already defined.
Then we search for the shortest singular string in $(\nu,J)^{(a)}$
that is longer than or equal to $\ell_{(a+1)}$ and has not yet been selected
as $\ell^{(a)}$. Let $\ell_{(a)}$ be the length of this string if it exists and set $\ell_{(a)}=\infty$ otherwise.
If $\ell_{(a)}=\infty$, set $k=\overline{a+1}$ and stop.
Otherwise continue.
If $\ell_{(1)}<\infty$, set $k=\overline{1}$ and stop.
\item[(4)]
Once the process has stopped, remove the rightmost box of each selected row
specified by $\ell^{(a)}$ or $\ell_{(a)}$.
The result gives the output $\nu'$.
\item[(5)]
Define the new riggings $J'$ as follows.
For the rows that are not selected by $\ell^{(a)}$ or $\ell_{(a)}$,
take the corresponding riggings from $J$.
In order to define the remaining riggings, we use $B'$ to compute all the vacancy numbers for $\nu'$.
Then the remaining riggings are defined so that all the corresponding
rows become singular with respect to the new vacancy number.
\end{enumerate}
\end{definition}

As we will see in Section~\ref{sec:weldef}, $\delta$ on the rigged configuration side
corresponds to the left-hat operation of Definition~\ref{def:operations_on_tableaux} on the path side.
Although the above definition only deals with paths of the form $B^{1,1} \otimes B'$,
the essence of the rigged configuration bijection is contained in the operation $\delta$
which was originally discovered by \cite{KKR:1986} for type $A^{(1)}_n$
and generalized for type $D^{(1)}_n$ by \cite{OSS:2003a}.

The remaining maps involved in the rigged configuration bijection are the
counterparts of left-box and left-split of Definition~\ref{def:operations_on_tableaux}.

\begin{definition}\label{def:RC_operations}
Let $B=B^{r_k,s_k}\otimes B^{r_{k-1},s_{k-1}}\otimes\cdots\otimes B^{r_1,s_1}$ with multiplicity array $L$
and $(\nu,J) \in \RC(L)$.
\begin{enumerate}
\item[(1)]
If $B^{r_k,s_k}=B^{r,s}$ with $s>1$, then $\gamma$ replaces a length $s$ row of $\mu^{(r)}$
by two rows of lengths $s-1$ and $1$ of $\mu^{(r)}$ and otherwise leaves $(\nu,J)$ unchanged.
\item[(2)]
If $B^{r_k,s_k}=B^{r,1}$ with $1<r\le n-2$, then $\beta$ removes a length one row from $\mu^{(r)}$, adds a length 
one row to each of $\mu^{(1)}$ and $\mu^{(r-1)}$ and adds a length one singular string to each of 
$\bigl(\nu^{(a)},J^{(a)}\bigr)$ for $1\le a<r$.
\end{enumerate}
\end{definition}

We give a summary of the correspondence of the operations on the rigged configuration side and the path
side in Table~\ref{table:operations}. It will be shown in Proposition~\ref{prop:Phi_welldef1} that these operations
intertwine under $\Phi$.

\begin{table}[h]
\[
\begin{array}{|c|c|}
\hline
B & \RC(L) \\
\hline
\ls & \gamma \\
\lb & \beta \\
\lh & \delta \\
\hline
\end{array}
\]
\caption{The operations on rigged configurations and paths and their correspondence, where $L$
is the multiplicity array of $B$.}
\label{table:operations}
\end{table}

A formal definition of the rigged configuration bijection $\Phi$ and its inverse $\Phi^{-1}$ will be given in 
Section~\ref{sec:weldef}. Roughly speaking, the algorithm for $\Phi^{-1}$ is given by successive applications of 
the operators $\gamma$, $\beta$ and $\delta$ and filling in the bottom leftmost unfilled entry in $B$ with $k$ 
on each application of $\delta$, where $k$ is as given in Definition~\ref{def:delta}. The algorithm for $\Phi$ 
is the reverse procedure of each step, where for $\delta^{-1}$ we add a box to the largest singular row at most 
as long as the previously selected row. However, the well-definedness of $\Phi$ and $\Phi^{-1}$ is a highly 
nontrivial fact and will be the subject of Section~\ref{sec:weldef}. A detailed example of $\Phi^{-1}$ is given 
in Appendix~\ref{appendix:def_RCbijection}.

\begin{example}
\label{ex:rc_to_crystal}
Consider the rigged configuration $(\nu, J)$ of Example~\ref{ex:rigged_config}. Then $(\nu,J)$ corresponds 
under $\Phi^{-1}$ to the following $I_0$-highest weight element in $B$:
\[
\Phi^{-1}(\nu,J)=
\Yvcentermath1
\young(13,4,\mfive)\otimes\young(1,3,4)\otimes\young(12,2\mone)
\otimes\young(11)\otimes\young(1)\,.
\]
\end{example}

\begin{remark}
We can perform the composition $\delta \circ \beta \circ \gamma$ at once
(see, for example, \cite{OSS:2013,Sak:2013}). Suppose that we consider $B^{r,s}$ with $s>1$.
Then $\delta \circ \beta \circ \gamma$ is a modification of $\delta$. Set $\ell^{(r-1)}=s$ and start 
from the corresponding part of Step~(1) of Definition \ref{def:delta}.
For practical purposes, especially for hand computations, it is convenient to do $\delta \circ \beta \circ \gamma$ 
simultaneously. However, for the proof of several important properties of the rigged configuration bijection in
Section~\ref{sec:weldef}, it is convenient to consider three separate operations $\gamma$, $\beta$ and $\delta$.
\end{remark}

In~\cite[Proposition 3.3]{OSS:2013}, the following result is proved
with the help of~\cite{Kleber:1998}.

\begin{proposition}\label{prop:highestRC}
Let $u_\lambda$ be the $I_0$-highest weight element of $B^{r,s}$ $(1\le r\leq n-2)$ of weight $\lambda$.
Denote by $\lambda^c$ the Young diagram obtained as the complement of $\lambda$
within the rectangle of height $r$ and width $s$.
Then $\Phi(u_\lambda)$ has the following form:
\begin{itemize}
\item
$\nu^{(a)}=\lambda^c$ for $r\leq a\leq n-2$.
$\nu^{(n-1)}=\nu^{(n)}$ is obtained by replacing all dominoes $\Yboxdim4pt\yng(1,1)$
of $\lambda^c$ by $\Yboxdim4pt\yng(1)$.
\item
$\nu^{(a-i)}$ is obtained by removing the top $i$ rows from $\lambda^c$.
If $i$ exceeds the length of $\lambda^c$, we understand that $\nu^{(a-i)}=\emptyset$.
\item
All riggings are 0.
\end{itemize}
\end{proposition}
Note that all the vacancy numbers of $\Phi(u_\lambda)$ are $0$ so that the
requirement~\eqref{eq:def_RC} determines the riggings uniquely.

\subsection{Dual operations on the rigged configurations}
\begin{definition}
Let $(\nu,J) \in \rc(L)$. For each string $\bigl(\nu^{(a)}_i, J^{(a)}_i\bigr) \in \bigl(\nu^{(a)},J^{(a)}\bigr)$ of a 
given rigged configuration $(\nu,J)$, the operation $\theta$ is defined by
\begin{align*}
	\theta \colon \quad \qquad \rc(L) &\longrightarrow \rc(L)\\
	\Bigl(\nu^{(a)}_i,J^{(a)}_i\Bigr) &\longmapsto \Bigl(\nu^{(a)}_i,P^{(a)}_{\nu^{(a)}_i}(\nu)-J^{(a)}_i\Bigr)
	\in \Bigl(\widetilde{\nu}^{(a)},\widetilde{J}^{(a)}\Bigr),
\end{align*}
where $\theta \colon (\nu,J)\longmapsto (\widetilde{\nu},\widetilde{J})$. We then extend $\theta$ to 
all unrestricted rigged configurations $\RC(L)$ by extending it as a classical crystal automorphism.
\end{definition}

Rather non-trivially, we will show in Proposition~\ref{prop:Phi_welldef1}(7)
that the operator $\diamond$ on paths corresponds to the operation $\theta$
on rigged configuration under the bijection $\Phi$.
By using $\theta$, we define the dual operations of $\gamma$, $\beta$ and $\delta$:
\begin{align*}
\widetilde{\gamma}&:=\theta\circ\gamma\circ\theta,\\
\widetilde{\beta}&:=\theta\circ\beta\circ\theta,\\
\widetilde{\delta}&:=\theta\circ\delta\circ\theta.
\end{align*}
We provide direct descriptions of these operators in the case when $(\nu,J) \in \rc(L)$.
Recall that for the string $(\nu^{(a)}_i,J^{(a)}_i) \in (\nu^{(a)},J^{(a)})$, the quantity 
$P^{(a)}_{\nu^{(a)}_i}(\nu)-J^{(a)}_i$ is called the corigging.
If the rigging is 0, we call such a string \defn{cosingular}.
Let $B=B^{r_k,s_k}\otimes B^{r_{k-1},s_{k-1}}\otimes\cdots\otimes B^{r_1,s_1}$,
$L$ of the corresponding multiplicity array, and $\mu=\mu(L)$.
\begin{enumerate}
\item[$\widetilde{\gamma}$:]
If $B^{r_k,s_k}=B^{r,s}$ with $s>1$, then $\widetilde{\gamma}$ replaces a length $s$ row of $\mu^{(r)}$
by two rows of lengths $s-1$ and $1$ of $\mu^{(r)}$ and add 1 to the rigging of
each string in $(\nu^{(r)},J^{(r)})$ of length strictly less than $s$.
\item[$\widetilde{\beta}$:]
If $B^{r_k,s_k}=B^{r,1}$, where $1 < r \le n-2$, then $\widetilde{\beta}$ removes a length one row from 
$\mu^{(r)}$, adds a length one row to each of $\mu^{(1)}$ and $\mu^{(r-1)}$ and
adds a length one cosingular string to each of $(\nu^{(a)},J^{(a)})$ for $a<r$.
\item[$\widetilde{\delta}$:] Suppose $B^{r_k,s_k}=B^{1,1}$.
This is a corigging version of $\delta$.
Instead of selecting singular strings, it selects cosingular strings and makes them
into cosingular strings of lengths shortened by 1.
For unselected strings, it keeps coriggings constant by changing riggings.
\end{enumerate}

Note that $\gamma$ preserves riggings whereas $\widetilde{\gamma}$ preserves coriggings. Also note that 
$\widetilde{\beta}$ preserves vacancy numbers since $\beta$ preserves vacancy numbers~\cite[Lemma~4.2]{S:2005}.

The following very important result is quite hard to prove.

\begin{theorem}[Appendix C of~\cite{S:2005}]\label{th:delta_commute}
Let $L$ be a multiplicity array with $L_1^{(1)}\ge 2$. Then on $\rc(L)$
the operators $\delta$ and $\widetilde{\delta}$ commute:
\[
[\delta,\widetilde{\delta}]=0.
\]
\end{theorem}

We will also need the following propositions.

\begin{proposition} \label{prop:comm2}
We have the following relations on rigged configurations in $\rc(L)$:
\begin{enumerate}
\item $[\delta,\widetilde{\delta}] = [\delta,\widetilde{\beta}] = [\delta,\widetilde{\gamma}] = 0$;
\item $[\beta,\widetilde{\delta}] = [\beta,\widetilde{\beta}] = [\beta,\widetilde{\gamma}] = 0$;
\item $[\gamma,\widetilde{\delta}] = [\gamma,\widetilde{\beta}] = [\gamma,\widetilde{\gamma}] = 0$.
\end{enumerate}
\end{proposition}

\begin{proof}
The relation $[\delta,\widetilde{\delta}] = 0$ is Theorem~\ref{th:delta_commute}. The relation 
$[\delta, \widetilde{\beta}] = 0$ is proven in~\cite[Lemma~5.4]{S:2005}, whereas $[\delta,\widetilde{\gamma}] = 0$ 
follows from the fact that $\widetilde{\gamma}$ preserves coriggings, specifically singular strings stay singular. 
The relations $[\beta,\widetilde{\beta}] = [\beta,\widetilde{\gamma}] = 0$ follow from the fact that $\beta$ and 
$\widetilde{\beta}$ preserve vacancy numbers and that $\widetilde{\gamma}$ preserves coriggings. The relation 
$[\gamma,\widetilde{\gamma}] = 0$ follows from the fact that $\gamma$ and $\widetilde{\gamma}$ preserve 
riggings and coriggings respectively. The remaining relations can be deduced by conjugation by $\theta$.
\end{proof}

\begin{proposition} \label{prop:rs vs gamma-tilde}
Let $1\le r\le n-2,s\ge2$. For $u_\la\in B^{r,s}$ we have
\begin{equation} \label{rs vs gamma-tilde}
\Phi(\rs(u_\la))=\widetilde{\gamma}(\Phi(u_\la)).
\end{equation}
\end{proposition}

\begin{proof}
Let $\la=k_p\varpi_p+k_q\varpi_q+\sum_{0\le j<q}k_j\varpi_j$ ($p>q,k_p,k_q>0,k_p+k_q+\sum_{0\le j<q}k_j=s$).
We prove the statement by dividing into the same cases as in Definition \ref{def:left split}.

Consider Case~(\ref{ls_p_eq_r}), where $p=r$. In this case, $\ls(u_\la)=c\ot u_{\la-\varpi_r}$, where $c=r \cdots 21$
by Definition \ref{def:left split} (\ref{ls_p_eq_r}). Hence, during the process of removing the leftmost column of $u_\la$,
the corresponding rigged configuration never changes. Proposition~\ref{lem:rc on KN} says that the left
component of $\rs(u_\la)$ is $u_{\la-\varpi_r}$. Therefore, the removing procedures of $\delta$ performed to obtain
the corresponding rigged configuration is completely parallel from $u_{\la-\varpi_r}$ and from $\rs(u_\la)$,
which justifies \eqref{rs vs gamma-tilde}.

Next consider Case~(\ref{ls_p_lt_r}), where $p<r$ and $k_p\ge2$. In this case, the left component of $\ls(u_\la)$
is the same as the leftmost column of $\rs(u_\la)$, which is $\ol{p+1}\cdots \ol{r}p \cdots 1$, the leftmost 
column of the right component of $\ls(u_\la)$ is $r \cdots 1$, and the right $s-2$ columns of $\ls(u_\la)$
and the right $s-2$ columns of the left component of $\rs(u_\la)$ are the same. Hence, the applications
of $\delta$ to remove the leftmost column of $u_\la$ and the one of the left component of $\rs(u_\la)$
are completely parallel and arrive at Case~(\ref{ls_p_eq_r}). Thus, the proof is done also in this case.

Finally, we consider Case~(\ref{ls_kp1}), where $p<r$ and $k_p=1$. Divide further into
two cases: (a) $s>2$ and (b) $s=2$. First consider Case~(a), and set 
$\la=\varpi_p+\varpi_q+\varpi_{q'}+\dots$ ($r>p>q\ge q'$). Just as in Case~(\ref{ls_p_lt_r}),
the right $s-2$ columns of $u_\la$ and the right $s-2$ columns of the left 
component of $\rs(u_\la)$ are the same. Hence, it is sufficient to show that
the process of removing the leftmost column of the left component of $\rs(u_\la)$
is parallel to removing the left two columns of $u_\la$. We list below the length of 
the row in which a box is removed by the $i$th application of $\delta$ along the 
leftmost column of the left component of $\rs(u_\la)$. Here $a$ stands for the 
position of the configuration when we move along it in the increasing order of $a$
during the process of $\delta$ while $\ol{a}$ the one when we move in the decreasing 
order, and $\ell$ is the length of the row in which a box is removed.
\begin{itemize}
\item[] When $1\le i\le p-q$:
\begin{itemize}
\item[] if $i$ is odd,\\ $a=r-i+1,\dots,r-1\rightarrow \ell=s-2$ and
	$a=r,\dots,\ol{q+i}\rightarrow \ell=s-1$;
\item[] if $i$ is even,\\ $a=r-i+1,\dots,\ol{r}\rightarrow \ell=s-2$ and
	$a=\ol{r-1},\dots,\ol{q+i}\rightarrow \ell=s-1$.
\end{itemize}
\item[]When $p-q< i\le r-q$, set $j=i-(p-q)$:
\begin{itemize}
\item[] if $i$ is odd,\\ $a=r-i+1,\dots,r-j\rightarrow \ell=s-2$, 
	$a=r-j+1,\dots,r-1\rightarrow \ell=s-1$ and
	$a=r,\dots,\ol{q+i}\rightarrow \ell=s$;
\item[] if $i$ is even,\\ $a=r-i+1,\dots,r-j\rightarrow \ell=s-2$, 
	$a=r-j+1,\dots,\ol{r}\rightarrow \ell=s-1$ and
	$a=\ol{r-1},\dots,\ol{q+i}\rightarrow \ell=s$.
\end{itemize}
\item[]When $r-q<i\le r-q'$:
\begin{itemize}
\item[] $a=r-i+1,\dots,r-i+p-q\rightarrow \ell=s-2$.
\end{itemize}
\item[]When $r-q'<i\le r$:
\begin{itemize}
\item[] no box is removed.
\end{itemize}
\end{itemize}
In view of \cite[Proposition 3.3]{OSS:2013}, we see that the above deletions
are in fact parallel to removing the left two columns of $u_\la$, and hence the
proof is done.

For Case~(b), set $\la=\varpi_p+\varpi_q$ ($r>p>q$) and $z=(p+q)/2$. We provide 
a similar list.
\begin{itemize}
\item[] When $1\le i\le (p-q)/2$:
\begin{itemize}
\item[] $a=r-i+1,\dots,\ol{q+i}\rightarrow \ell=1$.
\end{itemize}
\item[]When $(p-q)/2< i\le r-z$, set $j=i-(p-q)/2$:
\begin{itemize}
\item[] if $i$ is odd,\\ $a=r-i+1,\dots,r-1\rightarrow \ell=1$ and
	$a=r,\dots,\ol{p+j}\rightarrow l=2$;
\item[] if $i$ is even,\\ $a=r-i+1,\dots,\ol{r}\rightarrow \ell=1$ and
	$a=\ol{r-1},\dots,\ol{p+j}\rightarrow \ell=2$.
\end{itemize}
\item[]When $r-z<i\le r$:
\begin{itemize}
\item[] no box is removed.
\end{itemize}
\end{itemize}
\end{proof}

\subsection{Spin cases}
Following~\cite{S:2005}, let us introduce the analogue of $\delta$ for the spin cases $B^{r,1}$ where $r=n-1$
or $n$. Fix a multiplicity array $L$ and let $\mu = \mu(L)$. We introduce the embedding of the unrestricted 
rigged configurations
\begin{equation*}
\begin{split}
	\emb \colon \RC(L) &\longrightarrow \RC(L')\\
	(\nu,J) & \longmapsto (\nu',J'),
\end{split}
\end{equation*}
where $\nu'^{(a)}_i=2\nu^{(a)}_i$, $J'^{(a)}_i=2J^{(a)}_i$, and $L'$ is determined by
$\mu_i'^{(a)}=2\mu^{(a)}_i$ for $\mu'=\mu(L')$. Note that $\emb$ is a similarity map with a scaling factor of $2$ 
as described in~\cite[Theorem~3.1]{K96} in the following sense:
\begin{subequations}
\label{eq:emb_crystal_morphism}
\begin{align}
\emb( e_i b ) & = e_i^2\bigl( \emb(b) \bigr),
\\ \emb( f_i b ) & = f_i^2\bigl( \emb(b) \bigr),
\\ \wt\bigl( \emb(b) \bigr) & = 2 \wt(b).
\end{align}
\end{subequations}

\begin{remark}
The image of $\emb$ is characterized by the condition that all parts of $\mu$ and $\nu$ as well as all 
riggings are even.
\end{remark}

The analogue of $\delta$ for the spin case is denoted by $\delta_s$. It corresponds to removing an
entire spin column $B^{r,1}$ for $r=n-1$ or $n$ and is defined as
\begin{equation}
\label{eq:delta_spin}
	\delta_s = \emb^{-1} \circ (\delta \circ \beta)^{n-2} \circ \delta \circ \overline{\beta} \circ \delta \circ
	\beta^{(r)} \circ \emb,
\end{equation}
where $\beta^{(r)}$ and $\overline{\beta}$ are given as follows:
\begin{itemize}
\item Define $\beta^{(n)}$ as the map which replaces a length two row of $\mu^{(n)}$ by a length one row,
adds a length one row to each of $\mu^{(1)}$ and $\mu^{(n-1)}$, and
adds a length one singular string to each of $(\nu^{(a)},J^{(a)})$ for $a\leq n-2$ and $a=n-1$.
The map $\beta^{(n-1)}$ is defined in the same way with $n$ and $n-1$ interchanged.

\item Define $\overline{\beta}$ as the map which removes a length one row from each of $\mu^{(n-1)}$ and $\mu^{(n)}$,
adds a length one row to each of $\mu^{(1)}$ and $\mu^{(n-2)}$, and
adds a length one singular string to each of $(\nu^{(a)},J^{(a)})$ for $a\leq n-2$.
\end{itemize}

The following lemma is a straightforward computation.
\begin{lemma}
\label{lemma:spin_betas_preserve}
The maps $\beta^{(r)}$ and $\overline{\beta}$ preserve the vacancy numbers.
\end{lemma}

It was shown in~\cite{S:2005} that $\delta_s$ is well-defined.
Similarly, we define $\widetilde{\delta}_s = \theta \circ \delta_s \circ \theta$.

\begin{proposition}
\label{prop:delta_s_commute}
We have the following relations on rigged configurations in $\rc(L)$:
\begin{enumerate}
\item \label{item:delta} $[\delta_s,\widetilde{\delta}_s] = [\delta_s,\widetilde{\delta}] = [\widetilde{\delta}_s, \delta] = 0$,
\item \label{item:other_tilde} $[\delta_s, \widetilde{\beta}] = [\delta_s, \widetilde{\gamma}] = 0$,
\item \label{item:other} $[\widetilde{\delta}_s, \beta] = [\widetilde{\delta}_s, \gamma] = 0$.
\end{enumerate}
\end{proposition}

\begin{proof}
We note that $\theta$ commutes with $\emb$ and $\emb^{-1}$. Additionally, we have
\[
\emb \circ \delta = \delta \circ \delta \circ \gamma \circ \emb
\]
as a similar statement and proof of~\cite[Lemma~3.5]{O:15} holds for type $D_n^{(1)}$. Hence the 
relations for~(\ref{item:delta}) follow from Theorem~\ref{th:delta_commute}, Proposition~\ref{prop:comm2},
and the definitions of $\widetilde{\delta}_s$ and $\delta_s$.

Fix an $1\le r\le n$ such that $\mu^{(r)}$ has a row of length $s > 2$.
Define $\widetilde{\gamma}^e$ by replacing a length $s$ row of $\mu^{(r)}$ by two rows of length 
$s-2$ and $2$ of $\mu^{(r)}$, leaving $\nu$ unchanged, and preserving all colabels. Also for $1\le r \leq n - 2$, define 
$\widetilde{\beta}^e$ by removing a length $2$ row from $\mu^{(r)}$, adding a length $2$ row 
to each of $\mu^{(1)}$ and $\mu^{(r-1)}$, and adding a length $2$ string with rigging $0$ to each 
of $(\nu^{(a)}, J^{(a)})$ for $1 \leq a < r$. It is straightforward to verify that
\[
\emb \circ \widetilde{\gamma} = \widetilde{\gamma}^e \circ \emb,
\qquad
\emb \circ \widetilde{\beta} = \widetilde{\beta}^e \circ \emb.
\]

We have $[\delta, \widetilde{\gamma}^e] = [\beta, \widetilde{\gamma}^e] = [\beta^{(m)}, \widetilde{\gamma}^e] 
= [\overline{\beta}, \widetilde{\gamma}^e] = 0$, for $m = n-1, n$, since $\widetilde{\gamma}^e$ preserves 
colabels and since Lemma~\ref{lemma:spin_betas_preserve} holds, similar to the proof of Proposition~\ref{prop:comm2}. 
Hence $[\delta_s, \widetilde{\gamma}] = 0$ follows from the definition of $\delta_s$.

The proof of $[\delta, \widetilde{\beta}^e] = 0$ follows the proof for $[\delta, \widetilde{\beta}] = 0$ 
given in~\cite[Lemma~5.4]{S:2005}.
We need to show that for the selected strings $\ell^{(a)}$ (resp. $s^{(a)}$) by $\delta$ in $(\nu, J)$ 
(resp. $\widetilde{\beta}^e(\nu, J)$), we have $\ell^{(a)} = s^{(a)}$ for all $a \in I_0$. The case for 
$\overline{\ell}^{(a)} = \overline{s}^{(a)}$ is similar. First, we have that $\ell^{(a)} = 1$
if and only if $s^{(a)} = 1$ for any $a \in I_0$ as $\delta$ selects the smallest singular string and 
$\widetilde{\beta}^e$ preserves the vacancy numbers and adds a string of length $2$.
Note that we cannot have $s^{(a)} > \ell^{(a)}$ because $\widetilde{\beta}^e$ preserves the vacancy numbers 
and possibly only adds a row of length 2 to $\nu^{(a)}$.
Next, suppose $a$ is minimal such that $2 = s^{(a)} < \ell^{(a)}$, thus we must have $P_2^{(a)} = 0$ 
as the rigging of the string added by $\widetilde{\beta}^e$ is $0$ (i.e., $\widetilde{\beta}^e$ adds a 
cosingular string of length $2$). Hence we must have $m_2^{(a)} = 0$. From the definition of the vacancy 
numbers, we have
\begin{equation}
\label{eq:explicit_convexity}
-P_{i-1}^{(a)} + 2P_i^{(a)} - P_{i+1}^{(a)} = L_i^{(a)} - \sum_{b \in I_0} A_{ab} m_i^{(b)},
\end{equation}
and in particular for $i = 2$, we have
\[
-P_1^{(a)} - P_3^{(a)} = L_2^{(a)} + \sum_{b \sim a} m_2^{(b)}.
\]
Since $P_i^{(a)} \geq 0$ for all $i > 0$, we must have $P_1^{(a)} = L_2^{(a)} = m_2^{(b)} = 0$ for all $b \sim a$.

Recall from the definition of $\delta$ that $s^{(b)} \leq s^{(a)}$ for all $b < a$, and hence
$s^{(a-1)} \leq 2$. If $s^{(a-1)} = 2$, then by the assumption $a$ is minimal such that $s^{(a)} < \ell^{(a)}$, 
we have $2 = s^{(a-1)} = \ell^{(a-1)}$. However, this contradicts that $m_2^{(a-1)} = 0$, and therefore we 
must have $s^{(a-1)} = \ell^{(a-1)} = 1$.
This implies that $s^{(b)} = \ell^{(b)} = 1$ for all $b < a$. If $m_1^{(a)} > 0$, then $\ell^{(a)} = 1$ because 
$0 \leq x \leq P_1^{(a)} = 0$, which contradicts our assumption that $2 < \ell^{(a)}$. Hence we have 
$m_1^{(a)} = 0$. Consider the case when $a = 1$, and~\eqref{eq:explicit_convexity} for $i = 1$ results in
\[
0 = L_1^{(1)} + m_1^{(2)},
\]
which is a contradiction since $L_1^{(1)} > 0$ as we are applying $\delta$.
Now consider the case $a > 1$. Hence,~\eqref{eq:explicit_convexity} for $i = 1$ results in
\[
0 = m_1^{(a-1)} + m_1^{(a+1)} + L_1^{(a)},
\]
and therefore $m_1^{(a-1)} = 0$, which is a contradiction. Thus, we have $[\delta, \widetilde{\beta}^e] = 0$.

Next we have $[\beta^{(m)}, \widetilde{\beta}^e] = [\overline{\beta}, \widetilde{\beta}^e] = 0$ as all these maps 
preserve vacancy numbers. Thus $[\delta_s, \widetilde{\beta}^e] = 0$ from the definition of $\delta_s$.
This proves~\eqref{item:other_tilde}.

The remaining relations follow from conjugation by $\theta$.
\end{proof}

\section{Proof of the well-definedness of the bijection}\label{sec:weldef}
The main purpose of this section is to show that the bijection $\Phi$ is well-defined.

In the proof, we will use diagrams of the following kind as in \cite{KSS:2002}:
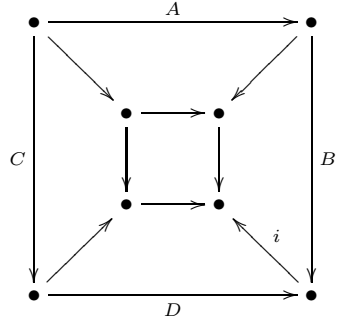
\begin{equation*}
\xymatrix{
 {\bullet} \ar[rrr]^{A} \ar[ddd]_{C} \ar[dr] & & &
        {\bullet} \ar[ddd]^{B} \ar[dl] \\
 & {\bullet} \ar[r] \ar[d] & {\bullet} \ar[d] & \\
 & {\bullet} \ar[r]   & {\bullet}  & \\
 {\bullet} \ar[rrr]_{D} \ar[ur] & & & {\bullet} \ar[ul]_{i}
}
\end{equation*}
We regard this as a cube with front face given by the large square.
Suppose that the square diagrams given by the faces of the cube except
for the front face commute and $i$ is the injective map.
Then the front face also commutes since we have
\[
i\circ B\circ A=i\circ D\circ C
\]
by diagram chasing.

\begin{proposition}\label{prop:Phi_welldef1}
Let $B=B^{r_k,s_k} \otimes \cdots \otimes B^{r_2,s_2} \otimes B^{r_1,s_1}$ be a tensor
product of KR crystals with multiplicity array $L$.
Then there exists a unique family of injections $\Phi \colon \p(B,\lambda) \to \rc (L,\lambda)$
such that the empty path maps to the empty rigged configuration and satisfy
the following sequence of commutative diagrams.
\begin{enumerate}
\item[(1)]
Suppose $B=B^{1,1}\otimes B'$.
Let $\lh(B)=B'$ with multiplicity array $\lh(L)$.
Then the diagram
\[
\xymatrix{
\p(B,\lambda) \ar[r]^{\Phi} \ar[d]_{\lh} & \rc(L,\lambda) \ar[d]^{\delta} \\
\bigcup_{\mu}\p(\lh(B),\mu) \ar[r]_{\Phi} & \bigcup_{\mu}\rc(\lh(L),\mu)
}
\]
commutes.

\item[(1')]
Suppose $B=B^{r,1}\otimes B'$ for $r=n-1$ or $n$.
Let $\lh_s(B)=B'$ with multiplicity array $\lh_s(L)$.
Then the diagram
\[
\xymatrix{
\p(B,\lambda) \ar[r]^{\Phi} \ar[d]_{\lh_s} & \rc(L,\lambda) \ar[d]^{\delta_s} \\
\bigcup_{\mu}\p(\lh_s(B),\mu) \ar[r]_{\Phi} &
\bigcup_{\mu}\rc(\lh_s(L),\mu)
}
\]
commutes.

\item[(2)]
Suppose $B=B^{r,1}\otimes B'$ with $2\le r\leq n-2$.
Let $\lb (B)=B^{1,1}\otimes B^{r-1,1}\otimes B'$
with multiplicity array $\lb (L)$.
Then the diagram
\[
\xymatrix{
\p(B,\lambda) \ar[r]^{\Phi} \ar[d]_{\lb} & \rc(L,\lambda) \ar[d]^{\beta} \\
\p(\lb(B),\lambda) \ar[r]_{\Phi} & \rc(\lb(L),\lambda)
}
\]
commutes.

\item[(3)]
Suppose $B=B^{r,s}\otimes B'$ with $s\geq 2$.
Let $\ls(B)=B^{r,1}\otimes B^{r,s-1}\otimes B'$ with multiplicity array $\ls(L)$.
Then the diagram
\[
\xymatrix{
\p(B,\lambda) \ar[r]^{\Phi} \ar[d]_{\ls} & \rc(L,\lambda) \ar[d]^{\gamma}\\
\p(\ls(B),\lambda) \ar[r]_{\Phi} & \rc(\ls(L),\lambda)
}
\]
commutes.

\item[(4)]
Suppose $B=B'\otimes B^{1,1}$.
Let $\rh(B)=B'$ with multiplicity array $\rh(L)$.
Then the diagram
\[
\xymatrix{
\p(B,\lambda) \ar[r]^{\Phi} \ar[d]_{\rh} & \rc(L,\lambda) \ar[d]^{\widetilde{\delta}} \\
\bigcup_{\mu}\p(\rh(B),\mu) \ar[r]_{\Phi} & \bigcup_{\mu}\rc(\rh(L),\mu)
}
\]
commutes.

\item[(4')]
Suppose $B=B'\otimes B^{r,1}$ for $r=n-1$ or $n$.
Let $\rh_s(B)=B'$ with multiplicity array $\rh_s(L)$.
Then the diagram
\[
\xymatrix{
\p(B,\lambda) \ar[r]^{\Phi} \ar[d]_{\rh_s} & \rc(L,\lambda) \ar[d]^{\widetilde{\delta}_s} \\
\bigcup_{\mu}\p(\rh_s(B),\mu) \ar[r]_{\Phi} & \bigcup_{\mu}\rc(\rh_s(L),\mu)
}
\]
commutes.

\item[(5)]
Suppose $B=B'\otimes B^{r,1}$ with $2\le r\le n-2$.
Let $\rb(B)=B'\otimes B^{r-1,1}\otimes B^{1,1}$
with multiplicity array $\rb(L)$.
Then the diagram
\[
\xymatrix{
\p(B,\lambda) \ar[r]^{\Phi} \ar[d]_{\rb} & \rc(L,\lambda) \ar[d]^{\widetilde{\beta}} \\
\p(\rb(B),\lambda) \ar[r]_{\Phi} & \rc(\rb(L),\lambda)
}
\]
commutes.

\item[(6)]
Suppose $B=B'\otimes B^{r,s}$ with $s\geq 2$.
Let $\rs(B)=B'\otimes B^{r,s-1}\otimes B^{r,1}$ with multiplicity array $\rs(L)$.
Then the diagram
\[
\xymatrix{
\p(B,\lambda) \ar[r]^{\Phi} \ar[d]_{\rs} & \rc(L,\lambda) \ar[d]^{\widetilde{\gamma}} \\
\p(\rs(B),\lambda) \ar[r]_{\Phi} & \rc(\rs(L),\lambda)
}
\]
commutes.

\item[(7)]
The diagram
\[
\xymatrix{
\p(B,\lambda) \ar[r]^{\Phi} \ar[d]_{\diamond} &\rc(L,\lambda) \ar[d]^{\theta} \\
\p(B^{\lusz},\lambda) \ar[r]_{\Phi} & \rc(L,\lambda)
}
\]
commutes.
\end{enumerate}
\end{proposition}

\begin{proof}
For $B=B^{r_k,s_k}\otimes\cdots\otimes B^{r_2,s_2}\otimes B^{r_1,s_1}$, we set
\[
	\| B\|=\left(\sum_ir_is_i,\,\sum_i(r_i-1),\,\sum_i(s_i-1)\right).
\]
We prove the statement by induction on the lexicographic order of $\| B\|$.
More precisely, at each induction level, we check that $\Phi$ is well-defined
from (1), (1'), (2), (3) and show that this $\Phi$ satisfies (4)---(7) for the next induction step.

The well-definedness of (1) is shown in~\cite{OSS:2003a}.

The well-definedness of (1') is shown in~\cite{S:2005}.

The well-definedness of (2) is shown in~\cite{S:2005}.

Now we prove the well-definedness of (3).
When $B=B^{r,s}$ for $1\le r\le n-2$, the well-definedness of $\Phi$ and, in particular, (3)
are established in \cite[Theorem 5.9]{OSS:2013} by directly computing
the bijection based on Proposition \ref{prop:highestRC} and checking
that the result agrees with the definition of the Kirillov--Reshetikhin tableau
for the $I_0$-highest weight element of $B^{r,s}$. 
When $B=B^{r,s}$ for $r=n-1$ or $n$, both $\p(B^{r,s})$ and $\rc(L)$ consist of a single
element and the property is easy to check.

For more general $B = B^{r,s} \otimes B'$ with $s\ge 2$, consider the following diagram:
\begin{equation*}
\xymatrix{
 {\p(B)} \ar@{-->}[rrr]^{\Phi} \ar[ddd]_{\diamond} \ar[dr]_{\ls}& & & {\rc(L)} \ar[ddd]^{\theta} \ar[dl]_{\gamma} \\
 {\hspace{11mm}}& {\p(\ls{(B)})} \ar[r]^{\Phi} \ar[d]_{\diamond} & {\rc(\ls(L))} \ar[d] ^{\theta}&\\
 & {\p(\rs(B^{\lusz}))} \ar[r]^{\Phi}  & {\rc(\rs(L))}  &\\
 {\p(B^{\lusz})}  \ar[ur] ^{\rs}& & &{\rc(L)} \ar[ul]_{\widetilde{\gamma}}
}
\end{equation*}
We see that the left face and right face commute by definition
and the back face commutes by the induction hypothesis that (7) holds.
We wish to show that $(\Phi\circ\ls)(\p(B))\subseteq\gamma(\rc(L))$.
If we have this relation, we can define $\Phi(\p(B))=(\gamma^{-1}\circ\Phi\circ\ls)(\p(B))$
since $\gamma$ is invertible on $\mathrm{Im}(\gamma)$.
Since $\diamond$ and $\theta$ are bijections,
the commutativity of the left, back and right faces of the above diagram shows that
it is enough to prove that  $(\Phi\circ\rs)(\p(B^{\lusz}))\subseteq\widetilde{\gamma}(\rc(L))$.

Let us show $(\Phi\circ\rs)(\p(B))\subseteq\widetilde{\gamma}(\rc(L))$.
For this it is enough to check that the strings of $\nu^{(r)}$
of the image of $\Phi\circ\rs$ have strictly positive riggings.
We prove this claim by dividing into cases.

When $B=B^{1,1}\otimes B'\otimes B^{r,s}$ consider the following diagram:
\begin{equation*}
\xymatrix{
 {\p(B)} \ar[rrr]^{\Phi} \ar[ddd]_{\lh} \ar[dr]_{\rs}& & &  {\rc(L)} \ar[ddd]^{\delta_1}\\
 {\hspace{17mm}} & {\p(\rs{(B)})} \ar[r]^{\Phi} \ar[d]_{\lh} & {\rc(\rs(L))} \ar[d] ^{\delta_2}& \\
 & {\p(\lh(\rs(B)))} \ar[r]^{\Phi}  & {\rc(\lh(\rs(L)))}  &\\
 {\p(\lh(B))} \ar[rrr]^{\Phi} \ar[ur] ^{\rs}& & &{\rc(\lh(L))} \ar[ul]_{\widetilde{\gamma}}
}
\end{equation*}
The front face commutes by already proved (1),
the left face commutes by definition of $\lh$ and $\rs$,
and the bottom face and the back face commute by induction hypothesis.
There are two $\delta$ which we distinguish by denoting them by $\delta_1$ and $\delta_2$.
By diagram chasing we see that
$(\delta_2\circ\Phi\circ\rs)(\p(B)) = (\widetilde{\gamma}\circ\delta_1\circ\Phi)(\p(B))$.
Then by definition of $\widetilde{\gamma}$ we see that the algorithm for
$\delta_1$ and $\delta_2$ choose the same strings
and the only difference is the fact that, during the process of $\delta_2$,
the riggings of the strings $\{(l, x) \in \bigl(\nu^{(r)}, J^{(r)} \bigr) \mid l < s\}$ are greater by 1
compared with the case of $\delta_1$.
Since the riggings of the elements of $\rc(L)$ are nonnegative,
we see that the riggings for the strings of $\nu^{(r)}$ of the corresponding elements of $\rc(\rs(L))$
which are shorter than $s$ are strictly positive.
Therefore we have $(\Phi\circ\rs)(\p(B))\subseteq\widetilde{\gamma}(\rc(L))$ in this case.

For $B=B^{r',1}\otimes B'\otimes B^{r,s}$ when $r'=n-1$ or $n$, the same arguments as above
go through with $\delta$ and $\lh$ replaced by $\delta_s$ and $\lh_s$, respectively.  

When $B=B^{r',1}\otimes B'\otimes B^{r,s}$ for $2\le r'\le n-2$ consider the following diagram:
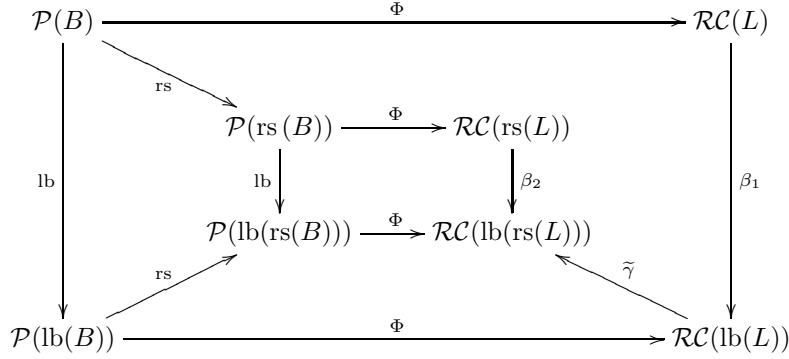
\begin{equation*}
\xymatrix{
 {\p(B)} \ar[rrr]^{\Phi} \ar[ddd]_{\lb} \ar[dr]_{\rs}& & &  {\rc(L)} \ar[ddd]^{\beta_1}\\
 {\hspace{17mm}} & {\p(\rs{(B)})} \ar[r]^{\Phi} \ar[d]_{\lb} & {\rc(\rs(L))} \ar[d] ^{\beta_2}& \\
 & {\p(\lb(\rs(B)))} \ar[r]^{\Phi}  & {\rc(\lb(\rs(L)))}  &\\
 {\p(\lb(B))} \ar[rrr]^{\Phi} \ar[ur] ^{\rs}& & &{\rc(\lb(L))} \ar[ul]_{\widetilde{\gamma}}
}
\end{equation*}
The front face commutes by already proved (2),
the left face commutes by definition of $\lb$ and $\rs$,
and the bottom face and the back face commute by induction hypothesis.
In this case we have
$(\beta_2\circ\Phi\circ\rs)(\p(B)) = (\widetilde{\gamma}\circ\beta_1\circ\Phi)(\p(B))$.
Recall that $\beta$ does not change riggings of untouched strings.
Therefore, by definition of $\widetilde{\gamma}$ we see that the riggings for strings of $\nu^{(r)}$
of the elements of $\rc(\rs(L))$, which are shorter than $s$ are greater by 1
compared with the corresponding elements of $\rc(L)$.
Therefore we have $(\Phi\circ\rs)(\p(B))\subseteq\widetilde{\gamma}(\rc(L))$ in this case.

Finally let us consider the case $B=B^{r',s'}\otimes B'\otimes B^{r,s}$.
In this case, we have the following diagram:
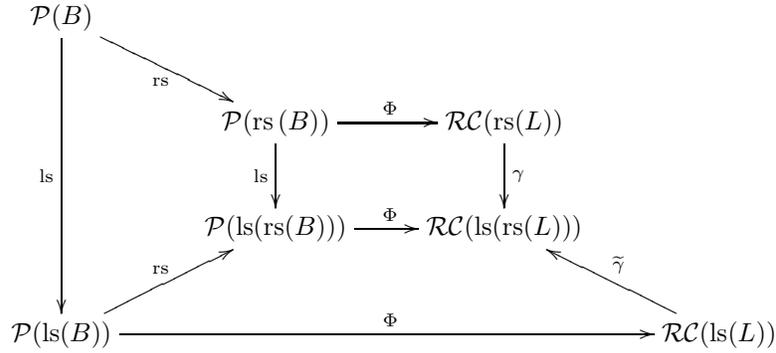
\begin{equation*}
\xymatrix{
 {\p(B)} \ar[ddd]_{\ls} \ar[dr]_{\rs}& & & \\
 {\hspace{17mm}} & {\p(\rs{(B)})} \ar[r]^{\Phi} \ar[d]_{\ls} & {\rc(\rs(L))} \ar[d]^{\gamma}& \\
 & {\p(\ls(\rs(B)))} \ar[r]^{\Phi}  & {\rc(\ls(\rs(L)))}  &\\
 {\p(\ls(B))} \ar[rrr]^{\Phi} \ar[ur] ^{\rs}& & &{\rc(\ls(L))} \ar[ul]_{\widetilde{\gamma}}
}
\end{equation*}
The left face commutes by definition of $\ls$ and $\rs$,
and the bottom face and the back face commute by induction hypothesis.
In this case we have
$(\gamma\circ\Phi\circ\rs)(\p(B)) = (\widetilde{\gamma}\circ\Phi\circ\ls)(\p(B))$.
Since $\gamma$ does not change any strings, we see that riggings for the
strings of $\nu^{(r)}$ of the elements of $\rc(\rs(L))$ which are shorter than
$s$ have strictly positive riggings.
Therefore we have $(\Phi\circ\rs)(\p(B))\subseteq\widetilde{\gamma}(\rc(L))$ in this case.
This completes the proof of (3).

We move to the proof of (4). We again divide into cases.

When $B=B^{1,1}$, this follows from a special case of \cite{S:2005}.

When $B=B^{1,1}\otimes B'\otimes B^{1,1}$ consider the following diagram:
\begin{equation}
\label{equation.rh tilde delta}
\xymatrix{
 {\p(B)} \ar[rrr]^{\Phi} \ar[ddd]_{\rh} \ar[dr]_{\lh}& & &  {\rc(L)} \ar[ddd]^{\widetilde{\delta}} \ar[dl]^{\delta}\\
 {\hspace{17mm}} & {\p(\lh{(B)})} \ar[r]^{\Phi} \ar[d]_{\rh} & {\rc(\lh(L))} \ar[d] ^{\widetilde{\delta}}& \\
 & {\p(\rh(\lh(B)))} \ar[r]^{\Phi}  & {\rc(\rh(\lh(L)))}  &\\
 {\p(\rh(B))} \ar[rrr]^{\Phi} \ar[ur] ^{\lh}& & &{\rc(\rh(L))} \ar[ul]_{\delta}
}
\end{equation}
The top face commutes by (1).
The left face commutes by definition of $\lh$ and $\rh$, and the right face
commutes by the fundamental relation $[\delta,\widetilde{\delta}]=0$ (Theorem \ref{th:delta_commute}).
The back and bottom faces commute by induction hypothesis.
Since $\delta$ is injective, the front face commutes.

When $B=B^{r',1}\otimes B'\otimes B^{1,1}$ for $r'=n-1$ or $n$, we have the same commutative
diagram as in~\eqref{equation.rh tilde delta} with $\lh$ and $\delta$ replaced by
$\lh_s$ and $\delta_s$, respectively. The commutativity 
$[\delta_s, \widetilde{\delta}]=0$ was shown in 
Proposition~\ref{prop:delta_s_commute} (i).

When $B=B^{r',1}\otimes B'\otimes B^{1,1}$ for $2\le r'\le n-2$
consider the following diagram and similarly show the front face commutes, using 
Proposition~\ref{prop:comm2} (ii):
\begin{equation*}
\xymatrix{
 {\p(B)} \ar[rrr]^{\Phi} \ar[ddd]_{\rh} \ar[dr]_{\lb}& & &  {\rc(L)} \ar[ddd]^{\widetilde{\delta}} \ar[dl]^{\beta}\\
 {\hspace{17mm}} & {\p(\lb{(B)})} \ar[r]^{\Phi} \ar[d]_{\rh} & {\rc(\lb(L))} \ar[d] ^{\widetilde{\delta}}& \\
 & {\p(\rh(\lb(B)))} \ar[r]^{\Phi}  & {\rc(\rh(\lb(L)))}  &\\
 {\p(\rh(B))} \ar[rrr]^{\Phi} \ar[ur] ^{\lb}& & &{\rc(\rh(L))} \ar[ul]_{\beta}
}
\end{equation*}

When $B=B^{r',s'}\otimes B'\otimes B^{1,1}$
consider the following diagram and similarly show the front face commutes, using 
Proposition~\ref{prop:comm2} (iii):
\begin{equation*}
\xymatrix{
 {\p(B)} \ar[rrr]^{\Phi} \ar[ddd]_{\rh} \ar[dr]_{\ls}& & &  {\rc(L)} \ar[ddd]^{\widetilde{\delta}} \ar[dl]^{\gamma}\\
 {\hspace{17mm}} & {\p(\ls{(B)})} \ar[r]^{\Phi} \ar[d]_{\rh} & {\rc(\ls(L))} \ar[d] ^{\widetilde{\delta}}& \\
 & {\p(\rh(\ls(B)))} \ar[r]^{\Phi}  & {\rc(\rh(\ls(L)))}  &\\
 {\p(\rh(B))} \ar[rrr]^{\Phi} \ar[ur] ^{\ls}& & &{\rc(\rh(L))} \ar[ul]_{\gamma}
}
\end{equation*}

The proofs of (5) and (6) are parallel to the argument of (4).
When $B=B^{r,1}$,
(5) follows from \cite{S:2005}. When $B=B^{r,s}$, (6) follows from 
Proposition~\ref{prop:rs vs gamma-tilde} for $r\le n-2$, and from the fact that both
$\p(B^{r,s})$ and $\rc(L)$ consist of a single element for $r=n-1,n$. We also need
Propositions~\ref{prop:comm2} and~\ref{prop:delta_s_commute} (ii).

Case (4') follows from cases (4) and (5) by noting that $\delta_s$ is realized by a composition
of the doubling map, a sequence of $n$ compositions of $\delta \circ \beta$, and a halving map. 

Finally, we prove (7).
We again divide into cases.

When $B=B^{1,1}\otimes B'$ we can use the following diagram
to show that the front face commutes:
\begin{equation*}
\xymatrix{
 {\p(B)} \ar[rrr]^{\Phi} \ar[ddd]_{\diamond} \ar[dr]_{\lh}& & &  {\rc(L)} \ar[ddd]^{\theta} \ar[dl]^{\delta}\\
 & {\p(\lh{(B)})} \ar[r]^{\Phi} \ar[d]_{\diamond} & {\rc(\lh(L))} \ar[d] ^{\theta}&  {\hspace{10mm}}\\
 & {\p(\lh(B)^{\lusz})} \ar[r]^{\Phi}  & {\rc(\lh(L))}  &\\
 {\p(B^{\lusz})} \ar[rrr]^{\Phi} \ar[ur] ^{\rh}& & &{\rc(L)} \ar[ul]_{\widetilde{\delta}}
}
\end{equation*}

The proofs when $B=B^{r,1}\otimes B'$ or $B=B^{r,s}\otimes B'$
are almost the same.
\end{proof}

We can now state the main result of this section.

\begin{theorem}\label{th:welldef_highest}
Let $B$ be a tensor product of KR crystals.
Then $\Phi \colon \p(B,\lambda) \to \rc (L,\lambda)$ is a well-defined bijection.
\end{theorem}

\begin{proof}
By Proposition~\ref{prop:Phi_welldef1}, $\Phi$ is a well-defined injection.
By~\cite{kedem.difrancesco.2008} or~\cite{naoi.2012}
we have $\lvert \p(B,\lambda) \rvert = \lvert \rc (L,\lambda) \rvert$, which proves that $\Phi$ is a bijection.
\end{proof}

In order to generalize the rigged configuration bijection to include
non-highest weight elements, we invoke the following result stated
in~\cite[Thm.~4.1]{Sak:2013}.

\begin{theorem}[{\cite{Sak:2013}}]
\label{th:welldef}
The rigged configuration bijection $\Phi \colon \p(B) \to \rc (L)$ can be extended 
to a bijection $\Phi \colon B \to \RC(L)$ by requiring $\Phi$ to be a classical crystal isomorphism:
\begin{equation}
\label{equation.Phi commute f}
	[\Phi,e_i]=[\Phi,f_i]=0\qquad (i\in I_0).
\end{equation}
\end{theorem}

\begin{example}
Consider the rigged configuration $(\nu, J)$ in Example~\ref{ex:rigged_config}, $f_4(\nu, J)$ from Example~\ref{ex:rc_kashiwara}, and the tensor product of KR crystal elements in Example~\ref{ex:rc_to_crystal}. We have
\[
\Phi^{-1}\bigl(f_4(\nu,J)\bigr) =
\Yvcentermath1
\young(13,4,\mfive)\otimes\young(1,3,5)\otimes\young(12,2\mone)
\otimes\young(11)\otimes\young(1) = f_4 \Phi^{-1}(\nu, J).
\]
\end{example}

\section{$R$-invariance of rigged configurations}\label{sec:Rinv}
In this section we prove that the combinatorial $R$-matrix on unrestricted rigged configurations is the identity
under the bijection $\Phi$. We do so in several steps. First we show that the combinatorial $R$-matrix
$R \colon B^{n,1} \otimes B^{r,s} \to B^{r,s} \otimes B^{n,1}$ corresponds to the identity map under $\Phi$
(see Proposition~\ref{prop:spin_Rinv} which follows from Lemmas~\ref{lem:Rinv_spin1} and~\ref{lem:Rinv_spin2}
and Proposition~\ref{prop:mohamad}). The general statement (see Theorem~\ref{th:main}) can then be deduced
by passing a spin column through enough times to reduce to the type $A_n^{(1)}$ case which was proven 
in~\cite{KSS:2002}.

We begin with the $R$-matrix $R \colon B^{n,1}\otimes B^{r,s} \to B^{r,s} \otimes B^{n,1}$ for $1\le r\le n-2$ 
in type $D_n^{(1)}$. Let $b \otimes u_\eta \in B^{n,1}\otimes B^{r,s}$ be an $I_0$-highest weight element,
where $u_\eta \in B(\eta) \subseteq B^{r,s}$ is the unique $I_0$-highest weight element of highest
weight $\eta$. Note that, for fixed $\eta$, the element $b\otimes u_\eta$ is uniquely specified by the weight 
$\lambda = \wt(b\otimes u_\eta)$ since the multiplicity of any weight space in 
$B(\clfw_n) \iso B^{n,1}$ (as classical crystals) is at most 1.

Let us introduce the following notation. Write $\lambda = \sum_{i\in I_0} m_i \clfw_i$
in terms of the classical fundamental weights $\clfw_i$, where all $m_i$ are nonnegative integers
since $\lambda$ is dominant. Define $\overline{\lambda} = \sum_{i \in I_0 \setminus \{n-1,n\}} m_i \clfw_i$.
We can interpret $\overline{\lambda}$ as a partition (where each fundamental weight
$\clfw_i$ for $1\le i\le n-2$ contributes a column of height $i$).
Then let $\overline{\lambda}^c$ be the complement of the partition $\overline{\lambda}$
in the $r\times s$ rectangle. Similarly, $\eta^c$ is the complement of $\eta$ (interpreted as a partition)
in the $r\times s$ rectangle. The skew partition $\overline{\lambda}^c / \eta^c$ can have at most one box
in each row. Denote the cells in $\overline{\lambda}^c / \eta^c$ from top to bottom (in English convention for partitions)
by $c_1,c_2,\ldots,c_\ell$. Finally, $\mathring{\eta}$ is obtained by replacing all 
$\Yboxdim4pt\yng(1,1)$ of $\eta^c$ by $\Yboxdim4pt\yng(1)$ (recall that this is well-defined since $\eta$
is obtained from the $r\times s$ rectangle by removing vertical dominoes).

Note that $m_{n-1}+m_n=1$ due to the fact that $r\le n-2$ and $b \in B^{n,1}$. Also, $\ell+m_{n-1}$ is even
since $b\in B^{n,1}$ (it would be odd for $b\in B^{n-1,1}$).

We now define the configuration that we will show under the bijection $\Phi$ corresponds to the $I_0$-highest weight
elements $b \otimes u_\eta \in B^{n,1} \otimes B^{r,s}$ for $1\le r\le n-2$. Recall that $b\otimes u_\eta$ is uniquely 
determined by the two weights $\eta$ and $\lambda$. Let $\nu=\Gamma(\eta, \lambda)$ be the following configuration:
\begin{enumerate}
\item
For $r\leq a\le n-2$, $\nu^{(a)}=\overline{\lambda}^c$.
\item
For $1\le a<r$, $\nu^{(a)}$ is obtained from $\nu^{(n-2)}$ by removing $r-a$ rows
starting from longer rows.
\item
If $m_{n-1}=0$, then $\nu^{(n-1)}$ (resp. $\nu^{(n)}$) is obtained from $\mathring{\eta}$ by adding the cells 
$c_1, c_3,\ldots$ (resp. $c_2, c_4,\ldots$) to the same row length as in $\overline{\lambda}^c / \eta^c$.\newline
If $m_{n-1}=1$, then $\nu^{(n-1)}$ (resp. $\nu^{(n)}$) is obtained from $\mathring{\eta}$ by adding the cells 
$c_2, c_4,\ldots$ (resp. $c_1, c_3,\ldots$) to the same row length as in $\overline{\lambda}^c / \eta^c$.
\end{enumerate}

Let us examine the vacancy numbers for elements $\nu \in \Gamma(\eta,\lambda)$. It is not hard to check
that $P_i^{(a)}(\nu)=0$ for all $1\le a \le n-2$ and $i\ge 0$. Furthermore, from weight considerations we have 
$P_i^{(n-1)}(\nu)=m_{n-1}$ and $P_i^{(n)}(\nu)=1-m_{n-1}$ for large $i$. Call $1\le h_1\le h_2 \le \cdots \le h_\ell$
the ordered row labels of the cells in $\nu^{(n-1)}/\mathring{\eta}$ and $\nu^{(n)}/\mathring{\eta}$.
Starting with the largest row in $\nu^{(n-1)}$ (resp. $\nu^{(n)}$), the vacancy number switches from 1 to 0
(or 0 to 1) whenever one of the row lengths $h_i$ is crossed. In particular,
$P_i^{(n-1)}(\nu)+P_i^{(n)}(\nu)=1$ for all $i\ge 0$.

\def\lr#1{\multicolumn{1}{|@{\hspace{.6ex}}c@{\hspace{.6ex}}|}{\raisebox{-.3ex}{{\color{blue}$\mathbf #1$}}}}
\begin{example}
\label{ex:minimal_sp_nonsp}
We illustrate the construction of $\Gamma(\eta,\lambda)$ with two examples.
Consider type $D_{10}^{(1)}$ and $B = B^{10,1} \otimes B^{8,5}$. Let
$b=(+,-,+,+,+,-,+,+,+,+)$ and
\[
u_{\eta} =
\begin{tikzpicture}[baseline]
\matrix [tab] 
 {
 	\node[draw]{1}; & 
 	\node[draw]{1}; & 
 	\node[draw]{1}; & 
 	\node[draw]{1}; & 
 	\node[draw]{1}; \\
 	\node[draw]{2}; & 
 	\node[draw]{2}; & 
 	\node[draw]{2}; & 
 	\node[draw]{2}; & 
 	\node[draw]{2}; \\
 	\node[draw]{3}; & 
 	\node[draw]{3}; & 
 	\node[draw]{3}; & 
 	\node[draw]{3}; \\
 	\node[draw]{4}; & 
 	\node[draw]{4}; & 
 	\node[draw]{4}; & 
 	\node[draw]{4}; \\
	\node[draw]{5}; & 
 	\node[draw]{5}; & 
 	\node[draw]{5}; & 
 	\node[draw]{5}; \\
 	\node[draw]{6}; & 
 	\node[draw]{6}; & 
 	\node[draw]{6}; & 
 	\node[draw]{6}; \\
 	\node[draw]{7}; & 
 	\node[draw]{7}; & 
 	\node[draw]{7}; \\
 	\node[draw]{8}; & 
 	\node[draw]{8}; & 
 	\node[draw]{8}; \\
 };
\end{tikzpicture}\, ,
\]
where $\eta = \clfw_2 + \clfw_6 + 3 \clfw_8$ and $\lambda = \clfw_1 + \clfw_5 + 3 \clfw_8 + \clfw_{10}$, and hence
$m_{10}=1$ and $m_9=0$. Thus we have
as partitions $\eta^c = 221111$ and $\overline{\lambda}^c = 2221111$, and so
\[
\overline{\lambda}^c / \eta^c =
\begin{tikzpicture}[baseline]
\matrix [tab] 
 {
 	\node[draw,fill=gray!40]{}; & 
	\node[draw,fill=gray!40]{}; \\
	\node[draw,fill=gray!40]{}; & 
	\node[draw,fill=gray!40]{}; \\
	\node[draw,fill=gray!40]{}; & 
	\node[draw]{{\color{blue}\mathbf 1}}; \\
  	\node[draw,fill=gray!40]{}; \\
	\node[draw,fill=gray!40]{}; \\
	\node[draw,fill=gray!40]{}; \\
	\node[draw]{{\color{blue}\mathbf 2}}; \\
 };
\end{tikzpicture}\; ,
\]
where we have indicated the cells $c_1$ and $c_2$ in the construction by blue letters $1$ and $2$, respectively.
Then $\Phi(b \otimes u_{\eta})$ is equal to
\[
\scalebox{0.6}{$
\raisebox{-10pt}{$\emptyset$} 
\quad
 {
\begin{array}[t]{r|c|l}
 \cline{2-2} 0 &\phantom{|}& 0 \\
 \cline{2-2} 
\end{array}
} 
\quad
 {
\begin{array}[t]{r|c|l}
 \cline{2-2} 0 &\phantom{|}& 0 \\
 \cline{2-2}  &\phantom{|}& 0 \\
 \cline{2-2} 
\end{array}
} 
\quad
 {
\begin{array}[t]{r|c|l}
 \cline{2-2} 0 &\phantom{|}& 0 \\
 \cline{2-2}  &\phantom{|}& 0 \\
 \cline{2-2}  &\phantom{|}& 0 \\
 \cline{2-2} 
\end{array}
} 
\quad
 {
\begin{array}[t]{r|c|l}
 \cline{2-2} 0 &\phantom{|}& 0 \\
 \cline{2-2}  &\phantom{|}& 0 \\
 \cline{2-2}  &\phantom{|}& 0 \\
 \cline{2-2}  &\phantom{|}& 0 \\
 \cline{2-2} 
\end{array}
} 
\quad
 {
\begin{array}[t]{r|c|c|l}
 \cline{2-3} 0 &\phantom{|}&\phantom{|}& 0 \\
 \cline{2-3} 0 &\phantom{|}& \multicolumn{2 }{l}{ 0 } \\
 \cline{2-2}  &\phantom{|}& \multicolumn{2 }{l}{ 0 } \\
 \cline{2-2}  &\phantom{|}& \multicolumn{2 }{l}{ 0 } \\
 \cline{2-2}  &\phantom{|}& \multicolumn{2 }{l}{ 0 } \\
 \cline{2-2} 
\end{array}
} 
\quad
 {
\begin{array}[t]{r|c|c|l}
 \cline{2-3} 0 &\phantom{|}&\phantom{|}& 0 \\
 \cline{2-3}  &\phantom{|}&\phantom{|}& 0 \\
 \cline{2-3} 0 &\phantom{|}& \multicolumn{2 }{l}{ 0 } \\
 \cline{2-2}  &\phantom{|}& \multicolumn{2 }{l}{ 0 } \\
 \cline{2-2}  &\phantom{|}& \multicolumn{2 }{l}{ 0 } \\
 \cline{2-2}  &\phantom{|}& \multicolumn{2 }{l}{ 0 } \\
 \cline{2-2} 
\end{array}
} 
\quad
 {
\begin{array}[t]{r|c|c|l}
 \cline{2-3} 0 &\phantom{|}&\phantom{|}& 0 \\
 \cline{2-3}  &\phantom{|}&\phantom{|}& 0 \\
 \cline{2-3}  &\phantom{|}&\lr{1}& 0 \\
 \cline{2-3} 0 &\phantom{|}& \multicolumn{2 }{l}{ 0 } \\
 \cline{2-2}  &\phantom{|}& \multicolumn{2 }{l}{ 0 } \\
 \cline{2-2}  &\phantom{|}& \multicolumn{2 }{l}{ 0 } \\
 \cline{2-2}  &\lr{2}& \multicolumn{2 }{l}{ 0 } \\
 \cline{2-2} 
\end{array}
} 
\quad
 {
\begin{array}[t]{r|c|c|l}
 \cline{2-3} 0 &\phantom{|}&\phantom{|}& 0 \\
 \cline{2-3}  &\phantom{|}&\lr{1}& 0 \\
 \cline{2-3} 1 &\phantom{|}& \multicolumn{2 }{l}{ 0 } \\
 \cline{2-2} 
\end{array}
} 
\quad
 {
\begin{array}[t]{r|c|c|l}
 \cline{2-3} 1 &\phantom{|}&\phantom{|}& 0 \\
 \cline{2-3} 0 &\phantom{|}& \multicolumn{2 }{l}{ 0 } \\
 \cline{2-2}  &\phantom{|}& \multicolumn{2 }{l}{ 0 } \\
 \cline{2-2}  &\lr{2}& \multicolumn{2 }{l}{ 0 } \\
 \cline{2-2} 
\end{array}
}
$}\raisebox{-15pt}{,}
\]
where again we labelled the cells added to $\eta^c$ in $\nu^{(n-2)}$ and
$\mathring{\eta}$ in $\nu^{(n-1)}$ and $\nu^{(n)}$ by blue letters $1$ and $2$.

Next we consider $b' = (+,-,+,+,+,-,-,-,+,+)$ and 
$\mu := \wt(b' \otimes u_\eta)= \clfw_1 + \clfw_5 + \clfw_6 + 2\clfw_8 + \clfw_{10}$.
Note that $\lambda \neq \mu$ and
\[
\overline{\mu}^c / \eta^c =
\begin{tikzpicture}[baseline]
\matrix [tab] 
 {
 	\node[draw,fill=gray!40]{}; & 
 	\node[draw,fill=gray!40]{}; & 
	\node[draw]{{\color{blue}\mathbf 1}}; \\
	\node[draw,fill=gray!40]{}; & 
	\node[draw,fill=gray!40]{}; & 
	\node[draw]{{\color{blue}\mathbf 2}}; \\
	\node[draw,fill=gray!40]{}; & 
	\node[draw]{{\color{blue}\mathbf 3}}; \\
  	\node[draw,fill=gray!40]{}; \\
	\node[draw,fill=gray!40]{}; \\
	\node[draw,fill=gray!40]{}; \\
	\node[draw]{{\color{blue}\mathbf 4}}; \\
 };
\end{tikzpicture}\; .
\]
Then $\Phi(b' \otimes u_{\eta})$ is
\[
\scalebox{0.6}{$
\raisebox{-10pt}{$\emptyset$} 
\quad
 {
\begin{array}[t]{r|c|l}
 \cline{2-2} 0 &\phantom{|}& 0 \\
 \cline{2-2} 
\end{array}
} 
\quad
 {
\begin{array}[t]{r|c|l}
 \cline{2-2} 0 &\phantom{|}& 0 \\
 \cline{2-2}  &\phantom{|}& 0 \\
 \cline{2-2} 
\end{array}
} 
\quad
 {
\begin{array}[t]{r|c|l}
 \cline{2-2} 0 &\phantom{|}& 0 \\
 \cline{2-2}  &\phantom{|}& 0 \\
 \cline{2-2}  &\phantom{|}& 0 \\
 \cline{2-2} 
\end{array}
} 
\quad
 {
\begin{array}[t]{r|c|l}
 \cline{2-2} 0 &\phantom{|}& 0 \\
 \cline{2-2}  &\phantom{|}& 0 \\
 \cline{2-2}  &\phantom{|}& 0 \\
 \cline{2-2}  &\phantom{|}& 0 \\
 \cline{2-2} 
\end{array}
} 
\quad
 {
\begin{array}[t]{r|c|c|l}
 \cline{2-3} 0 &\phantom{|}&\phantom{|}& 0 \\
 \cline{2-3} 0 &\phantom{|}& \multicolumn{2 }{l}{ 0 } \\
 \cline{2-2}  &\phantom{|}& \multicolumn{2 }{l}{ 0 } \\
 \cline{2-2}  &\phantom{|}& \multicolumn{2 }{l}{ 0 } \\
 \cline{2-2}  &\phantom{|}& \multicolumn{2 }{l}{ 0 } \\
 \cline{2-2} 
\end{array}
} 
\quad
 {
\begin{array}[t]{r|c|c|c|l}
 \cline{2-4} 0 &\phantom{|}&\phantom{|}&\phantom{|}& 0 \\
 \cline{2-4} 0 &\phantom{|}&\phantom{|}& \multicolumn{2 }{l}{ 0 } \\
 \cline{2-3} 0 &\phantom{|}& \multicolumn{3 }{l}{ 0 } \\
 \cline{2-2}  &\phantom{|}& \multicolumn{3 }{l}{ 0 } \\
 \cline{2-2}  &\phantom{|}& \multicolumn{3 }{l}{ 0 } \\
 \cline{2-2}  &\phantom{|}& \multicolumn{3 }{l}{ 0 } \\
 \cline{2-2} 
\end{array}
} 
\quad
 {
\begin{array}[t]{r|c|c|c|l}
 \cline{2-4} 0 &\phantom{|}&\phantom{|}&\lr{1}& 0 \\
 \cline{2-4}  &\phantom{|}&\phantom{|}&\lr{2}& 0 \\
 \cline{2-4} 0 &\phantom{|}&\lr{3}& \multicolumn{2 }{l}{ 0 } \\
 \cline{2-3} 0 &\phantom{|}& \multicolumn{3 }{l}{ 0 } \\
 \cline{2-2}  &\phantom{|}& \multicolumn{3 }{l}{ 0 } \\
 \cline{2-2}  &\phantom{|}& \multicolumn{3 }{l}{ 0 } \\
 \cline{2-2}  &\lr{4}& \multicolumn{3 }{l}{ 0 } \\
 \cline{2-2} 
\end{array}
} 
\quad
 {
\begin{array}[t]{r|c|c|c|l}
 \cline{2-4} 0 &\phantom{|}&\phantom{|}&\lr{1}& 0 \\
 \cline{2-4} 0 &\phantom{|}&\lr{3}& \multicolumn{2 }{l}{ 0 } \\
 \cline{2-3} 1 &\phantom{|}& \multicolumn{3 }{l}{ 0 } \\
 \cline{2-2} 
\end{array}
} 
\quad
 {
\begin{array}[t]{r|c|c|c|l}
 \cline{2-4} 1 &\phantom{|}&\phantom{|}&\lr{2}& 1 \\
 \cline{2-4} 0 &\phantom{|}& \multicolumn{3 }{l}{ 0 } \\
 \cline{2-2}  &\phantom{|}& \multicolumn{3 }{l}{ 0 } \\
 \cline{2-2}  &\lr{4}& \multicolumn{3 }{l}{ 0 } \\
 \cline{2-2} 
\end{array}
}$}\raisebox{-15pt}{.}
\]
\end{example}

\begin{lemma}\label{lem:Rinv_spin1}
Consider the $I_0$-highest weight element $b \otimes u_\eta \in B^{n,1}\otimes B^{r,s}$
of weight $\lambda$ with $1\le r\leq n-2$, where $u_\eta$ is the $I_0$-highest weight element of highest weight $\eta$.
Then the rigged configuration 
\[
	(\nu,J)=\Phi(b\otimes u_\eta)
\]
satisfies $\nu = \Gamma(\eta,\lambda)$ and $J^{(a)}_i=0$ for all $1\le a \le n-2$ and $i\ge 0$.
$J_i^{(n-1)}$ and $J_i^{(n)}$ are determined as follows. Since $P_i^{(n-1)}(\nu)+P_i^{(n)}(\nu)=1$
either $J_i^{(n-1)}$ or $J_i^{(n)}$ can contain 1's and the other riggings must be 0. The number of 1's
in $J_i^{(n-1)}$ (or $J_i^{(n)}$) is equal to the number of vertical dominoes in
column $i$ of $\overline{\lambda}^c/ \eta^c$.
\end{lemma}

\begin{proof}
Let us define $d = \emb(b)$. Let $x$ be the letter in $d$ at height $h$. Define $\delta' = \beta^{-1} \circ \delta^{-1}$ 
unless $h = 1$ in which case $\delta' = \delta^{-1}$. Thus $\delta'$ under the bijection $\Phi$ corresponds
to adding the letter $x$ to the step in the path using the algorithm for $\delta^{-1}$, except we terminate after 
adding the last box to $(\nu^{(h)}, J^{(h)})$. Let $w_i$ denote the row length of $\eta^c$ corresponding to the row
in $\overline{\la}^c/\eta^c$ containing the cell $c_i$ as given in the construction $\Gamma(\eta, \lambda)$ 
(recall $\ell = \lvert \overline{\lambda}^c/\eta^c \rvert$). Set $w_0 = \infty$ and $w_i = 0$ for all $i > \ell$. Define $\widetilde{w}_i = w_{\ell - i}$. In particular, we have $w_i \leq w_{i+1}$ for all 
$0 \leq i < \ell$. Define $\mu(x) = \#\{x \leq x' \leq n-2 \mid \overline{x}' \in d\}$.

First note that $\Phi(u_{\eta})$ is given by Proposition~\ref{prop:highestRC}. Next we want to add $b$, 
which by the definition of $\delta_s$ involves the doubling map $\emb$. Hence we need to double $\Phi(u_{\eta})$.
Denote the rigged configuration before having added the letter $x$ by $(\overline{\nu}, \overline{J})$ and $(\overline{\nu}',\overline{J}') =\delta'(\overline{\nu}, \overline{J})$.
Suppose we are adding the letter $x \leq n-1$ to $d$ at height $h$.

\noindent {\bf Claim:}
We obtain $\delta'(\overline{\nu}, \overline{J})$ from $(\overline{\nu},\overline{J})$ by adding
a box to a (singular) string of length $2 \widetilde{w}_{a-h}$ to $\overline{\nu}^{(a)}$ for $h \leq a < x$.
Note that this is the $(a - h+1)$-th cell from the bottom of $\overline{\lambda}^c / \eta^c$. 
Moreover, all riggings are $0$.

We prove the claim by induction on $h$. Suppose the letter $y$ was added at height $h-1$ in the
previous step (which might be vacuous if $x$ is at height $1$ in which case we set $y=0$).
Note that $P_i^{(a)}(\overline{\nu}) > 0$ for all $i$ such that $2 \widetilde{w}_{a-h} < i \le 2 \widetilde{w}_{a-h+1} $ and $h-1 \leq a \leq y<x$.

Observe that the number of minus signs appearing in $b$ before position $x$ is $x-h$. 
The application of $\delta'$ adds a box to the row of length $2\widetilde{w}_{x-h-1}$ in $\overline{\nu}^{(x-1)}$. This 
follows from the fact that no previous application of $\delta'$ changed $\overline{\nu}^{(x-1)}$ (our induction hypothesis),
the description of $\overline{\nu}$ as given in Proposition~\ref{prop:highestRC}, and that all rows strictly longer than 
$2 \widetilde{w}_{x-h-1}$ are non-singular. Next $\delta'$ adds a box to a row of length $2\widetilde{w}_{x-h-2}$ in $\overline{\nu}^{(x-2)}$ 
since all rows of length strictly between $2\widetilde{w}_{x-h-2}$ and $2\widetilde{w}_{x-h-1}$ are nonsingular, $\overline{\nu}^{(x-2)}$
is the same as $\overline{\nu}^{(x-1)}$ with the top row removed, and we must select a row $i$ 
such that $\overline{\nu}_i^{(x-2)} \leq 2 \widetilde{w}_{x-h-1}$ by the definition of $\delta'$. A similar argument holds for 
$\overline{\nu}^{(a)}$ for the remaining $h \leq a < x-1$ for the application of $\delta'$. From~\eqref{eq:def_vacancy}, the resulting riggings are all $0$, and 
$P_i^{(a)}(\overline{\nu}') > 0$ for all $i$ such that $2 \widetilde{w}_{a-h-1} < i \le 2 \widetilde{w}_{a-h} $ and $h \leq a \leq x$.

For $x = n, \mn$, the application of $\delta'$ is similar to the above except we do not change 
$\overline{\nu}^{(n)}, \overline{\nu}^{(n-1)}$ respectively. Next suppose we are adding $\overline{x}$ at height $h$ with $1\le x\le n-1$.

\noindent {\bf Claim:}
We obtain $\delta'(\overline{\nu}, \overline{J})$ from $(\overline{\nu},\overline{J})$ by the following steps:
\begin{enumerate}
\item\label{claim:step1} add a box to the largest odd length string in $\overline{\nu}^{(a)}$ for $x \leq a \leq n-2$,
\item\label{claim:step2} for both $\overline{\nu}^{(n-1)}, \overline{\nu}^{(n)}$, either add to the largest odd length row or the row of 
length $2 w_{\mu(x)}$ depending on if there is an odd length row or not, respectively, and
\item\label{claim:step3} add a box to a row of length $2 \widetilde{w}_{a - h}$ in $\overline{\nu}^{(a)}$ for $h \leq a < n-1$.
\end{enumerate}

We prove the claim by using induction on $\overline{x}$ (i.e., as $x$ decreases).
Note that the largest singular row in $\overline{\nu}^{(x)}$ has length $2 w_{\mu(x)} + 1$ and it is the first row of $\overline{\nu}^{(x)}$.

By our induction assumption and the definition of $\delta'$, we add a box to a row of length $2 w_{\mu(x)} + 1$ in all $\overline{\nu}^{(a)}$ for $x \leq a \leq n-2$. Thus step~(\ref{claim:step1}) follows.
For step~(\ref{claim:step2}), let $a = n-1,n$. Note that if there exists an odd length row, then from our induction, there is exactly one such row and it is of length $2 w_{\mu(x)} + 1$. Therefore we add a box to this row. Otherwise all rows are currently of even length, so we add a box to a row of length $2 w_{\mu(x)}$. Note that $P_{2w_{\mu(x)}}^{(a)}(\overline{\nu}) = 0$ because we have only added boxes to $\overline{\nu}^{(n-2)}$ in rows of length at least $2w_{\mu(x)}$ (i.e., the resulting rows are strictly larger and don't contribute to $P_{2w_{\mu(x)}}^{(a)}(\overline{\nu})$).
The proof of step~(\ref{claim:step3}) of the claim is similar to the proof of the previous claim for $x \leq n - 2$.

From the aforementioned claims, it immediately follows that we obtain the desired configuration of $\Gamma(\eta, \lambda)$ after applying $\emb^{-1}$.

Let us now turn to the statement about the riggings. A straightforward check from the above claim shows that all riggings are $0$ in $\overline{\nu}^{(a)}$ for $a \leq n - 2$.
Next, each removable domino in $\overline{\lambda}^c/\eta^c$ corresponds to two consecutive minus signs in $b$,
which in turns corresponds to two barred letters $\overline{x}$ and $\overline{x+1}$.
Note that when adding $\overline{x}$, the resulting vacancy number 
$P^{(a)}_{2 w_{\mu(x)}+2}(\overline{\nu}')=2$ for either $a=n-1$ or $a=n$ since the vertical domino implies we have an additional (as compared with $\Phi(u_{\eta})$) contribution of 6 boxes in $\overline{\nu}^{(n-2)}$, but only 2 boxes in $\overline{\nu}^{(a)}$. Hence the corresponding rigging
is 2. From the above description of $\delta'$, we do not change the corresponding row again when we add the 
remaining letters, thus $\emb^{-1}(\overline{\nu},\overline{J}) = \Phi(b \otimes u_\eta)$ is the desired rigged configuration.
\end{proof}

\begin{example}
We consider $B = B^{8,1} \otimes B^{6,5}$ in type $D_8^{(1)}$ and fix $\lambda = \clfw_1 + 2 \clfw_6 + \clfw_7$. 
We now look at the set $\{ b \otimes u_{\eta} \in \p(B) \mid \wt(b \otimes u_{\eta}) = \lambda \}$. We start with 
$\eta = \clfw_2 + 2 \clfw_6$ and $b = (+,-,+,+,+,+,+,-)$. Thus we have
\[
\raisebox{-15pt}{$\Phi(b \otimes u_{\eta}) = $} \hspace{10pt} \scalebox{0.7}{$
\raisebox{-10pt}{$\emptyset$}
\quad
 {
\begin{array}[t]{r|c|l}
 \cline{2-2} 0 &\phantom{|}& 0 \\
 \cline{2-2} 
\end{array}
} 
\quad
 {
\begin{array}[t]{r|c|l}
 \cline{2-2} 0 &\phantom{|}& 0 \\
 \cline{2-2}  &\phantom{|}& 0 \\
 \cline{2-2} 
\end{array}
} 
\quad
 {
\begin{array}[t]{r|c|l}
 \cline{2-2} 0 &\phantom{|}& 0 \\
 \cline{2-2}  &\phantom{|}& 0 \\
 \cline{2-2}  &\phantom{|}& 0 \\
 \cline{2-2} 
\end{array}
} 
\quad
 {
\begin{array}[t]{r|c|l}
 \cline{2-2} 0 &\phantom{|}& 0 \\
 \cline{2-2}  &\phantom{|}& 0 \\
 \cline{2-2}  &\phantom{|}& 0 \\
 \cline{2-2}  &\phantom{|}& 0 \\
 \cline{2-2} 
\end{array}
} 
\quad
 {
\begin{array}[t]{r|c|l}
 \cline{2-2} 0 &\phantom{|}& 0 \\
 \cline{2-2}  &\phantom{|}& 0 \\
 \cline{2-2}  &\phantom{|}& 0 \\
 \cline{2-2}  &\phantom{|}& 0 \\
 \cline{2-2}  &\phantom{|}& 0 \\
 \cline{2-2} 
\end{array}
} 
\quad
 {
\begin{array}[t]{r|c|l}
 \cline{2-2} 1 &\phantom{|}& 0 \\
 \cline{2-2}  &\phantom{|}& 0 \\
 \cline{2-2} 
\end{array}
} 
\quad
 {
\begin{array}[t]{r|c|l}
 \cline{2-2} 0 &\phantom{|}& 0 \\
 \cline{2-2}  &\phantom{|}& 0 \\
 \cline{2-2}  &\phantom{|}& 0 \\
 \cline{2-2} 
\end{array}
}$}\raisebox{-15pt}{.}
\]
Next we consider $\eta' = \clfw_4 + 2 \clfw_6$ and $b' = (+,-,-,-,+,+,+,-)$, and we have
\[
\raisebox{-15pt}{$\Phi(b' \otimes u_{\eta'}) = $} \hspace{10pt} \scalebox{0.7}{$
\raisebox{-10pt}{$\emptyset$}
\quad
 {
\begin{array}[t]{r|c|l}
 \cline{2-2} 0 &\phantom{|}& 0 \\
 \cline{2-2} 
\end{array}
} 
\quad
 {
\begin{array}[t]{r|c|l}
 \cline{2-2} 0 &\phantom{|}& 0 \\
 \cline{2-2}  &\phantom{|}& 0 \\
 \cline{2-2} 
\end{array}
} 
\quad
 {
\begin{array}[t]{r|c|l}
 \cline{2-2} 0 &\phantom{|}& 0 \\
 \cline{2-2}  &\phantom{|}& 0 \\
 \cline{2-2}  &\phantom{|}& 0 \\
 \cline{2-2} 
\end{array}
} 
\quad
 {
\begin{array}[t]{r|c|l}
 \cline{2-2} 0 &\phantom{|}& 0 \\
 \cline{2-2}  &\phantom{|}& 0 \\
 \cline{2-2}  &\phantom{|}& 0 \\
 \cline{2-2}  &\phantom{|}& 0 \\
 \cline{2-2} 
\end{array}
} 
\quad
 {
\begin{array}[t]{r|c|l}
 \cline{2-2} 0 &\phantom{|}& 0 \\
 \cline{2-2}  &\phantom{|}& 0 \\
 \cline{2-2}  &\phantom{|}& 0 \\
 \cline{2-2}  &\phantom{|}& 0 \\
 \cline{2-2}  &\phantom{|}& 0 \\
 \cline{2-2} 
\end{array}
} 
\quad
 {
\begin{array}[t]{r|c|l}
 \cline{2-2} 1 &\phantom{|}& 1 \\
 \cline{2-2}  &\phantom{|}& 0 \\
 \cline{2-2} 
\end{array}
} 
\quad
 {
\begin{array}[t]{r|c|l}
 \cline{2-2} 0 &\phantom{|}& 0 \\
 \cline{2-2}  &\phantom{|}& 0 \\
 \cline{2-2}  &\phantom{|}& 0 \\
 \cline{2-2} 
\end{array}
}$}\raisebox{-15pt}{.}
\]
Finally, we take $\eta'' = 3 \clfw_6$ and $b'' = (+,-,-,-,-,-,+,-)$, so that
\[
\raisebox{-15pt}{$\Phi(b'' \otimes u_{\eta''}) = $} \hspace{10pt} \scalebox{0.7}{$
\raisebox{-10pt}{$\emptyset$}
\quad
 {
\begin{array}[t]{r|c|l}
 \cline{2-2} 0 &\phantom{|}& 0 \\
 \cline{2-2} 
\end{array}
} 
\quad
 {
\begin{array}[t]{r|c|l}
 \cline{2-2} 0 &\phantom{|}& 0 \\
 \cline{2-2}  &\phantom{|}& 0 \\
 \cline{2-2} 
\end{array}
} 
\quad
 {
\begin{array}[t]{r|c|l}
 \cline{2-2} 0 &\phantom{|}& 0 \\
 \cline{2-2}  &\phantom{|}& 0 \\
 \cline{2-2}  &\phantom{|}& 0 \\
 \cline{2-2} 
\end{array}
} 
\quad
 {
\begin{array}[t]{r|c|l}
 \cline{2-2} 0 &\phantom{|}& 0 \\
 \cline{2-2}  &\phantom{|}& 0 \\
 \cline{2-2}  &\phantom{|}& 0 \\
 \cline{2-2}  &\phantom{|}& 0 \\
 \cline{2-2} 
\end{array}
} 
\quad
 {
\begin{array}[t]{r|c|l}
 \cline{2-2} 0 &\phantom{|}& 0 \\
 \cline{2-2}  &\phantom{|}& 0 \\
 \cline{2-2}  &\phantom{|}& 0 \\
 \cline{2-2}  &\phantom{|}& 0 \\
 \cline{2-2}  &\phantom{|}& 0 \\
 \cline{2-2} 
\end{array}
} 
\quad
 {
\begin{array}[t]{r|c|l}
 \cline{2-2} 1 &\phantom{|}& 1 \\
 \cline{2-2}  &\phantom{|}& 1 \\
 \cline{2-2} 
\end{array}
} 
\quad
 {
\begin{array}[t]{r|c|l}
 \cline{2-2} 0 &\phantom{|}& 0 \\
 \cline{2-2}  &\phantom{|}& 0 \\
 \cline{2-2}  &\phantom{|}& 0 \\
 \cline{2-2} 
\end{array}
}$}\raisebox{-15pt}{.}
\]
\end{example}

\bigskip

Let us now turn our attention to $I_0$-highest weight elements in $B^{r,s} \otimes B^{n,1}$ for $1\le r\le n-2$
of weight $\lambda$. They are of the form $b \otimes u$ with $u:=u_{\clfw_n}=(+,\ldots,+) \in B^{n,1}$ being 
(the unique) $I_0$-highest weight element and 
$b\in B(\mu) \subseteq B^{r,s}$ for some $\mu$. Since $\varphi_i(u)=0$ for $1\le i\le n-1$ and $\varphi_n(u)=1$,
we have that the part of $b$ as a KN tableau without the letters $n$ and $\overline{n}$ is highest weight 
of highest weight $\overline{\lambda}$ and the skew shape $\mu/\overline{\lambda}$ is a vertical strip. In $b$, 
the vertical strip $\mu/\overline{\lambda}$ contains the letter $n$ and $\overline{n}$ in alternating order.
Define $\eta$ as follows. In column $i$ of $b$, let $x$ be the number of removable 
$n$ and $\overline{n}$ pairs and $y$ the number of addable $n$ and $\overline{n}$ pairs.
Form a new tableau with the same $\overline{\lambda}$, but $x$ and $y$ interchanged for each column $i$. 
Then $\eta$ is the shape of this tableau.

\begin{example}
\label{example:otherside}
Consider type $D_{10}^{(1)}$ and
\[
\begin{tikzpicture}[baseline]
\matrix [matrix of math nodes,column sep=-.4, row sep=-.4,text height=9pt,text width=12pt,align=center,inner sep=1.5] 
 {
 	\node[draw,fill=gray!40]{1}; & 
	\node[draw,fill=gray!40]{1}; & 
	\node[draw,fill=gray!40]{1}; \\
 	\node[draw,fill=gray!40]{2}; & 
	\node[draw,fill=gray!40]{2}; & 
	\node[draw,fill=gray!15]{\overline{10}}; \\
	\node[draw,fill=gray!40]{3}; & 
	\node[draw,fill=gray!40]{3}; & 
	\node[draw]{10}; \\
	\node[draw,fill=gray!40]{4}; & 
	\node[draw,fill=gray!40]{4}; & 
	\node[draw]{\overline{10}}; \\
	\node[draw,fill=gray!40]{5}; & 
	\node[draw,fill=gray!40]{5}; \\
	\node[draw,fill=gray!40]{6}; & 
	\node[draw,fill=gray!15]{10}; \\
	\node[draw,fill=gray!40]{7}; \\
	\node[draw,fill=gray!15]{\overline{10}}; \\
 };
\end{tikzpicture}
\otimes
\begin{tikzpicture}[baseline]
\matrix [tab] 
 {
	\node[draw]{+}; \\
	\node[draw]{+}; \\
	\node[draw]{+}; \\
	\node[draw]{+}; \\
	\node[draw]{+}; \\
	\node[draw]{+}; \\
	\node[draw]{+}; \\
	\node[draw]{+}; \\
	\node[draw]{+}; \\
 };
\end{tikzpicture}
= b \otimes u \in B^{8,3} \otimes B^{10,1}.
\]
Thus we have $\wt(b \otimes u) = \lambda = \clfw_1 + \clfw_5 + \clfw_7 + \clfw_9$ and $\eta = \clfw_2 + \clfw_6 + \clfw_8$.
\end{example}

We use~\eqref{equation.Phi commute f} in order to simplify the proof of the following lemma (instead of proving it directly using $\Phi$).

\begin{lemma}\label{lem:Rinv_spin2}
Consider the $I_0$-highest weight element $b\otimes u\in B^{r,s}\otimes B^{n,1}$ of weight $\lambda$ 
with $1\le r\leq n-2$ and $\eta$ as defined above. Then the rigged configuration
\[
	(\nu,J) = \Phi(b \otimes u)
\]
is the same as in Lemma~\ref{lem:Rinv_spin1}.
\end{lemma}

\begin{proof}
Recall that $\Phi(u)$ is the empty rigged configuration $\nu^{(a)}=\emptyset$.
Therefore, by the combinatorial procedure of $\Phi$, we see that the only
difference between $\Phi(b\otimes u)$ and $\Phi(b)$ is that the riggings
of $\nu^{(n)}$ for the latter case is smaller than those for the former case by 1
since there is a contribution to the vacancy number by the tensor factor $B^{n,1}$
to $P_i^{(n)}(\nu)$. Thus it is enough to show that for $(\widetilde{\nu}, \widetilde{J})=\Phi(b)$, we
have $\widetilde{\nu}=\Gamma(\eta,\lambda)$ and 
$\widetilde{J}_i^{(a)}=0$ whenever $J_i^{(a)}=1$ in Lemma~\ref{lem:Rinv_spin1} and
$\widetilde{J}_i^{(a)}=-1$ otherwise for $a=1,2$.

We are going to prove the lemma by induction on $\lvert \mu/\overline{\lambda} \rvert$. In fact, for simplicity
we prove the statement simultaneously for the case stated as well as with all letters $n$ and $\overline{n}$
as well as $u_{\varpi_n}$ and $u_{\varpi_{n-1}}$ interchanged. The claim for $\Phi(b\otimes u_{\varpi_{n-1}})$
is the same as in Lemma~\ref{lem:Rinv_spin1} with $(\nu,J)^{(n-1)}$ and $(\nu,J)^{(n)}$ interchanged.
The base case $\lvert \mu/\overline{\lambda} \rvert = 0$ follows directly from Proposition~\ref{prop:highestRC}. 

Now suppose that $\lvert \mu/\overline{\lambda} \rvert > 0$ and assume that the topmost cell in $\mu/\overline{\lambda}$
is filled with $\overline{n}$ (resp. $n$). Suppose this cell is at height $h$.
Consider
\[
	b' = e_h e_{h+1} \cdots e_{n-2} e_n (b) \qquad (\text{resp. } b' = e_h e_{h+1} \cdots e_{n-2} e_{n-1} (b)).
\]
Compared to $b$, the element $b'$ has a letter $h$ instead of $\overline{n}$ (resp. $n$) in this cell, so that 
$\overline{\lambda}' = \overline{\lambda}+\epsilon_h$. Hence by induction hypothesis
$\Phi(b')$ is given as stated by the extension of the lemma. Furthermore, by Theorem~\ref{th:welldef}
the Kashiwara operators $e_i$ and $\Phi$ commute, so that it suffices to check that
\[
	e_h e_{h+1} \cdots e_{n-2} e_n (\widetilde{\nu}, \widetilde{J}) \qquad (\text{resp. }
	e_h e_{h+1} \cdots e_{n-2} e_{n-1} (\widetilde{\nu}, \widetilde{J}))
\]
is indeed $(\widetilde{\nu}',\widetilde{J}')=\Phi(b')$ as stated in the extension of the lemma. Recall the action of $e_i$ 
on rigged configurations as given in Definition~\ref{definition.crystal operators on rc}. The smallest string with
rigging -1 in $(\widetilde{\nu},\widetilde{J})^{(n)}$ (resp. $(\widetilde{\nu},\widetilde{J})^{(n-1)}$)
is of length $h$, so that $e_n$ (resp. $e_{n-1}$) removes a box from this string. By the form of 
$(\widetilde{\nu},\widetilde{J})$ as stated in the lemma, it is not hard to see that 
$e_{i+1} e_{i+2} \cdots e_{n-2} e_n(\widetilde{\nu}, \widetilde{J})$ 
(resp. $e_{i+1} e_{i+2} \cdots e_{n-2} e_{n-1}(\widetilde{\nu}, \widetilde{J})$) has a negative rigging $-1$
of smallest length in a string of length $h$ in the $i$-th rigged partition for each $h\le i \le n-2$. 
This shows that $\widetilde{\nu}^{'(n-2)} = (\overline{\lambda}')^{c}$ as required and all $\widetilde{\nu}^{'(a)}$ with 
$1\le a <n-2$ are as desired.
The cells $c_1', c_2', \ldots$ in the construction of $\widetilde{\nu}^{'(n-1)}$ and $\widetilde{\nu}^{'(n)}$
are obtained from $(\overline{\lambda}')^c/(\eta')^c$. There are two cases to consider:
\begin{enumerate}
\item[(1)] $e_n$ (resp. $e_{n-1}$) removes one of the cells $c_i$ in $(\widetilde{\nu},\widetilde{J})^{(a)}$
for $a=n$ (resp. $a=n-1$).
\item[(2)] $e_n$ (resp. $e_{n-1}$) does not remove any of the cells $c_i$ in $(\widetilde{\nu},\widetilde{J})^{(a)}$
for $a=n$ (resp. $a=n-1$).
\end{enumerate}
Recall that the cells $c_j$ are those in $\overline{\lambda}^c/\eta^c$ or alternatively $\eta/\overline{\lambda}$.

In Case~(1), the cell containing the letter $\overline{n}$ (resp. $n$) at height $h$, which is also $c_i$, is part of $\eta$
and not part of a removable $(n,\overline{n})$-pair. Hence $\eta'=\eta$. 
This implies that the list $c_1',c_2',\ldots$ is the same as the list $c_1,c_2,\ldots$, but with $c_i$ missing.
If in the algorithm for $(\widetilde{\nu},\widetilde{J})$, the cell $c_1$ is added to the $a$-th rigged partition
for $a=n-1$ or $n$, then in $(\widetilde{\nu}',\widetilde{J}')$, the cell $c_1'$ is added to the opposite rigged partition.
All addable $(n,\overline{n})$-pairs at height $<h$ come in pairs of the same row length. Hence still cells are added to the
same rows. The cell $c_i$ is missing from the list $c_j'$ and for $j\ge i$, $c_j'$ adds the same cell as $c_{j+1}$ to the 
same rigged partition, which proves the claim.

In Case~(2), we have $\eta'=\eta+\epsilon_h+\epsilon_{h+1}$ since the $\overline{n}$ (resp. $n$) removed by the string
of $e_j$ is part of a removable $(n,\overline{n})$-pair. Hence $(\overline{\lambda}')^c/(\eta')^c$ compared to 
$\overline{\lambda}^c/\eta^c$ has an extra cell at height $h+1$, which implies that the sequence $c_j'$ compared
to $c_j$ has an extra cell $c_i'$. As in Case~(1), the cells at height $<h$ come in pairs and hence it does not matter
whether one first adds them to the $a$-th rigged partitions for $a=n-1$ or $n$ or vice versa.
Since $c_j'$ contains an extra cell $c_i$, $c_j'$ for $j>i$ adds the same cell as $c_{j-1}$ to the 
same rigged partition, which proves the claim.
\end{proof}

To deal with the combinatorial $R$-matrices between two spin cases, we recall the following
result by Mohammad~\cite{Mohamad}.

\begin{proposition}[\cite{Mohamad}]
\label{prop:mohamad}
For $r = n-1, n$, the $I_0$-highest weight decomposition of $B^{n,1} \otimes B^{r,s}$ and 
$B^{r,s} \otimes B^{n,1}$ is multiplicity free. As a consequence, the combinatorial $R$-matrix
\[
R \colon B^{n,1} \otimes B^{r,s} \to B^{r,s} \otimes B^{n,1}
\]
is uniquely determined by being an $I_0$-crystal isomorphism.
\end{proposition}

By combining these results, we now establish the $R$-invariance of
the rigged configuration bijection for the case of $B^{n,1}\otimes B^{r,s}$
as follows:

\begin{proposition}\label{prop:spin_Rinv}
Consider an $I_0$-highest element $s_1 \otimes b_1$ of $B^{n,1}\otimes B^{r,s}$.
Suppose that 
\[
	R \colon s_1 \otimes b_1 \longmapsto b_2\otimes s_2\in B^{r,s}\otimes B^{n,1}
\]
under the combinatorial $R$-matrix. Then we have
\[
	\Phi(s_1\otimes b_1)=\Phi(b_2\otimes s_2).
\]
\end{proposition}

\begin{proof}
If $r = n-1,n$, this follows from Proposition~\ref{prop:mohamad} and Theorem~\ref{th:welldef}.

Thus assume $r \leq n-2$. By Lemmas~\ref{lem:Rinv_spin1} and~\ref{lem:Rinv_spin2}, it suffices to show that under 
the combinatorial $R$-matrix, an $I_0$-highest weight element in $B^{n,1} \otimes B^{r,s}$ corresponding to the tuple 
$(\eta, \lambda)$ goes to the $I_0$-highest weight element in $B^{r,s} \otimes B^{n,1}$ with the same tuple 
$(\eta, \lambda)$ as given in the corresponding lemmas. We do so by induction on $\lvert \lambda \rvert$ (from largest to smallest)
by either increasing $\lvert \eta \rvert$ or decreasing $\lvert \eta/\overline{\lambda} \rvert$.

Note that there is a unique element of weight $s\varpi_r+\varpi_n$ in both $B^{n,1} \otimes B^{r,s}$
and $B^{r,s} \otimes B^{n,1}$, so that $R(u_{\varpi_n} \otimes u_{s\varpi_r}) = u_{s\varpi_r} \otimes u_{\varpi_n}$. 
We set up the induction in several steps. Let $s_1 \otimes u_\eta \in B^{n,1} \otimes B^{r,s}$ be a highest 
weight element different from $u_{\varpi_n} \otimes u_{s\varpi_r}$. Perform one of the following steps to obtain a new 
$s_1' \otimes u_{\eta'}$ closer to $u_{\varpi_n} \otimes u_{s\varpi_r}$ by induction:

\medskip

\noindent \textbf{Step 1.} Suppose $\varphi_0(s_1 \otimes u_\eta)>0$. Then let $s_1' \otimes u_{\eta'}$ be
the $I_0$-highest weight element in the same component as $f_0(s_1 \otimes u_\eta)$. 

\medskip

\noindent \textbf{Step 2.} Suppose $\varphi_0(s_1 \otimes u_\eta)=0$. Note that in this case $\eta_1=s-1$ or $s$.
Set $b_1 = u_\eta$ and repeat the following steps until $\varphi_0(s_1 \otimes b_1)>0$.
Then let $s_1' \otimes u_{\eta'}$ be the $I_0$-highest weight element in the same component as $f_0 (s_1 \otimes b_1)$.
\begin{enumerate}
\item Let $1\le h < n-1$ be minimal such that $\varphi_h(s_1 \otimes b_1)>0$.
\item Compute $s_1'\otimes b_1' = f_h(s_1\otimes b_1)$.
\item Reset $s_1 \otimes b_1$ to $s_1' \otimes b_1'$.
\end{enumerate}

\medskip

\noindent \textbf{Step 3.} If neither Step 1 nor Step 2 applies or yields a $c=s_1 \otimes b_1$ with $\varphi_0(c)>0$, 
we have $\eta = (s-1) \varpi_r + \varpi_k$ for some $0\le k < r$ and $s_1=(+,\ldots,+)$ or $(+,\ldots,+,-,+,\ldots,+,-)$ with one 
minus in position $n$ and the other in position $k$ or $r$. Note that, for fixed weight $\lambda$, these elements
minimize $\lvert \eta \rvert$ and there is a unique $I_0$-highest weight element corresponding to $(\eta,\lambda)$ with
minimal $\lvert \eta \rvert$. On the $B^{r,s} \otimes B^{n,1}$ side, these are precisely the $I_0$-highest weight elements
$b_2\otimes u_{\varpi_n}$ with $b_2 \in B(s\varpi_r) \subseteq B^{r,s}$ (since minimization of $\lvert \eta \rvert$ corresponds 
to maximization of $\lvert \mu \rvert$). Due to the uniqueness for a given weight class, they must map to each other
under $R$.

\medskip

We are now going to check for Steps 1 and 2 inductively, that if an $I_0$-highest weight element
$s_1 \otimes u_\eta$ characterized by the tuple $(\eta,\lambda)$ goes to $(\eta',\lambda')$ under a given step,
then $b_2 \otimes u_{\varpi_n}$ corresponding to the same $(\eta,\lambda)$ also goes to
$(\eta',\lambda')$ under the same step. 

We begin with Step 1. Suppose $f_0(s_1 \otimes u_\eta) = s_1 \otimes f_0(u_\eta)$.
Take $s_1 \otimes u_\eta$ of weight $\lambda$. By the combinatorial rules for $f_0$ on 
elements in $B^{r,s}$ provided in~\cite{Sch:2008}, we have that $\eta'$ (resp. $\lambda$) is $\eta$ (resp. $\lambda$)
with a vertical domino added to rows 1 and 2. More precisely, $\eta_i'=\eta_i+1$ for $i=1,2$.
Note also that $\eta_1<s$ (and in particular $r$ even) and $\overline{\lambda}_1<\eta_1$ if $\eta_1=s-1$ since 
otherwise $f_0$ does not apply. Now take $b_2 \otimes u_{\varpi_n}$ corresponding to $(\eta,\lambda)$ with $b_2\in B(\mu)$. 
The action of $f_0$ is given by Lemma~\ref{lemma:f0_bump}. If case  Lemma~\ref{lemma:f0_bump}(i) applies, we have
$\mu_i = \eta_i = \overline{\lambda}_i+1$ and $\eta_i'=\mu_i'+1= \overline{\lambda}_i'+1$ for $i=1,2$, so
that indeed $\eta_i'=\eta_i+1$.
Lemma~\ref{lemma:f0_bump}(ii) does not apply since $r$ cannot be odd. If Lemma~\ref{lemma:f0_bump}(iii)
applies, we have $\mu_i=\overline{\lambda}_i+1=\eta_i+1$ and $\mu_i'=\overline{\lambda}'_i+1=\eta'_i+1$
with $\mu_i'=\mu_i+1$ for $i=1,2$, so that again $\eta_i'=\eta_i+1$. This proves the claim for Step 1 with 
$f_0(s_1 \otimes u_\eta) = s_1 \otimes f_0(u_\eta)$.

Step 1 with $f_0(s_1 \otimes u_\eta) = f_0(s_1) \otimes u_\eta$ can only apply when $\eta_1=s$ and $s_1$ has $-$ in 
positions 1 and 2. Under $f_0$ these two minus in $s_1$ are turned into $+$, so that $\lambda_i'=\lambda_i+1$ for $i=1,2$, 
$\lambda_i'=\lambda_i$ for $i\ge 3$, and $\eta'=\eta$. When $r$ is even, the correspondence gives 
$\mu_i=\overline{\lambda}_i=s-1$ for $i=1,2$.
Then Lemma~\ref{lemma:f0_bump}(iii) applies and yields $\mu_i'=s=\overline{\lambda}_i'=\eta_i'$ for $i=1,2$
as desired. When $r$ is odd, the correspondence yields $\eta_1=s$, $\overline{\lambda}_1=s-1$,
and $\mu_1=s$. Lemma~\ref{lemma:f0_bump}(i) or (ii) applies and it can again be easily checked that
$\eta'=\eta$, proving Step~1 with $f_0(s_1 \otimes u_\eta) = f_0(s_1) \otimes u_\eta$.

Next consider Step 2. The algorithm of always picking the smallest applicable $f_h$ and applying it, 
has the following effect on $s_1 \otimes u_\eta$. It raises the entries in the shortest columns of $u_\eta$
and/or lowers the lowest minus signs in $s_1$ until $f_0$ is applicable. We are going to consider
several cases:
\begin{enumerate}
\item[(a)] No minus sign in $s_1$ moves to positions 1 and 2.
\item[(b)] Only one minus sign in $s_1$ moves to position 1 or 2.
\item[(c)] Two minus signs in $s_1$ move to positions 1 and 2.
\end{enumerate}
First consider Case~(a). Note that $\varepsilon_0(s_1)=1$, so that the sequence of $f_h$ must raise 
the entries in at least the two shortest columns in order to achieve $\varphi_0(s_1 \otimes u_\eta)>0$.
Let $h_1 \le h_2<r$ be the height of the two rightmost columns in $\eta$. Then the algorithm for Step 2
adds a vertical domino to the leftmost column of height $h_2$ in $\eta$ and does the same to $\overline{\la}$.
On the other side $b_2 \otimes u_{\varpi_n}$ with $b_2 \in B(\mu)$, the subpart of $b_2$ of shape $\overline{\la}$ 
is a tableau with $i$ in row $i$, whereas $\mu/\overline{\la}$ is filled with the letters $n$ and 
$\overline{n}$. Since we only apply $f_h$ with $1\le h\le n-2$, the letters $n$ and $\overline{n}$ never change
in $b_2$. By the correspondence between $(\eta,\la)$ and $\mu$ as described before Example~\ref{example:otherside},
the vertical strip $\mu/\overline{\la}$ has boxes at height $h$ with $h_1< h \le h_2$ and $b_2$ contains
letters $n$ and $\overline{n}$ in these cells. The steps in the algorithm add a vertical strip at height $h_2$
to $\overline{\la}$ and as a consequence push any letters $n$ and/or $\overline{n}$ at height $h_2+1$ and $h_2+2$
one column to the right. Hence $\mu'=\mu+\epsilon_{h_2+1}+\epsilon_{h_2+2}$ which implies that the
correspondence between $(\eta',\lambda')$ and $\mu'$ also holds, proving the claim.

Next consider Case~(b). Let $h_1 \le h_2<r$ again be the height of the two rightmost columns in $\eta$.
The spin column $s_1$ must have precisely one minus sign in position $1\le h \le h_2$, which is either 
in position $h_1$ or position $h_2$. First consider that it is in position $h_2>h_1$. Then the algorithm for Step 2
adds a vertical domino to the column of height $h_1$ in $\eta$ and moves the minus sign in $s_1$
from position $h_2$ to position $h_1+2$. This means $\eta'=\eta+\epsilon_{h_1+1} +\epsilon_{h_1+2}$
and $\overline{\la}' = \overline{\la} + \epsilon_{h_2} + \epsilon_{h_1+1}$. On the $b_2 \otimes u_{\varpi_n}$
side, the vertical strip $\mu/\overline{\la}$ has boxes at heights $h_1<h\le h_2-2$ and $h_2$.
The algorithm for Step 2 on this side again changes $\overline{\la}'=\overline{\la} + \epsilon_{h_2} + \epsilon_{h_1+1}$.
The letters $n$ and $\overline{n}$ in $b_2$ in the rightmost column move up one and the letter $\overline{n}$
at height $h_2$ moves from the second to last to last column. This implies $\eta'=\eta+\epsilon_{h_1+1} +\epsilon_{h_1+2}$
as before, proving the claim.

Next consider Case~(b) with a minus in $s_1$ at height $h_1 \le h_2<r$. In this case, the algorithm for Step 2
adds a vertical domino at height $h_2$ to $\eta$ and $\overline{\la}$, so that 
$\eta'=\eta +\epsilon_{h_2+1} + \epsilon_{h_2+2}$ and $\overline{\la}'=\overline{\la} +\epsilon_{h_2+1} + \epsilon_{h_2+2}$.
On the $b_2 \otimes u_{\varpi_n}$ side, the vertical strip $\mu/\overline{\la}$ has boxes at heights $h_1\le h\le h_2$.
The algorithm for Step 2 also changes $\overline{\la}'=\overline{\la} +\epsilon_{h_2+1} + \epsilon_{h_2+2}$
and pushes any $n$ and $\overline{n}$ at height $h_2+1$ and $h_2$ one column to the right.
This implies that indeed $\eta'=\eta +\epsilon_{h_2+1} + \epsilon_{h_2+2}$ as desired.

Finally consider Case~(c). In this case, there must be at least two minus signs in $s_1$ in positions
$1\le i \le h_2$, where $h_1\le h_2<r$ are again the heights of the rightmost columns in $\eta$.
The algorithm for Step 2 removes the lowest two minus signs in $s_1$. Let $i_1<i_2$ be the positions
of these two minus signs. Then $\overline{\la}'=\overline{\la}+\epsilon_{i_1}+\epsilon_{i_2}$
and $\eta'=\eta$. On the $b_2 \otimes u_{\varpi_n}$ side, the vertical strip $\mu/\overline{\la}$
has boxes at height $i_1\le i \le i_2-2$ and $i_2$ if $i_1\le i_2-2$ and at height $i_1$ and $i_1+1$
if $i_2=i_1+1$. The algorithm of Step 2 also changes $\overline{\la}'=\overline{\la}+\epsilon_{i_1}+\epsilon_{i_2}$.
In addition it moves the letters $n$ and $\overline{n}$ in $b_2$ up by one in the last column of $b_2$ and
$n$ at height $i_2$ one column to the right if $i_1\le i \le i_2-2$. When $i_2=i_1+1$, it moves the 
letters $n$ and $\overline{n}$ in the last column of $b_2$ up by two. Comparing $\mu'$ and $\overline{\la}'$
shows that indeed $\eta'=\eta$ as desired.
\end{proof}

Finally, we prepare the following proposition to prove the $R$-invariance.
Let $u\in B^{n,1}$ be the $I_0$-highest weight element. For $b\in B$ 
we define $\mathbb{T}(b)$ by $R(b\otimes u)=u'\otimes \mathbb{T}(b)$. 
Here we view $\mathbb{T} \colon B \to B$ as an operator
by passing through $u$ via the combinatorial $R$-matrix.

\begin{proposition} \label{prop:induction}
Suppose $r_2\le r_1$.
Let $b_2\ot b_1\in B^{r_2,s_2}\ot B^{r_1,s_1}$ be an $I_0$-highest weight element.
Then there exists an integer $N$ such that $\mathbb{T}^N(b_2\ot b_1)$ 
belongs to $B(s_2\varpi_{r_2})\ot B(s_1\varpi_{r_1})$ and the right factor 
is the $I_0$-highest weight element. Moreover,
\begin{enumerate}
\item if $r_2\le n-2$ and $r_1\le n-1$, then the left factor does not contain barred 
	letters;
\item otherwise, the left factor is the $I_0$-highest weight element.
\end{enumerate}
\end{proposition}

\begin{proof}
We first assume
$r_1,r_2\le n-2$. Suppose $b_2\ot b_1\in B^{r_2,s_2}\ot B^{r_1,s_1}$ is $I_0$-highest 
weight. Then $b_1=u_\mu$ where $u_\mu$ is the $I_0$-highest weight element of
$B(\mu)\subset B^{r_1,s_1}$.
In view of the description of the combinatorial $R\colon B^{r_1,s_1}\ot B^{n,1}\rightarrow
B^{n,1}\ot B^{r_1,s_1}$ in terms of the pairs $(\eta,\lambda)$ as described in the proof of 
Proposition~\ref{prop:spin_Rinv}, successive applications of $R$ take this $u_\mu$ to 
$u_{s_1\varpi_{r_1}}$. Thus we can assume $b_1=u_{s_1\varpi_{r_1}}$ and 
\begin{equation} \label{induction}
e_i^{s_1\delta_{i,r_1}+1}b_2=0\text{ for }i\in I_0.
\end{equation}
We also know $R(u_{s_1\varpi_{r_1}}\ot u)=u\ot u_{s_1\varpi_{r_1}}$. So we consider
$R(b_2\ot u)$ next. Consider $\mathrm{high}(b_2\ot u)$. The left component
may contain $n,\ol{n}$ pairs. If so, $R\colon B^{r_2,s_2}\ot B^{n,1}\rightarrow
B^{n,1}\ot B^{r_2,s_2}$ removes them.
Therefore, we can assume that there do not exist any $n,\ol{n}$ pairs.
Once all $n,\ol{n}$ pairs are removed, suppose $b_2 \in B(\lambda)$.
Then the right component of $R(b_2\ot u)$ belongs to $B(\la')$, where $\la'$ is strictly
greater than $\la$ in the dominance order. Hence, we can assume 
$\la=s_2\varpi_{r_2}$. From \eqref{induction}, the first barred letter in the inverse
column reading word of $b_2$ is $\ol{r}_1$. Let 
$R(b_2\ot u)=c\ot \widetilde{b}_2\in B^{n,1}\ot B^{r_2,s_2}$. Since $c$ contains a $-$,
the number of barred letters in $\widetilde{b}_2$ is at least one smaller
than that of $b_2$. Therefore, the claim follows.

A similar argument when $r_2\le n-2$ and $r_1=n-1$ finishes the proof of (i). 
Case (ii) is easier.
\end{proof}

\begin{example}
Consider type $D_8^{(1)}$, $B = B^{4,4} \otimes B^{6,3}$ and
\[
b_2 \otimes b_1 =
\Yvcentermath1
\young(1111,227\mfour,35,4\meight) \otimes \young(111,222,333,444,5,6)
\]
Then we have
\begin{align*}
\mathbb{T}(b_2 \otimes b_1) & = \Yvcentermath1
\young(1111,2227,337,45\mseven) \otimes \young(111,222,333,444,55,66)\;,
&& u' = (+,-,-,-,-,-,+,-),
\\ \mathbb{T}^2(b_2 \otimes b_1) & = \Yvcentermath1
\young(1111,2222,3337,467\msix) \otimes \young(111,222,333,444,555,666)\;,
&& u' = (+,-,-,+,+,-,-,+),
\\ \mathbb{T}^3(b_2 \otimes b_1) & = \Yvcentermath1
\young(1111,2222,3333,4777) \otimes \young(111,222,333,444,555,666)\;,
&& u' = (+,+,-,+,+,+,-,+),
\end{align*}
where $u'$ is defined by $R(b\otimes u)=u'\otimes \mathbb{T}(b)$ with $b=b_2 \otimes b_1$.
This illustrates Case~(i) of Proposition~\ref{prop:induction}.
\end{example}

\begin{example}
Consider type $D_8^{(1)}$, $B = B^{4,4} \otimes B^{8,3}$ and
\[
b_2 \otimes b_1 =
\Yvcentermath1
\young(111\meight,222\mseven,3,4) \otimes u_{3\varpi_8}.
\]
Then we have
\begin{align*}
\mathbb{T}(b_2 \otimes b_1) & =
\Yvcentermath1
\young(111,222,33,44) \otimes u_{3\varpi_8},
&& u' = (+,+,-,-,+,+,-,-), 
\\ \mathbb{T}^2(b_2 \otimes b_1) & =
\Yvcentermath1
\young(1111,2222,333,444) \otimes u_{3\varpi_8},
&& u' = (+,+,-,-,+,+,-,-),
\\ \mathbb{T}^3(b_2 \otimes b_1) & =
\Yvcentermath1
\young(1111,2222,3333,4444) \otimes u_{3\varpi_8},
&& u' = (+,+,-,-,+,+,+,+). 
\end{align*}
This illustrates Case~(ii) of Proposition~\ref{prop:induction}.
\end{example}

The following property is one of the most profound properties of the rigged configuration bijection.

\begin{theorem}[$R$-invariance] \label{th:main}
Let $B=B^{r_k,s_k}\otimes\cdots\otimes B^{r_2,s_2}\otimes B^{r_1,s_1}$
and let $B'$ be an arbitrary reordering of the factors of $B$.
For a given element $b\in B$, suppose that $b'\in B'$ is isomorphic to
$b$ under the combinatorial $R$-matrix:
\[
R \colon b\iso b'.
\]
Then the corresponding unrestricted rigged configurations are invariant:
\[
\Phi(b)=\Phi(b').
\]
\end{theorem}

\begin{proof}
We divide the proof into three steps.
\medskip

\noindent
{\bf Step 1.}
By Theorem \ref{th:welldef}, we can assume that $b$ is an $I_0$-highest weight element.
Suppose that $b=b_k\otimes b_{k-1}\otimes b_{k-2}\otimes\cdots\otimes b_1\in B$.
By the recursive structure of the algorithm of the rigged configuration bijection $\Phi$,
we see that it is enough to consider the case
$b'=\widetilde{b}_{k-1}\otimes \widetilde{b}_k\otimes b_{k-2}\otimes\cdots\otimes b_1$,
where $R \colon b_k\otimes b_{k-1} \iso \widetilde{b}_{k-1}\otimes \widetilde{b}_k$.
Then the relationship between the $\diamond$ duality and the $\theta$ operation
established in Proposition~\ref{prop:Phi_welldef1}(7) shows that it is enough to consider the case $k=2$.
Therefore we concentrate on the case when $b=b_2\otimes b_1$ is $I_0$-highest weight.
\medskip

\noindent
{\bf Step 2.}
Let $u\in B^{n,1}$ be the $I_0$-highest weight element and $b=b_2 \otimes b_1 \in B$ be $I_0$-highest weight. 
We shall show that
\begin{align}\label{eq:spin_Rinv}
\Phi(b\otimes u)=\Phi(u'\otimes \mathbb{T}(b)),
\end{align}
where $R \colon b\otimes u\iso u'\otimes \mathbb{T}(b)$.
By Proposition~\ref{prop:spin_Rinv} we have
\[
\Phi(b_2\otimes b_1\otimes u)=\Phi(b_2\otimes u''\otimes b_1'),
\]
where $R \colon b_1\otimes u\iso u''\otimes b_1'$.
Then by Proposition~\ref{prop:Phi_welldef1}(7) we have
\[
\Phi(\widehat{u}\otimes \widehat{b}_1 \otimes \widehat{b}_2) =
\Phi(\widehat{b}_1' \otimes \widehat{u}'' \otimes \widehat{b}_2),
\]
where $(b_2\otimes b_1\otimes u)^{\diamond} = \widehat{u} \otimes \widehat{b}_1 \otimes \widehat{b}_2$
and $(b_2\otimes u''\otimes b_1')^{\diamond} = \widehat{b}_1' \otimes \widehat{u}'' \otimes \widehat{b}_2$.
Since $\widehat{u}'' \in B^{n,1}$, again by Proposition~\ref{prop:spin_Rinv} we have
\[
\Phi(\widehat{b}_1' \otimes \widehat{u}'' \otimes \widehat{b}_2)=
\Phi(\widehat{b}_1' \otimes \widehat{b}_2' \otimes \widehat{u}'),
\]
where $R \colon \widehat{u}'' \otimes \widehat{b}_2 \iso \widehat{b}_2' \otimes \widehat{u}'$.
Note that $(\widehat{b}_1' \otimes \widehat{b}_2' \otimes \widehat{u}')^{\diamond}
= u'\otimes \mathbb{T}(b)$ since $\diamond$ is involutive.
Therefore by Proposition~\ref{prop:Phi_welldef1}(7) we obtain~\eqref{eq:spin_Rinv}.
\medskip

\noindent
{\bf Step 3.} 
Let $b_2\ot b_1\in B$ be $I_0$-highest weight and $R(b_2\ot b_1)=\tilde{b}_1\ot\tilde{b}_2$.
Then we have 
\[
b_2\ot b_1\ot u\simeq u'\ot\mathbb{T}(b_2\ot b_1)
\text{ and }
\tilde{b}_1\ot\tilde{b}_2\ot u\simeq u'\ot\mathbb{T}(\tilde{b}_1\ot\tilde{b}_2).
\]
Thus by Step 2, we have
\[
\Phi(b_2\ot b_1\ot u)=\Phi(u'\ot\mathbb{T}(b_2\ot b_1))
\text{ and }
\Phi(\tilde{b}_1\ot\tilde{b}_2\ot u)=\Phi(u'\ot\mathbb{T}(\tilde{b}_1\ot\tilde{b}_2)).
\]
Since $\Phi(b_2\ot b_1\ot u)=\Phi(\tilde{b}_1\ot\tilde{b}_2\ot u)$ is equivalent to
$\Phi(b_2\ot b_1)=\Phi(\tilde{b}_1\ot\tilde{b}_2)$, the claim is reduced to showing
$\Phi(\mathbb{T}(b_2\ot b_1))=\Phi(\mathbb{T}(\tilde{b}_1\ot\tilde{b}_2))$.
Thanks to Proposition~\ref{prop:induction}, we are left to show 
$\Phi(b_2\ot b_1)=\Phi(\tilde{b}_1\ot\tilde{b}_2)$ when either $b_2\ot b_1$ 
or $\tilde{b}_1\ot\tilde{b}_2$ has the property described in the proposition. 
If $r_1,r_2\le n-2$, this $R$-matrix is nothing but type $A$ by 
\cite[Proposition 9.1]{LOS12} and the $R$-invariance of type $A$ case is already 
treated in~\cite[Lemma 8.5]{KSS:2002}. 
If $r_1=n$, by Proposition~\ref{prop:induction} both factors are $I_0$-highest weight and the images of $\Phi$ 
are empty. When $r_1=n-1$, then by symmetry we can interchange in all arguments $n-1$ and $n$
(including in the definition of $\mathbb{T}$), so that this case reduces to the $r_1=n$ case.
\end{proof}

Theorem~\ref{th:main} means that the rigged configuration bijection gives
an explicit algorithm to compute the combinatorial $R$-matrices.
It is noteworthy that the rigged configuration bijection gives not only
the combinatorial $R$-matrices for twofold tensor products but also
much more general reorderings of multiple tensor products
by only computing $\Phi$ and $\Phi^{-1}$.
In the following, we give one such example.

\begin{example}\label{ex:combR}
Consider the following element of $B=B^{3,3}\otimes B^{2,4}\otimes B^{2,2}$ of type $D^{(1)}_5$:
\[
b=\Yvcentermath1
\young(11\mtwo,22,\mfive\mfive) \otimes
\young(123\mthree,23\mfour\mone)\otimes
\young(13,34)\;.
\]
The corresponding unrestricted rigged configuration is
\begin{center}
\unitlength 10pt
\begin{picture}(35,7)
\put(0,3.35){
\put(0,0){\yng(4,1,1)}
\put(1.4,0.2){0}
\put(1.4,1.2){0}
\put(4.8,2.5){$-1$}
}
\put(7,0){
\put(0,0){\yng(4,3,2,1,1,1)}
\multiput(1.4,0.2)(0,1.1){3}{$-1$}
\multiput(2.5,3.6)(1.1,1.1){3}{$-2$}
}
\put(14.0,0){
\put(0,0){\yng(5,2,2,1,1,1)}
\multiput(1.4,0.2)(0,1.1){3}{1}
\multiput(2.6,3.6)(0,1.1){2}{0}
\put(6,5.8){3}
}
\put(21.7,3.35){
\put(0,0){\yng(3,1,1)}
\multiput(1.4,0.2)(0,1.1){2}{0}
\put(3.8,2.5){0}
}
\put(27.2,3.35){
\put(0,0){\yng(5,1,1)}
\multiput(1.4,0.2)(0,1.1){2}{0}
\put(6,2.5){$-2$}
}
\end{picture}
\raisebox{30pt}{.}
\end{center}
From this unrestricted rigged configuration, we can obtain the image in
$B'=B^{2,2}\otimes B^{2,4}\otimes B^{3,3}$ as follows:
\[
b'=\Yvcentermath1
\young(1,\mfive) \otimes
\young(113\mthree,22\mfive\mone)\otimes
\young(123,234,3\mfour\mtwo)\;.
\]
Then by Theorem \ref{th:main} we have $R:b\iso b'$.
In this way we can obtain the combinatorial $R$-matrices as $R=\Phi^{-1}_{B'}\circ\Phi_B$.
\end{example}

\section{Energy function and fermionic formula} \label{sec:energy}

\subsection{Involution $\varsigma$ and energy function}

In \cite{LOS11}, the involution $\varsigma$ on $B^{r,s}$ ($1\le r\le n-2,s\ge1$)
was shown to exist, which is the unique map satisfying 
\begin{subequations}
\label{varsigma}
\begin{align}
\varsigma(e_ib) & = e_{n-i}\varsigma(b),
\\ \varsigma(f_ib) & =f_{n-i}\varsigma(b),
\end{align}
\end{subequations}
for any $i\in I,b\in B^{r,s}$. We need the following lemma later.

\begin{lemma}  \label{lem:sigma on KN}
For $u_\la\in B(\la)\subseteq B^{r,s}$, $\varsigma(u_\la)$ belongs to $B(s\varpi_r)$ and is given
as follows. To $\la$ associate a tuple $\mu=(\la'_s,\la'_{s-1},\ldots, \la'_1,r,\ldots,r)$ with $2s$ entries,
where $\lambda'$ is the transpose of the partition $\lambda$.
Denote the $j$-th entry of $\mu$ by $\mu_j$. Let $t$ be the KN tableau for $\varsigma(u_\la)$. Then the
$k$-th column of $t$ is given by
\[
\ol{n-\mu_{2k-1}+1} \cdots \ol{n-1} \ol{n} n \cdots n \ol{n} (n - \mu_{2k}) \cdots (n-r+2) (n-r+1).
\]
\end{lemma}

\begin{proof}
Since $u_\la$ is the unique element of $B^{r,s}$ of weight $\la$ satisfying $e_iu_\la=0$ 
for $i=1,\ldots,n$, the element $\varsigma(u_\la)$ should be the unique element of $-w^{A_{n-1}}_0\la$
satisfying $e_i\varsigma(u_\la)=0$ for $i=0,1,\ldots,n-1$, where $w^{A_{n-1}}_0$ 
is the longest element of the Weyl group of $A_{n-1}\subseteq D_n$. 
One checks that $t$ is an allowed KN tableau \cite{KN:1994} of the desired weight. Hence,
it is sufficient to show that $e_it=0$ for $i=0,1,\ldots,n-1$. For $i=1,\ldots,n-1$, this
is immediate. To see $e_0t=0$, calculate the $\{2,\ldots,n\}$-highest weight
element in the same component as $t$, which is the tableau whose columns are all $23\cdots r+1$, and hence
$\sigma(t)=t$ with $\sigma$ as in~\eqref{eq:01_involution}. In view of the definition of $e_0$
given in~\eqref{e0 f0}, we have $e_0t=0$, since $e_1t=0$.
\end{proof}

We need an involution $\varsigma$ also on $B^{n-1,s}$ and $B^{n,s}$. When $n$
is odd, it becomes a map from $B^{n-1,s}$ to $B^{n,s}$ and vice versa. To incorporate
this situation, we introduce $\eta$ defined to be $0$ if $n$ is even and $1$ otherwise.
Then there exists a map $\varsigma \colon B^{n-1,s} \to B^{n-1+\eta,s}$ (resp. $B^{n,s} \to B^{n-\eta,s}$)
satisfying \eqref{varsigma}. In fact, it is characterized by~\eqref{varsigma} and defining the image of 
$u_{s\varpi_{n-1}}$ (resp. $u_{s\varpi_n}$)
to be the unique element in $B^{n-1+\eta,s}$ (resp. $B^{n-\eta,s}$) of weight 
$-w^{A_{n-1}}_0(s\varpi_{n-1})$ (resp. $-w^{A_{n-1}}_0(s\varpi_n)$).

For a multiple tensor product of KR crystals $B=B^{r_k,s_k}\ot\cdots\ot B^{r_1,s_1}$, we define 
\[
\varsigma(b)=\varsigma(b_k)\ot\cdots\ot\varsigma(b_1)
\]
for $b=b_k\ot\cdots\ot b_1\in B$

\begin{proposition} \label{prop:rc and sigma}
For any $I_0$-highest weight element $u_\la$ of a single KR crystal $B^{r,s}$, we have
$(\rs\circ\varsigma)(u_\la)=(\varsigma\circ\rs)(u_\la)$.
\end{proposition}

\begin{proof}
The case when $r=n-1$ or $n$ is straightforward. Both $\rs\circ\varsigma$ and 
$\varsigma\circ\rs$ send $u_{s\varpi_{n-1}}$ (resp. $u_{s\varpi_n}$) to
$v_{-(s-1)w_0\varpi_{n-1}}\ot v_{-w_0\varpi_{n-1}}$ 
(resp. $v_{-(s-1)w_0\varpi_n}\ot v_{-w_0\varpi_n}$) where $w_0$ is the longest element
of $A_{n-1}\subseteq D_n$ and $v_\tau$ stands for the unique element in the crystal
of weight $\tau$.

Suppose $r\le n-2$. In this case, $(\rs\circ\varsigma)(u_\la)$ is given by 
cutting the rightmost column of $\varsigma(u_\la)$.
From Proposition~\ref{lem:rc on KN} we see $\rs(u_\la)=t\ot u_{\varpi_r}$,
where the tableau $t$ is described in the proposition. We use all the notation for $t$ from
Proposition~\ref{lem:rc on KN}.
Hence $(\varsigma\circ\rs)(u_\la)=\varsigma(t)\ot\varsigma(u_{\varpi_r})$.
By Lemma~\ref{lem:sigma on KN}, the right components
of $(\rs\circ\varsigma)(u_\la)$ and $(\varsigma\circ\rs)(u_\la)$ agree. 
We prove that the left components also agree by decreasing induction on $p$, where $p$
is as in Proposition~\ref{lem:rc on KN} for $t$.
If $p=r$, the claim is true, since in this case $t$ is $I_0$-highest weight and
the agreement can be shown by Lemma~\ref{lem:sigma on KN}.
Suppose $p<r$. Let $t'$ be the tableau that differs from $t$ only in the first column
with $\ol{p+1} \cdots \ol{r-1} \ol{r} q \cdots 21$ replaced by $\ol{p+3} \cdots \ol{r-1} \ol{r} q \cdots 21$. 
These tableaux are connected in the crystal graph as $t'=e_{{\bf a}_2}f_0f_{{\bf a}_1}t$ where
\begin{align*}
{\bf a}_1&={\bf b}(2,1+\mu_{\ge 2})\cdots{\bf b}(p-2,1+\mu_{\ge p-2})
{\bf b}(p,1+\mu_{\ge p}),\\
{\bf a}_2&={\bf b}(p,\mu_{\ge p})\cdots{\bf b}(4,\mu_{\ge 4}){\bf b}(2,\mu_{\ge 2}),
\end{align*}
${\bf b}(j,a)=j^a(j-1)^a(j+1)^aj^a$ and $\mu_{\ge c}=\sum_{j=c}^{\ell(\mu)}\mu_j$.
(See \eqref{e sequence} for the definition of $e_{\bf{a}}$. The definition of $f_{\bf{a}}$ is similar.)
In fact, the tableau $f_{{\bf a}_1}t$ has the following form. The first column is 
\[
\mone\,\mtwo\,\ol{p+3}\cdots \ol{r-1}\, \ol{r}\, (q+2) \cdots 43,
\]
and the entries in the other part of $t$ are shifted by 2. Thus $f_0$ removes 
$\mone\,\mtwo$ from the first column.

Since $\varsigma(e_{{\bf a}_2}f_0f_{{\bf a}_1}t)=e_{n-{\bf a}_2}f_nf_{n-{\bf a}_1}t$, where
for ${\bf a}=a_1\cdots a_m$, $n-{\bf a}$ stands for $(n-a_1)\cdots(n-a_m)$, we are left 
to show that $e_{n-{\bf a}_2}f_nf_{n-{\bf a}_1}t$ agrees with the first $s-1$ columns
of $\varsigma(u_\la)$ of Lemma \ref{lem:sigma on KN} with $\la'_1$ replaced by
$\la'_1+2$. This can be checked directly.
\end{proof}

In Section \ref{subsec:path}, we defined the combinatorial $R$-matrix
$R \colon B^{r_2,s_2}\otimes B^{r_1,s_1} \to B^{r_1,s_1}\otimes B^{r_2,s_2}$.
In addition, there exists a function $\ol{H} \colon B^{r_2,s_2}\otimes B^{r_1,s_1}\to \mathbb{Z}$, 
called the \defn{local energy function}, unique up to a global additive constant.
It is constant on $I_0$-components and satisfies for all $b_2\in B^{r_2,s_2}$
and $b_1\in B^{r_1,s_1}$ with $R(b_2\otimes b_1)=b_1'\otimes b_2'$
\begin{equation} \label{def H}
  \ol{H}(e_0(b_2\otimes b_1))=
  \ol{H}(b_2\otimes b_1)+
  \begin{cases}
    1 & \text{if $\varepsilon_0(b_2)>\varphi_0(b_1)$ and
    $\varepsilon_0(b_1')>\varphi_0(b_2')$,} \\
    -1 & \text{if $\varepsilon_0(b_2)\le\varphi_0(b_1)$ and
    $\varepsilon_0(b_1')\le \varphi_0(b_2')$,} \\
    0 & \text{otherwise.}
  \end{cases}
\end{equation}
We shall normalize the local energy function by the condition
$\ol{H}(u_{s_2\varpi_{r_2}}\otimes u_{s_1\varpi_{r_1}})=0$.

For a crystal $B^{r,s}$, the \defn{intrinsic energy} $\ol{D}_{B^{r,s}} \colon B^{r,s}\to \mathbb{Z}$ is defined as follows. 
When $r\le n-2$, we define $\ol{D}_{B^{r,s}}(b)=(rs - \lvert\la\rvert )/2$ for $b\in B(\la)\subseteq B^{r,s}$ 
(see~\eqref{eq:classical_decomposition_1}). 
When $r=n-1,n$, there is only one $I_0$-component 
(see~\eqref{eq:classical_decomposition_2}) and
we set $\ol{D}_{B^{r,s}}(b)=0$ for $b\in B^{r,s}$.
On the tensor product $B=B_k\otimes \cdots\otimes B_1$ of single KR crystals 
$B_j$, there is an intrinsic energy function $\ol{D}_B \colon B\to \mathbb{Z}$  defined 
in~\cite[(3.8)]{HKOTT02} and~\cite[Section 2.14]{OSS03III} by
\begin{equation} \label{eq:DNY}
  \ol{D}_B = \sum_{1 \le i < j \le k} \ol{H}_i R_{i+1} R_{i+2}\dotsm R_{j-1}
   + \sum_{j=1}^k \ol{D}_{B_j} \pi_1 R_1 R_2 \dotsm R_{j-1}.
\end{equation}
Here $R_i$ and $\ol{H}_i$ denote the combinatorial $R$-matrix and local energy
function, respectively, acting on the $i$-th and $(i+1)$-th tensor factors counting from the right.
Also $\pi_1$ is the projection onto the rightmost tensor factor.
The intrinsic energy $\ol{D}_B$ is constant on any $I_0$-component of $B$.
Let $B'$ be the same tensor product of KR crystals as $B$ except that the $i$-th and
$(i+1)$-th positions are interchanged. Let $b'=b'_k\ot\cdots\ot b'_1\in B'$ be
such that $b'_{i+1}\ot b'_i=R(b_{i+1}\ot b_i)$ and $b'_j=b_j$ for $j\ne i,i+1$. Then we have
$\ol{D}_B(b)=\ol{D}_{B'}(b')$. Similarly, let $B'$ be any reordering of the tensor
factors of $B$. The image $b'$ of $b$ under the compositions of the combinatorial
$R$-matrices does not depend on the order of applications due to the 
Yang--Baxter equation, and we also have $\ol{D}_B(b)=\ol{D}_{B'}(b')$,
see~\cite[Prop. 3.9]{HKOTT02} and~\cite[Prop. 2.15]{OSS03III}.
In what follows we often drop the subscript $B$ from $\ol{D}_B(b)$.

In \cite{LOS11,LOS12} we proved the following formula on the intrinsic energy
$\ol{D}$ under the assumption that the KR crystals $B^{r_j,s_j}$ appearing in $B$ 
satisfy $r_j\le n-2$. 
\begin{proposition}[\cite{LOS11,LOS12}] \label{prop:strange}
For an $I_0$-highest weight element $b$ in $B$ we have
\begin{equation} \label{strange}
\ol{D}(b)=\ol{D}(\varsigma(b))+\frac{\lvert B \rvert - \lvert \la(b) \rvert}2,
\end{equation}
where $\lvert B \rvert = \sum_{j=1}^k r_j s_j$, $\la(b)$ is the weight of $b$ and $\lvert \mu \rvert
=\sum_{j=1}^n\mu_j$ if $\mu=\sum_{j=1}^n\mu_j\epsilon_j$.
\end{proposition}
However, if one extends the definition of $|B|$ suitably and examines the proof of 
\cite[Prop. 3.1]{LOS11} carefully, one understands that it is valid even when $B$ 
contains $B^{n-1,s}$ or $B^{n,s}$. In fact, setting $\lvert B \rvert =\sum_{j=1}^k \lvert B_j \rvert$ where
\[
\lvert B^{r,s} \rvert = \lvert s\varpi_r \rvert = \begin{cases}
	rs\quad&\text{if $r\le n-2$,}\\
	(n-2)s/2&\text{if $r=n-1$,}\\
	ns/2&\text{if $r=n$,}
	\end{cases}
\]
one checks that the above proposition remains true.
When the weight $\la(b)$ of an element $b \in B$ satisfies $|\la(b)|=|B|$, we call $b$ 
\defn{maximal}. Note that if $b=b_k\ot\cdots\ot b_1$ is maximal, then $b_j$ is maximal 
for any $j$. If $b\in B^{r,s}$ is a maximal $I_0$-highest weight element, then $b=u_{s\varpi_r}$.
Recall that $\high(b)$ is the $I_0$-highest weight element in the same $I_0$-component as $b$.

\begin{lemma} \label{lem:H value}
We use the notation $b(\alpha)$ in Lemma~\ref{lem:e0 action} indicating the 
dependences of $r,s$ explicitly as $b^{r,s}(\alpha)$. If $r\le n-1$ and $\alpha\le s$, 
then we have
\[
\ol{H}(b^{n-1,s'}(\alpha)\ot u_{s\varpi_r})=\alpha.
\]
\end{lemma}

\begin{proof}
Suppose $2\le r\le n-2$. By weight consideration, we see 
$R(b^{n-1,s'}(\alpha)\ot u_{s\varpi_r})=b^{r,s}(\alpha)\ot u_{s'\varpi_{n-1}}$.
Using Lemma \ref{lem:e0 action} and \eqref{def H}, one finds
\[
\ol{H}(b^{n-1,s'}(\alpha)\ot u_{s\varpi_r})-\alpha=
\ol{H}\bigl(e_0^{\max}(b^{n-1,s'}(\alpha)\ot u_{s\varpi_r})\bigr).
\]
Since $e_0^{\max}(b^{n-1,s'}(\alpha)\ot u_{s\varpi_r})$ belongs to the same 
$I_0$-component as $u_{s'\varpi_{n-1}}\ot u_{s\varpi_r}$, we obtain the desired result.

The cases $r=1$ or $n-1$ are similar.
\end{proof}

\begin{proposition} \label{prop:rs invariance}
Suppose the rightmost factor of $B$ is a KR crystal $B^{r,s}$ such that $s\ge2$.
For an $I_0$-highest weight element $b \in B$, we have
\begin{equation} \label{rs invariance}
\ol{D}(b)=\ol{D}\bigl(\rs(b)\bigr).
\end{equation}
\end{proposition}

\begin{proof}
First we reduce the proposition to the case when $b$ maximal. Substituting $\rs(b)$ 
into $b$ in~\eqref{strange}, using Proposition~\ref{prop:rc and sigma}
and noting that $\rs$ preserves the weight, we obtain $\ol{D}(b)-\ol{D}\bigl(\rs(b)\bigr)
=\ol{D}\bigl(\varsigma(b)\bigr)-\ol{D}\bigl(\rs\bigl(\varsigma(b)\bigr)\bigr)$. Since $\rs$ commutes with
$e_i,f_i$ for $i\in I_0$ and $\ol{D}$ is constant on $I_0$-components, we have
\[
\ol{D}(b)-\ol{D}\bigl(\rs(b)\bigr) =
\ol{D}\bigl((\high\circ\varsigma)(b)\bigr)-\ol{D}\bigl(\rs\bigl((\high\circ\varsigma)(b)\bigr)\bigr).
\]
From Proposition~\ref{prop:strange}, we know that applying $\high\circ\varsigma$
finitely many times sends any $b$ to a maximal element. Hence we have reduced
the proposition to showing~\eqref{rs invariance} for all maximal elements of $B$.

Next we reduce the proposition to showing a certain relation for the energy function
$\ol{H}$. Suppose $b$ is a maximal $I_0$-highest weight element in $B$. 
Then the rightmost factor $b_1$ of $b$ is also an $I_0$-highest weight
element and maximal. Hence $b_1=u_{s\varpi_{r}}$. Denote it simply by $u$.
Let $\rs(u)=u_1\ot u_0\in B^{r,s-1}\ot B^{r,1}$, then $u_1 = u_{(s-1)\varpi_r}$,
$u_0 = u_{\varpi_r}$ as in Section~\ref{subsec:lr-split}.  So we have
\begin{align*}
b&=b_k\ot\cdots\ot b_2\ot u,\\
\rs(b)&=b_k\ot\cdots\ot b_2\ot u_1\ot u_0.
\end{align*}
For consistency, we regard the rightmost factor of $\rs(b)$ as the $0$-th
and the left adjacent one as the first factor. Define $u_1^{(0)}=\pi_0R_0(u_1 \otimes u_0)$ and for $j\ge2$
\[
b_j^{(0)}=\pi_0 R_0 R_1 \dotsm R_{j-1}\bigl(\rs(b)\bigr),\quad
b_j^{(i)}=\pi_i R_i R_{i+1} \dotsm R_{j-1}(b)\text{ for }i=1,2.
\]
For the definition of $R_i$ and $\pi_i$, see the explanations after~\eqref{eq:DNY}.
With this notation, one calculates
\[
\ol{D}(\rs(b))-\ol{D}(b)=\sum_{j=2}^k
\left(\ol{H}(b_j^{(2)}\ot u_1)+\ol{H}(\tilde{b}_j^{(2)}\ot u_0)-\ol{H}(b_j^{(2)}\ot u)\right).
\]
Here $\tilde{b}_j^{(2)}$ is the right component of $R(b_j^{(2)}\ot u_1)$ and
we have used 
\begin{alignat*}{2}
&u_1^{(0)}=u_{(s-1)\varpi_r}&&\\
&b_j^{(0)}=b_j^{(1)}=u_{s_j\varpi_{r_j}} && \text{for $b_j\in B^{r_j,s_j}$},\\
&\ol{H}(u_1\ot u_0)=0=\ol{D}(u_{s'\varpi_{r'}})\qquad && \text{for $u_{s'\varpi_{r'}}\in B^{r',s'}$}.
\end{alignat*}

Finally, suppose $b\ot u_{s\varpi_r}$ is a maximal $I_0$-highest weight element of 
$B^{r',s'}\ot B^{r,s}$ and let $\tilde{b}$ be the right factor of 
$R(b\ot u_{(s-1)\varpi_r}$). To complete the proof it is sufficient to show that
\begin{equation} \label{toshow}
\ol{H}(b\ot u_{(s-1)\varpi_r})+\ol{H}(\tilde{b}\ot u_{\varpi_r})=\ol{H}(b\ot u_{s\varpi_r}).
\end{equation}
We show this by dividing into the following cases:
\begin{alignat*}{3}
&\text{(i)}\; r,r'\le n-2,\quad&
&\text{(ii)}\; r\le n-2,r'=n,\quad&
&\text{(iii)}\; r\le n-2,r'=n-1, \\
&\text{(iv)}\; r=r'=n,&
&\text{(v)}\; r=n,r'=n-1,&
&\text{(vi)}\; r=r'=n-1.
\end{alignat*}

For (i), recall that the value of the energy function for maximal elements
for type $D_n^{(1)}$ is equal to that for type $A_{n-1}^{(1)}$ \cite[Section 9]{LOS12}.
Since \eqref{toshow} is true for type $A_{n-1}^{(1)}$, the proof is done for this case.

For (ii), since $b$ is maximal, it should agree with $u_{s'\varpi_{r'}}$ and also
$\tilde{b}=u_{s'\varpi_{r'}}$. Hence both sides of \eqref{toshow} are 0.
(iv) and (v) are similar.

For (iii), suppose $r\ge3$ first. The possibilities for $b$ are $\hat{b}(\alpha)$
for some $0\le\alpha\le\min(s,s')$, where $\hat{b}(\alpha)$ is the tableau whose left
$(s'-\alpha)$ half columns are $12\dots(n-1)\ol{n}$ and right $\alpha$ half columns
are $1\dots (r-1)(r+1)\dots n\ol{r}$. Suppose $\alpha\le s-1$. If $b=\hat{b}(\alpha)$,
then $\tilde{b}=\hat{b}(0)$. Since $f_{\bf a}(\hat{b}(\alpha)\ot u_{s\varpi_r})
=b^{n-1,s'}(\alpha)\ot u_{s\varpi_r}$, where 
${\bf a}=(2^\alpha 3^\alpha\dots (r-1)^\alpha)$, we have $\alpha+0=\alpha$
for~\eqref{toshow} from Lemma~\ref{lem:H value}. Now suppose $\alpha=s$. 
If $b=\hat{b}(s)$, then $\tilde{b}=\hat{b}(1)$. Noting that $f_{((n-1)\dots(r+1)r)}
(\hat{b}(s)\ot u_{(s-1)\varpi_r})=\hat{b}(s-1)\ot u_{(s-1)\varpi_r}$, we have 
$(\alpha-1)+1=\alpha$ for~\eqref{toshow} this time. 
The cases $r=1,2$ are similar. (vi) is also similar.
\end{proof}

\subsection{Preservation of statistics}

Similar to the intrinsic energy, there is a statistic called \defn{cocharge}
on the set of rigged configurations given by $cc(\nu,J)=cc(\nu)+\sum_{a,i} \lvert J^{(a,i)} \rvert$, where
\begin{equation}\label{eq:cc}
cc(\nu)=\frac{1}{2} \sum_{a,b\in J} \sum_{j,k\ge 1}
 (\alpha_a \mid \alpha_b) \min(j,k) m_j^{(a)}m_k^{(b)}.
\end{equation}
Let $\tilde{\Phi}=\theta\circ\Phi$, where $\theta \colon \rc(L,\la)\to\rc(L,\la)$
with $\theta(\nu,J)=(\nu,\tilde{J})$ being the function that complements the
riggings, meaning that $\tilde{J}$ is obtained from $J$ by
complementing all partitions $J^{(a,i)}$ in the $m_i^{(a)}\times p_i^{(a)}$ 
rectangle.

\begin{theorem}\label{thm:stat}
Let $B=B^{r_k,s_k}\otimes \cdots\otimes B^{r_1,s_1}$ be a tensor product of
KR crystals and $\la$ a dominant integral weight.
The bijection $\tilde{\Phi} \colon \p(B,\la)\to\rc(L(B),\la)$ preserves the statistics,
that is $\ol{D}(b)=cc(\tilde{\Phi}(b))$ for all $b\in \p(\la,B)$.
\end{theorem}

\begin{proof}
We define a map 
\[
\mathrm{col} \colon B=B^{r_k,s_k}\otimes \cdots\otimes B^{r_1,s_1}\longrightarrow
(B^{r_k,1})^{\ot s_k}\ot\cdots\ot(B^{r_1,1})^{\ot s_1}
\]
as follows. If $s_1>1$, apply $\rs$ to $B$ first. Exchange $B^{r_1,s_1-1}$ and 
$B^{r_1,1}$ by using the combinatorial $R$-matrix so that we have 
$B^{r_1,1}\ot B^{r_1,s_1-1}$ at the rightmost two factors. 
If $s_1-1>1$, apply again $\rs$ and $R$ so that we have 
$(B^{r_1,1})^{\ot 2}\ot B^{r_1,s_1-2}$ at the rightmost three factors. 
Continuing this procedure, we have $(B^{r_1,1})^{\ot s_1}$ at the rightmost $s_1$
factors. Next applying a sequence of combinatorial $R$-matrices, bring
$B^{r_2,s_2}$ to the rightmost factor. We then apply a similar procedure
to the rightmost $B^{r_2,s_2}$ and bring $(B^{r_2,1})^{\ot s_2}$ back to the original
position. At this point we have $(B^{r_2,1})^{\ot s_2}\ot(B^{r_1,1})^{\ot s_1}$ at the 
rightmost $(s_1+s_2)$ factors. Continuing these, we arrive at the image of
$\mathrm{col}$. 

Since $\rs$ and $R$ commute with $e_i,f_i$ ($i\in I_0$),
$\mathrm{col}$ restricts to the map from $\p(B,\la)$ to $\p(\mathrm{col}(B),\la)$.
By Proposition \ref{prop:rs invariance} and the invariance of $\ol{D}$ by the
combinatorial $R$-matrix, we have $\ol{D}(b)=\ol{D}(\mathrm{col}(b))$ for
$b\in\p(B,\la)$.

Finally, note that for $(\nu,J)\in\rc(L(B),\la)$ we have 
$cc(\nu,J)=cc(\gamma(\nu,J))$. In view of 
$\tilde{\Phi}\circ\rs=\gamma\circ\tilde{\Phi}$ and Theorem \ref{th:main},
the proof is reduced to the case when all the components of $B$ are of column
type $B^{r_j,1}$. The corresponding statement is proven in~\cite[Theorem 5.1]{S:2005}.
\end{proof}

An immediate corollary of Theorem~\ref{th:welldef_highest} and Theorem~\ref{thm:stat} is
the equality
\begin{equation*}
\sum_{b\in\p(B,\la)} q^{\ol{D}(b)} 
 = \sum_{(\nu,J)\in \rc(L(B),\la)} q^{cc(\nu,J)}.
\end{equation*}
The left-hand side is in fact the one-dimensional configuration sum
\begin{equation*}
X(B,\la)=\sum_{b\in\p(B,\la)} q^{\ol{D}(b)}.
\end{equation*}
The right-hand side can be simplified slightly by observing that the 
generating function of partitions in a box of width $p$ and height $m$
is the $q$-binomial coefficient
\begin{equation*}
\genfrac{[}{]}{0pt}{}{m+p}{m}=\frac{(q)_{p+m}}{(q)_m(q)_p},
\end{equation*}
where $(q)_m=(1-q)(1-q^2)\cdots (1-q^m)$. Hence the right-hand side becomes
the \defn{fermionic formula}
\begin{equation*}
M(B,\la)=\sum_{\nu\in C(L(B),\la)} q^{cc(\nu)} \prod_{\substack{i\ge 1\\ a\in I_0}} 
\genfrac{[}{]}{0pt}{}{m_i^{(a)}+P_i^{(a)}}{m_i^{(a)}},
\end{equation*}
where $m_i^{(a)}$ and $P_i^{(a)}$ are as defined in Section~\ref{sec:RC}.
This proves the following result conjectured in~\cite{HKOTY99,HKOTT02}.

\begin{corollary} 
\label{corollary.X=M}
Let $B=B^{r_k,s_k}\otimes \cdots\otimes B^{r_1,s_1}$ be 
a tensor product of KR crystals of type $D_n^{(1)}$ and $\la$ a dominant integral weight. Then we have
\begin{equation*}
X(B,\la)=M(B,\la).
\end{equation*}
\end{corollary}

\appendix
\section{Example of the rigged configuration bijection}
\label{appendix:def_RCbijection}

In this section, we provide an example of the algorithm $\Phi^{-1}$. The reader may easily infer from the following example the meaning of the correspondence between the operators summarized in Table~\ref{table:operations}.

\begin{example}
Consider the unrestricted rigged configuration $r_1 \in \RC(L(B))$ for $B=B^{3,2}\otimes B^{3,3}\otimes B^{2,3}$
of type $D^{(1)}_5$:
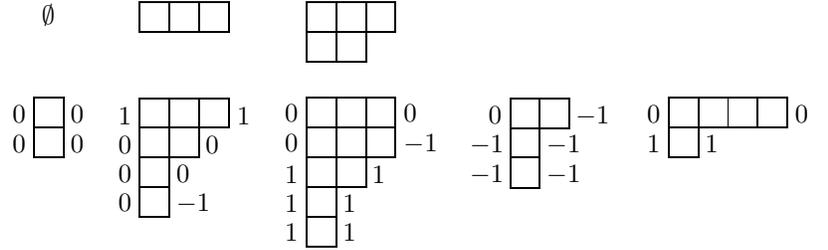
\begin{center}
\unitlength 10pt
\begin{picture}(30,10)
\put(0,3.25){
\put(-0.8,0.2){0}
\put(-0.8,1.3){0}
\put(0,0){\young(\ku,\ku)}
\put(1.4,0.2){0}
\put(1.4,1.3){0}
}
\put(4,1){
\put(-0.8,0.2){0}
\put(-0.8,1.3){0}
\put(-0.8,2.4){0}
\put(-0.8,3.5){1}
\put(0,0){\young(\ku\ku\ku,\ku\ku,\ku,\ku)}
\put(1.4,0.2){$-1$}
\put(1.4,1.3){0}
\put(2.5,2.4){0}
\put(3.7,3.5){1}
}
\put(10.3,-0.12){
\put(-0.8,0.2){1}
\put(-0.8,1.3){1}
\put(-0.8,2.4){1}
\put(-0.8,3.6){0}
\put(-0.8,4.7){0}
\put(0,0){$\young(\ku\ku\ku,\ku\ku\ku,\ku\ku,\ku,\ku)$}
\put(1.4,0.2){1}
\put(1.4,1.3){1}
\put(2.5,2.4){1}
\put(3.7,3.6){$-1$}
\put(3.7,4.7){0}
}
\put(18,2.11){
\put(-1.5,0.2){$-1$}
\put(-1.5,1.3){$-1$}
\put(-0.8,2.4){0}
\put(0,0){\young(\ku\ku,\ku,\ku)}
\put(1.4,0.2){$-1$}
\put(1.4,1.3){$-1$}
\put(2.5,2.4){$-1$}
}
\put(24,3.25){
\put(-0.8,0.2){1}
\put(-0.8,1.3){0}
\put(0,0){\young(\ku\ku\ku\ku,\ku)}
\put(1.4,0.2){1}
\put(4.8,1.3){0}
}
\put(0.3,8.3){
\put(0,0){$\emptyset$}
}
\put(4,8){
\put(0,0){\young(\ku\ku\ku)}
}
\put(10.3,6.87){
\put(0,0){\young(\ku\ku\ku,\ku\ku)}
}
\end{picture}\raisebox{30pt}{.}
\end{center}
In the above diagram, we show the partition $\mu^{(a)}$ as defined in Section~\ref{section.definition rc}
over the corresponding rigged partition $(\nu^{(a)},J^{(a)})$ in order to make it easier to see the operations 
$\gamma$ and $\beta$. In $\Phi^{-1}$, let us remove the $B^{3,2}$ part first.
We begin by applying $\gamma$ and obtain $r_2:=\gamma(r_1)$ which looks as follows:
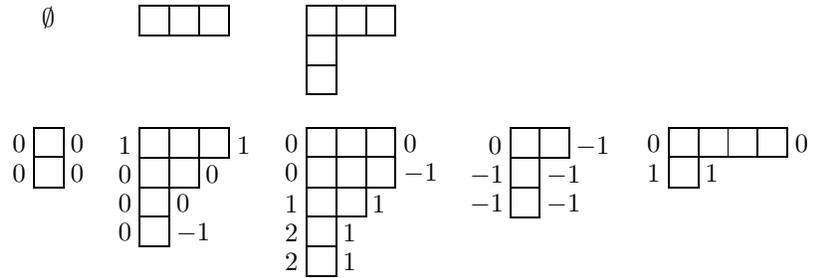
\begin{center}
\unitlength 10pt
\begin{picture}(30,11)
\put(0,3.25){
\put(-0.8,0.2){0}
\put(-0.8,1.3){0}
\put(0,0){\young(\ku,\ku)}
\put(1.4,0.2){0}
\put(1.4,1.3){0}
}
\put(4,1){
\put(-0.8,0.2){0}
\put(-0.8,1.3){0}
\put(-0.8,2.4){0}
\put(-0.8,3.5){1}
\put(0,0){\young(\ku\ku\ku,\ku\ku,\ku,\ku)}
\put(1.4,0.2){$-1$}
\put(1.4,1.3){0}
\put(2.5,2.4){0}
\put(3.7,3.5){1}
}
\put(10.3,-0.12){
\put(-0.8,0.2){2}
\put(-0.8,1.3){2}
\put(-0.8,2.4){1}
\put(-0.8,3.6){0}
\put(-0.8,4.7){0}
\put(0,0){$\young(\ku\ku\ku,\ku\ku\ku,\ku\ku,\ku,\ku)$}
\put(1.4,0.2){1}
\put(1.4,1.3){1}
\put(2.5,2.4){1}
\put(3.7,3.6){$-1$}
\put(3.7,4.7){0}
}
\put(18,2.11){
\put(-1.5,0.2){$-1$}
\put(-1.5,1.3){$-1$}
\put(-0.8,2.4){0}
\put(0,0){\young(\ku\ku,\ku,\ku)}
\put(1.4,0.2){$-1$}
\put(1.4,1.3){$-1$}
\put(2.5,2.4){$-1$}
}
\put(24,3.25){
\put(-0.8,0.2){1}
\put(-0.8,1.3){0}
\put(0,0){\young(\ku\ku\ku\ku,\ku)}
\put(1.4,0.2){1}
\put(4.8,1.3){0}
}
\put(0.3,9.3){
\put(0,0){$\emptyset$}
}
\put(4,9){
\put(0,0){\young(\ku\ku\ku)}
}
\put(10.3,6.75){
\put(0,0){$\young(\ku\ku\ku,\ku,\ku)$}
}
\end{picture}\raisebox{30pt}{.}
\end{center}
The changes are the shape of $\mu^{(3)}$ and the resulting change of
the vacancy numbers for $P^{(3)}_1(\nu)$ which makes the length 1 strings of $(\nu,J)^{(3)}$
non-singular.\footnote{
In general, if we consider $\gamma$ for $B^{r,s}$, the strings in $(\nu,J)^{(r)}$
which are shorter than $s$ become non-singular.}
This operation corresponds to
\begin{align}\label{eq:example_Phi_-1_0}
\operatorname{ls}\colon B^{3,2}\longrightarrow B^{3,1}\otimes B^{3,1}.
\end{align}
Then $r_3:=\beta(r_2)$ looks as follows:
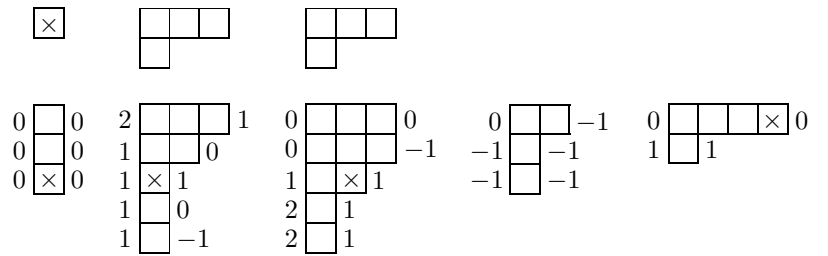
\begin{center}
\unitlength 10pt
\begin{picture}(30,10)
\put(0,2.11){
\put(-0.8,0.2){0}
\put(-0.8,1.3){0}
\put(-0.8,2.4){0}
\put(0,0){$\young(\ku,\ku,\times)$}
\put(1.4,0.2){0}
\put(1.4,1.3){0}
\put(1.4,2.4){0}
}
\put(4,-0.12){
\put(-0.8,0.2){1}
\put(-0.8,1.3){1}
\put(-0.8,2.4){1}
\put(-0.8,3.5){1}
\put(-0.8,4.7){2}
\put(0,0){$\young(\ku\ku\ku,\ku\ku,\times,\ku,\ku)$}
\put(1.4,0.2){$-1$}
\put(1.4,1.3){0}
\put(1.4,2.4){1}
\put(2.5,3.5){0}
\put(3.7,4.7){1}
}
\put(10.3,-0.12){
\put(-0.8,0.2){2}
\put(-0.8,1.3){2}
\put(-0.8,2.4){1}
\put(-0.8,3.6){0}
\put(-0.8,4.7){0}
\put(0,0){$\young(\ku\ku\ku,\ku\ku\ku,\ku\times,\ku,\ku)$}
\put(1.4,0.2){1}
\put(1.4,1.3){1}
\put(2.5,2.4){1}
\put(3.7,3.6){$-1$}
\put(3.7,4.7){0}
}
\put(18,2.11){
\put(-1.5,0.2){$-1$}
\put(-1.5,1.3){$-1$}
\put(-0.8,2.4){0}
\put(0,0){\young(\ku\ku,\ku,\ku)}
\put(1.4,0.2){$-1$}
\put(1.4,1.3){$-1$}
\put(2.5,2.4){$-1$}
}
\put(24,3.25){
\put(-0.8,0.2){1}
\put(-0.8,1.3){0}
\put(0,0){$\young(\ku\ku\ku\times,\ku)$}
\put(1.4,0.2){1}
\put(4.8,1.3){0}
}
\put(0,8){
\put(0,0){$\young(\times)$}
}
\put(4,6.87){
\put(0,0){$\young(\ku\ku\ku,\ku)$}
}
\put(10.3,6.87){
\put(0,0){\young(\ku\ku\ku,\ku)}
}
\end{picture}\raisebox{30pt}{.}
\end{center}
This corresponds to
\begin{align}\label{eq:example_Phi_-1}
\operatorname{lb} \colon B^{3,1}\longrightarrow B^{1,1}\otimes B^{2,1}.
\end{align}
Note that the vacancy numbers for $\nu^{(3)}$ do not change.
Since $\beta$ adds length 1 singular strings to $(\nu,J)^{(1)}$ and $(\nu,J)^{(2)}$,
applying $\delta$ removes the boxes with ``$\times$" in
the above diagram.\footnote{In general, if we consider $\beta$ for $B^{r,1}$,
the next $\delta$ removes length 1 singular strings of $(\nu,J)^{(1)},(\nu,J)^{(2)},\ldots,(\nu,J)^{(r-1)}$
added by $\beta$.}
Then $\delta$ gives the following rigged configuration $r_4:=\delta(r_3)$
together with the output letter $\overline{5}$ which fills the bottom left corner
of $B^{3,2}$ as $\,\Yvcentermath1\young(\ku\ku,\ku\ku,\mfive\ku)\,$.
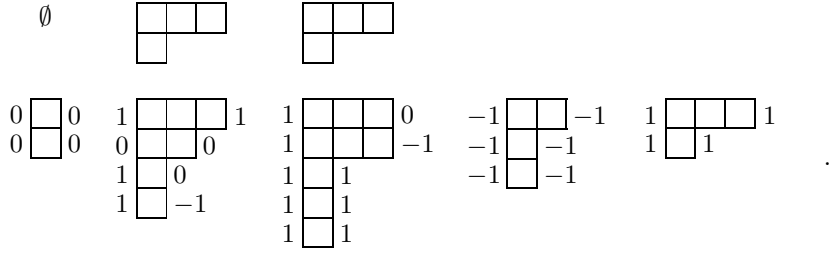
\begin{center}
\unitlength 10pt
\begin{picture}(30,10)
\put(0,3.25){
\put(-0.8,0.2){0}
\put(-0.8,1.3){0}
\put(0,0){\young(\ku,\ku)}
\put(1.4,0.2){0}
\put(1.4,1.3){0}
}
\put(4,1){
\put(-0.8,0.2){1}
\put(-0.8,1.3){1}
\put(-0.8,2.4){0}
\put(-0.8,3.5){1}
\put(0,0){\young(\ku\ku\ku,\ku\ku,\ku,\ku)}
\put(1.4,0.2){$-1$}
\put(1.4,1.3){0}
\put(2.5,2.4){0}
\put(3.7,3.5){1}
}
\put(10.3,-0.12){
\put(-0.8,0.2){1}
\put(-0.8,1.3){1}
\put(-0.8,2.4){1}
\put(-0.8,3.6){1}
\put(-0.8,4.7){1}
\put(0,0){$\young(\ku\ku\ku,\ku\ku\ku,\ku,\ku,\ku)$}
\put(1.4,0.2){1}
\put(1.4,1.3){1}
\put(1.4,2.4){1}
\put(3.7,3.6){$-1$}
\put(3.7,4.7){0}
}
\put(18,2.11){
\put(-1.5,0.2){$-1$}
\put(-1.5,1.3){$-1$}
\put(-1.5,2.4){$-1$}
\put(0,0){\young(\ku\ku,\ku,\ku)}
\put(1.4,0.2){$-1$}
\put(1.4,1.3){$-1$}
\put(2.5,2.4){$-1$}
}
\put(24,3.25){
\put(-0.8,0.2){1}
\put(-0.8,1.3){1}
\put(0,0){$\young(\ku\ku\ku,\ku)$}
\put(1.4,0.2){1}
\put(3.7,1.3){1}
}
\put(0.3,8.3){
\put(0,0){$\emptyset$}
}
\put(4,6.87){
\put(0,0){$\young(\ku\ku\ku,\ku)$}
}
\put(10.3,6.87){
\put(0,0){\young(\ku\ku\ku,\ku)}
}
\end{picture}\raisebox{30pt}{.}
\end{center}
Since the above $\delta$ determines $B^{1,1}$ of~\eqref{eq:example_Phi_-1},
we start to apply $\gamma$, $\beta$, and $\delta$ corresponding to $B^{2,1}$ of~\eqref{eq:example_Phi_-1}.
Since $\gamma(r_4) = r_4$, the unrestricted rigged configuration $r_5 := \beta(r_4)$ looks as follows:
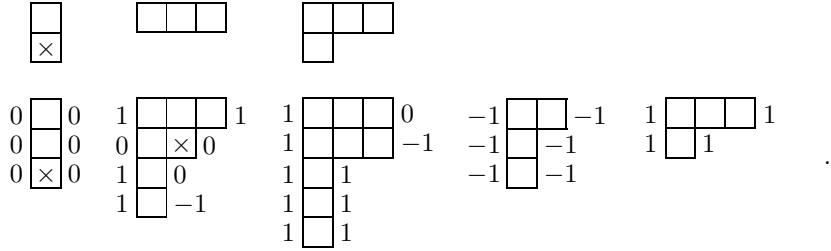
\begin{center}
\unitlength 10pt
\begin{picture}(30,10)
\put(0,2.11){
\put(-0.8,0.2){0}
\put(-0.8,1.3){0}
\put(-0.8,2.4){0}
\put(0,0){$\young(\ku,\ku,\times)$}
\put(1.4,0.2){0}
\put(1.4,1.3){0}
\put(1.4,2.4){0}
}
\put(4,1){
\put(-0.8,0.2){1}
\put(-0.8,1.3){1}
\put(-0.8,2.4){0}
\put(-0.8,3.5){1}
\put(0,0){$\young(\ku\ku\ku,\ku\times,\ku,\ku)$}
\put(1.4,0.2){$-1$}
\put(1.4,1.3){0}
\put(2.5,2.4){0}
\put(3.7,3.5){1}
}
\put(10.3,-0.12){
\put(-0.8,0.2){1}
\put(-0.8,1.3){1}
\put(-0.8,2.4){1}
\put(-0.8,3.6){1}
\put(-0.8,4.7){1}
\put(0,0){$\young(\ku\ku\ku,\ku\ku\ku,\ku,\ku,\ku)$}
\put(1.4,0.2){1}
\put(1.4,1.3){1}
\put(1.4,2.4){1}
\put(3.7,3.6){$-1$}
\put(3.7,4.7){0}
}
\put(18,2.11){
\put(-1.5,0.2){$-1$}
\put(-1.5,1.3){$-1$}
\put(-1.5,2.4){$-1$}
\put(0,0){\young(\ku\ku,\ku,\ku)}
\put(1.4,0.2){$-1$}
\put(1.4,1.3){$-1$}
\put(2.5,2.4){$-1$}
}
\put(24,3.25){
\put(-0.8,0.2){1}
\put(-0.8,1.3){1}
\put(0,0){$\young(\ku\ku\ku,\ku)$}
\put(1.4,0.2){1}
\put(3.7,1.3){1}
}
\put(0,6.87){
\put(0,0){$\young(\ku,\times)$}
}
\put(4,8){
\put(0,0){$\young(\ku\ku\ku)$}
}
\put(10.3,6.87){
\put(0,0){\young(\ku\ku\ku,\ku)}
}
\end{picture}\raisebox{30pt}{.}
\end{center}
This corresponds to
\begin{align}\label{eq:example_Phi_-1_2}
\operatorname{lb}:B^{2,1}\longrightarrow B^{1,1}\otimes B^{1,1}.
\end{align}
Then $\delta$ removes the boxes with ``$\times$" in the above diagram
which determines one of $B^{1,1}$ in~\eqref{eq:example_Phi_-1_2}.
As the result, we obtain the following unrestricted rigged configuration $r_6:=\delta(r_5)$
and the output letter within $B^{3,2}$ as $\,\Yvcentermath1\young(\ku\ku,3\ku,\mfive\ku)\,$:
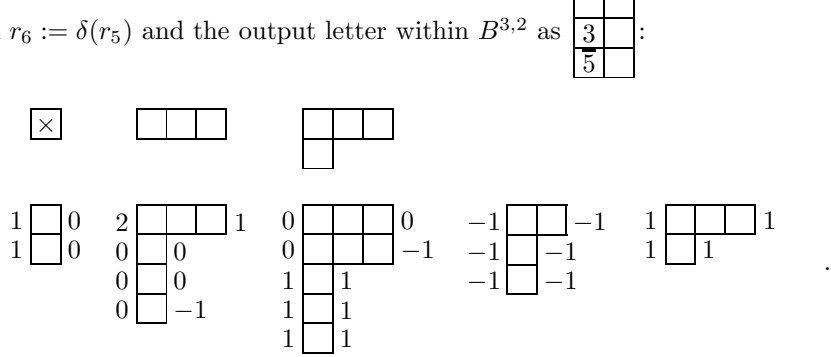
\begin{center}
\unitlength 10pt
\begin{picture}(30,10)
\put(0,3.25){
\put(-0.8,0.2){1}
\put(-0.8,1.3){1}
\put(0,0){$\young(\ku,\ku)$}
\put(1.4,0.2){0}
\put(1.4,1.3){0}
}
\put(4,1){
\put(-0.8,0.2){0}
\put(-0.8,1.3){0}
\put(-0.8,2.4){0}
\put(-0.8,3.5){2}
\put(0,0){$\young(\ku\ku\ku,\ku,\ku,\ku)$}
\put(1.4,0.2){$-1$}
\put(1.4,1.3){0}
\put(1.4,2.4){0}
\put(3.7,3.5){1}
}
\put(10.3,-0.12){
\put(-0.8,0.2){1}
\put(-0.8,1.3){1}
\put(-0.8,2.4){1}
\put(-0.8,3.6){0}
\put(-0.8,4.7){0}
\put(0,0){$\young(\ku\ku\ku,\ku\ku\ku,\ku,\ku,\ku)$}
\put(1.4,0.2){1}
\put(1.4,1.3){1}
\put(1.4,2.4){1}
\put(3.7,3.6){$-1$}
\put(3.7,4.7){0}
}
\put(18,2.11){
\put(-1.5,0.2){$-1$}
\put(-1.5,1.3){$-1$}
\put(-1.5,2.4){$-1$}
\put(0,0){\young(\ku\ku,\ku,\ku)}
\put(1.4,0.2){$-1$}
\put(1.4,1.3){$-1$}
\put(2.5,2.4){$-1$}
}
\put(24,3.25){
\put(-0.8,0.2){1}
\put(-0.8,1.3){1}
\put(0,0){$\young(\ku\ku\ku,\ku)$}
\put(1.4,0.2){1}
\put(3.7,1.3){1}
}
\put(0,8){
\put(0,0){$\young(\times)$}
}
\put(4,8){
\put(0,0){$\young(\ku\ku\ku)$}
}
\put(10.3,6.87){
\put(0,0){\young(\ku\ku\ku,\ku)}
}
\end{picture}\raisebox{30pt}{.}
\end{center}
The next $\delta$ removes the box with ``$\times$" in the above diagram
and determines the remaining $B^{1,1}$ in~\eqref{eq:example_Phi_-1_2}.
As the result, we obtain the following unrestricted rigged configuration $r_7 := \delta(r_6)$
and the output letter within $B^{3,2}$ as $\,\Yvcentermath1\young(1\ku,3\ku,\mfive\ku)\,$:
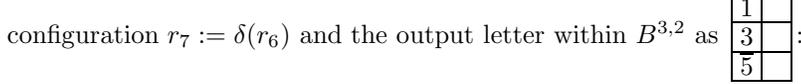
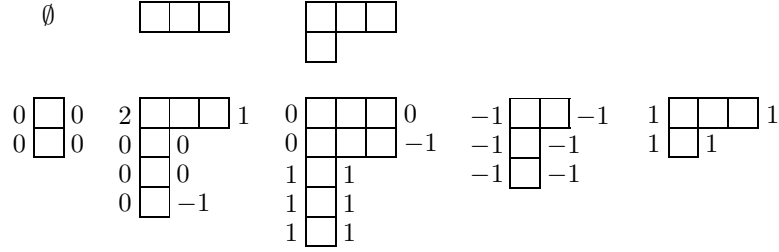
\begin{center}
\unitlength 10pt
\begin{picture}(30,10)
\put(0,3.25){
\put(-0.8,0.2){0}
\put(-0.8,1.3){0}
\put(0,0){$\young(\ku,\ku)$}
\put(1.4,0.2){0}
\put(1.4,1.3){0}
}
\put(4,1){
\put(-0.8,0.2){0}
\put(-0.8,1.3){0}
\put(-0.8,2.4){0}
\put(-0.8,3.5){2}
\put(0,0){$\young(\ku\ku\ku,\ku,\ku,\ku)$}
\put(1.4,0.2){$-1$}
\put(1.4,1.3){0}
\put(1.4,2.4){0}
\put(3.7,3.5){1}
}
\put(10.3,-0.12){
\put(-0.8,0.2){1}
\put(-0.8,1.3){1}
\put(-0.8,2.4){1}
\put(-0.8,3.6){0}
\put(-0.8,4.7){0}
\put(0,0){$\young(\ku\ku\ku,\ku\ku\ku,\ku,\ku,\ku)$}
\put(1.4,0.2){1}
\put(1.4,1.3){1}
\put(1.4,2.4){1}
\put(3.7,3.6){$-1$}
\put(3.7,4.7){0}
}
\put(18,2.11){
\put(-1.5,0.2){$-1$}
\put(-1.5,1.3){$-1$}
\put(-1.5,2.4){$-1$}
\put(0,0){\young(\ku\ku,\ku,\ku)}
\put(1.4,0.2){$-1$}
\put(1.4,1.3){$-1$}
\put(2.5,2.4){$-1$}
}
\put(24,3.25){
\put(-0.8,0.2){1}
\put(-0.8,1.3){1}
\put(0,0){$\young(\ku\ku\ku,\ku)$}
\put(1.4,0.2){1}
\put(3.7,1.3){1}
}
\put(0.3,8.3){
\put(0,0){$\emptyset$}
}
\put(4,8){
\put(0,0){$\young(\ku\ku\ku)$}
}
\put(10.3,6.87){
\put(0,0){\young(\ku\ku\ku,\ku)}
}
\end{picture}\raisebox{30pt}{.}
\end{center}
Next we determine the remaining $B^{3,1}$ in~\eqref{eq:example_Phi_-1_0}.
$r_8 = (\delta \circ \beta)(r_7)$ is the following unrestricted rigged configuration
with the output letter in $B^{2,3}$ as $\,\Yvcentermath1\young(1\ku,3\ku,\mfive\mone)\,$:
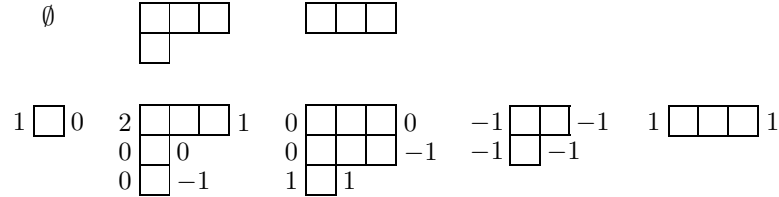
\begin{center}
\unitlength 10pt
\begin{picture}(30,8)
\put(0,2.11){
\put(-0.8,0.2){1}
\put(0,0){$\young(\ku)$}
\put(1.4,0.2){0}
}
\put(4,-0.12){
\put(-0.8,0.2){0}
\put(-0.8,1.3){0}
\put(-0.8,2.4){2}
\put(0,0){$\young(\ku\ku\ku,\ku,\ku)$}
\put(1.4,0.2){$-1$}
\put(1.4,1.3){0}
\put(3.7,2.4){1}
}
\put(10.3,-0.12){
\put(-0.8,0.2){1}
\put(-0.8,1.3){0}
\put(-0.8,2.4){0}
\put(0,0){$\young(\ku\ku\ku,\ku\ku\ku,\ku)$}
\put(1.4,0.2){1}
\put(3.7,1.3){$-1$}
\put(3.7,2.4){0}
}
\put(18,1){
\put(-1.5,0.2){$-1$}
\put(-1.5,1.3){$-1$}
\put(0,0){\young(\ku\ku,\ku)}
\put(1.4,0.2){$-1$}
\put(2.5,1.3){$-1$}
}
\put(24,2.11){
\put(-0.8,0.2){1}
\put(0,0){$\young(\ku\ku\ku)$}
\put(3.7,0.2){1}
}
\put(0.3,6.3){
\put(0,0){$\emptyset$}
}
\put(4,4.87){
\put(0,0){$\young(\ku\ku\ku,\ku)$}
}
\put(10.3,6){
\put(0,0){\young(\ku\ku\ku)}
}
\end{picture}\raisebox{30pt}{.}
\end{center}
$r_9 = (\delta \circ \beta)(r_8)$ is the following unrestricted rigged configuration
with the output letter in $B^{2,3}$ as $\,\Yvcentermath1\young(1\ku,3\mthree,\mfive\mone)\,$:
\begin{center}
\unitlength 10pt
\begin{picture}(30,6)
\put(0,1){
\put(-0.8,0.2){1}
\put(0,0){$\young(\ku)$}
\put(1.4,0.2){0}
}
\put(4,-0.12){
\put(-0.8,0.2){0}
\put(-0.8,1.3){1}
\put(0,0){$\young(\ku\ku\ku,\ku)$}
\put(1.4,0.2){$-1$}
\put(3.7,1.3){1}
}
\put(10.3,-0.12){
\put(-0.8,0.2){1}
\put(-0.8,1.3){1}
\put(0,0){$\young(\ku\ku\ku,\ku\ku)$}
\put(2.5,0.2){1}
\put(3.7,1.3){$-1$}
}
\put(18,1){
\put(-0.8,0.2){0}
\put(0,0){\young(\ku\ku)}
\put(2.5,0.2){$-1$}
}
\put(24,1){
\put(-0.8,0.2){0}
\put(0,0){$\young(\ku\ku)$}
\put(2.5,0.2){0}
}
\put(0,4){
\put(0,0){$\young(\ku)$}
}
\put(4,4){
\put(0,0){$\young(\ku\ku\ku)$}
}
\put(10.3,4){
\put(0,0){\young(\ku\ku\ku)}
}
\end{picture}\raisebox{30pt}{.}
\end{center}
$r_{10} = \delta(r_9)$ is the following unrestricted rigged configuration
with the output letter in $B^{2,3}$ as $\,\Yvcentermath1\young(11,3\mthree,\mfive\mone)\,$:
\begin{center}
\unitlength 10pt
\begin{picture}(30,6)
\put(0,1){
\put(-0.8,0.2){0}
\put(0,0){$\young(\ku)$}
\put(1.4,0.2){0}
}
\put(4,-0.12){
\put(-0.8,0.2){0}
\put(-0.8,1.3){1}
\put(0,0){$\young(\ku\ku\ku,\ku)$}
\put(1.4,0.2){$-1$}
\put(3.7,1.3){1}
}
\put(10.3,-0.12){
\put(-0.8,0.2){1}
\put(-0.8,1.3){1}
\put(0,0){$\young(\ku\ku\ku,\ku\ku)$}
\put(2.5,0.2){1}
\put(3.7,1.3){$-1$}
}
\put(18,1){
\put(-0.8,0.2){0}
\put(0,0){\young(\ku\ku)}
\put(2.5,0.2){$-1$}
}
\put(24,1){
\put(-0.8,0.2){0}
\put(0,0){$\young(\ku\ku)$}
\put(2.5,0.2){0}
}
\put(0.3,4.3){
\put(0,0){$\emptyset$}
}
\put(4,4){
\put(0,0){$\young(\ku\ku\ku)$}
}
\put(10.3,4){
\put(0,0){\young(\ku\ku\ku)}
}
\end{picture}\raisebox{30pt}{.}
\end{center}
For the KN tableau representation of the rectangular tableau, we need to apply the inverse of the
filling map of Definition~\ref{def:left split} to obtain $\,\Yvcentermath1\young(1\mthree,3,\mfive)$.
If we further determine $B^{3,3}$ and $B^{2,3}$ in this order, we obtain the empty rigged configuration
and the following path
\[
\Phi^{-1}(r_1)=
\Yvcentermath1
\young(1\mthree,3,\mfive)\otimes
\young(111,234,3\mfive\mthree)\otimes
\young(113,225) \; .
\]
\end{example}

\noindent
{\bf Summary.}
As we see in the above example, the algorithm $\Phi^{-1}$ is recursively defined as follows.
Suppose that we consider the unrestricted rigged configuration $(\nu,J)$ associated with the
tensor product of type
\[
B=B^{r,s}\otimes B'.
\]
Then we determine $b\in B^{r,s}$ by the following procedure.
In correspondence with the operation
\begin{align}\label{eq:def_Phi-1_1}
\ls \colon B^{r,s} \longrightarrow B^{r,1}\otimes B^{r,s-1},
\end{align}
we apply $\gamma$ to the unrestricted rigged configuration.
Then apply $\operatorname{lb}$ on $B^{r,1}$ in~\eqref{eq:def_Phi-1_1}
\begin{align}\label{eq:def_Phi-1_2}
\lb \colon B^{r,1}\longrightarrow B^{1,1}\otimes B^{r-1,1},
\end{align}
which on the rigged configuration side corresponds to $\beta$.
Finally we use $\lh$ and $\delta$ to remove $B^{1,1}$ of~\eqref{eq:def_Phi-1_2}.
In order to process the remaining $B^{r-1,1}$ of~\eqref{eq:def_Phi-1_2},
we apply $\beta$ and $\delta$ repeatedly for $(r-1)$-times.
We fill the leftmost column of $B^{r,s}$ from bottom to top
by the letters obtained by $\delta$ during this procedure.
In order to determine the remaining $B^{r,s-1}$ of~\eqref{eq:def_Phi-1_1},
we repeat the same procedure for the first column and fill
the remaining columns of $B^{r,s}$ from left to right.
Once $B^{r,s}$ is fully determined, we proceed to the leftmost rectangle
of $B'$ and repeat the same procedure used for $B^{r,s}$.
In this manner, we obtain the filling of shape $B$, which we denote by $b\in B$.
Then we define
\[
\Phi^{-1}:(\nu,J)\longmapsto b.
\]
The inverse procedure $\Phi$ is obtained by reversing all the steps of $\Phi^{-1}$.


\bibliographystyle{alpha}
\bibliography{bijection}{}
 
\end{document}